\documentclass[french,11pt,a4paper]{amsart}

\usepackage[a4paper]{geometry}

\usepackage[utf8]{inputenc}
\usepackage[T1]{fontenc}
\usepackage{lmodern}
\usepackage[francais]{babel}
\usepackage{amsmath}
\usepackage{amsfonts,amssymb,amsthm}

\usepackage{mathtools}
\usepackage[all]{xy}
\usepackage{hyperref}
\usepackage{verbatimbox}
\usepackage{tikz-cd}
\usepackage{geometry}
\usepackage{blkarray}
\usepackage{comment}
\usepackage{enumerate}
\usepackage{float}
\usepackage{mathrsfs}
\usepackage{relsize}
\usepackage{tabularx}

\makeatletter
\def\@settitle{\begin{center}%
  \baselineskip14\p@\relax
  \bfseries
  \uppercasenonmath\@title
  \@title
  \ifx\@subtitle\@empty\else
     \\[1ex]\uppercasenonmath\@subtitle
     \footnotesize\mdseries\@subtitle
  \fi
  \end{center}%
}
\def\subtitle#1{\gdef\@subtitle{#1}}
\def\@subtitle{}
\makeatother

\newtheorem{proposition}{Proposition}
\newtheorem{lemma}{Lemme} 
\newtheorem{thm}{Théorème}
\newtheorem{cor}{Corollaire}
\newtheorem*{thm*}{Théorème}

\theoremstyle{definition}
\newtheorem{defi}{Définition}
\newtheorem{rmk}{Remarque}
\newtheorem{convention}{Convention}

\DeclareMathOperator{\ch}{ch}

\DeclareMathOperator{\ind}{Ind}
\DeclareMathOperator{\res}{res}
\DeclareMathOperator{\diag}{diag}

\DeclareMathOperator{\Gr}{Gr}
\DeclareMathOperator{\KF}{FC}
\DeclareMathOperator{\SL}{SL}
\DeclareMathOperator{\image}{Im}
\DeclareMathOperator{\pr}{pr}
\DeclareMathOperator{\Ext}{Ext}
\DeclareMathOperator{\Hom}{Hom}

\title{Cohomologie des fibrés en droites sur \( \SL_{3}/B \) en
  caractéristique positive}
\subtitle{deux filtrations et conséquences}

\author{Linyuan Liu}

\address{Institut de Mathématiques de Jussieu-Paris Rive Gauche\\
Sorbonne Université -- Campus Pierre et Marie Curie\\
4, place Jussieu -- Boîte Courrier 247\\
F-75252 Paris Cedex 05\\
France}

\email{linyuan.liu@imj-prg.fr}

\date{} 
\begin{document}
\maketitle
\section{Introduction}
\subsection{Histoire et motivations du problème}
Soient \( G \) un schéma en groupes semi-simple déployé sur un corps \( k \)
 de caractéristique positive, \( B\) un
sous-groupe de Borel et \( T\subset B \) un tore maximal déployé. Soit \( X(T)
\) le groupe des caractères de \( T \). Pour tout \( \mu\in X(T) \),
considéré comme caractère de \( B \), on
note \( \mathcal{L}(\mu) \) le fibré en droites \( G \)-équivariant
induit par \( \mu \) et l'on pose \(
H^{i}(\mu):=H^{i}(G/B,\mathcal{L}(\mu)) \).

Non seulement ces groupes de cohomologie sont des objets intéressants
et fondamentaux 
dans la géométrie algébrique, mais ils sont également munis  d'une
structure de \( G \)-modules, ce qui en fait une classe d'objets
importante dans la théorie des représentations de \( G \). Par
exemple, les \( G \)-modules simples sont paramétrisés par les poids
dominants, et pour tout \( \lambda \) dominant, le \( G \)-module
simple \( L(\lambda) \) correspondant est isomorphe à l'unique
sous-module simple de \( H^{0}(\mu) \), dont le caractère est donné
par le formule de caractère de Weyl. Donc si on comprend bien les
structures de ces groupes de cohomologie, on pourra comprendre les
caractères des modules simples, qui est l'une des questions les plus
importantes dans la théorie des représentations modulaire.

En caractéristique \( 0 \), ce problème est complètement résolu, et la structure de \( H^{i}(\mu) \) est
simplement donnée par le Théorème de Borel-Weil-Bott (cf. \cite{Jan03}
II.5.5). Mais en caractéristique positive, le Théorème de
Borel-Weil-Bott n'est plus vrai, parce que s'il était vrai, alors pour
tout \( \mu \), il existerait au plus un \( i \) tel que \(
H^{i}(\mu)\neq 0 \). En 1978, Griffith (\cite{Gri80}) a étudié le
cas de \( G=\SL_{3} \) et déterminé la région de \( X(T) \), que l'on
appellera \og la région de Griffith\fg{}, où \( H^{1} \) et \( H^{2}
\) sont tous les deux non nuls. Presque simultanément en 1979,
Andersen (\cite{And79}) a découvert,
pour tout \( G \), une condition nécessaire et suffisante
pour que \( H^{1}(\mu)\neq 0 \). Il a aussi montré que chaque \(
H^{1}(\mu)\) non nul admet un socle simple. Ensuite, des résultats
concernant la structure de \( G \)-module de \( H^{i}(\mu) \) sous
certaines hypothèses de généricité ont été obtenus par différents
auteurs : \cite{Jan80},
\cite{KH84}, \cite{Irv86}, \cite{And86a}, \cite{And86b}, \cite{DS88}, \cite{Lin90}, \cite{Lin91}.
En 2002, Donkin a découvert une nouvelle approche, qui a donné dans
\cite{Don06}, des formules récursives pour les caractères de tous les
\( H^{i}(\mu) \) dans le cas de \( G=\SL_{3} \).

À ce stade, presque rien n'est connu pour la structure de \( G
\)-module de \( H^{i}(\mu) \) si \( i\neq 0 \) ou \( \dim G/B \) en
dehors du cas générique dans la \( p^{2} \)-alcôve du bas sauf le
socle de \( H^{1}(\mu) \).

\subsection{Résultats principaux}
Dans cet article, on étudiera le cas de \( G=\SL_{3} \), qui est le
premier cas non trivial, et on donnera
une description complète  récursive de la structure de \( H^{i}(\mu) \)
pour tout \( i \) et tout \( \mu \). Le théorème le plus important de
cet article est le suivant (voir le paragraphe \ref{subsection:Dfiltrationconclusion}):

\begin{thm*}
  Soit \( \mu\in X(T) \). Soit
\( 0=N_{0}\subset N_{1}\subset N_{2}\subset\cdots \subset N_{\ell}=\widehat{Z}(\mu) \)
une D-filtration de \(\widehat{Z}(\mu)\) (cf. le paragraphe
\ref{subsection:Zhat})
telle que \(N_{i}/N_{i-1}\cong
\widehat{L}(\nu_{i}^{0})\otimes E_{\delta_{i}}(\nu_{i}^{1})^{(1)}\) où
\(\delta_{i}\in \{0,\alpha,\beta\}\). Alors pour tout \( j\in
\mathbb{N} \), il existe une filtration 
\( 0=\widetilde{N_{0}}\subset \widetilde{N_{1}}\subset
\widetilde{N_{1}}\subset\cdots\subset \widetilde{N_{\ell}}=H^{j}(\mu) \)
où \( \widetilde{N_{i}}\cong H^{j}(G/BG_{1},N_{i}) \) et  \(\widetilde{N_{i}}/\widetilde{N_{i-1}}\cong L(\nu_{i}^{0})\otimes
H^{j}(E_{\delta_{i}}(\nu_{i}^{1}))^{(1)}.\)
\end{thm*}
Ce théorème généralise la \( p \)-filtration introduite par
Jantzen, pour tout \( \mu \) et tout \( i \).  On verra que les formules de récurrence
de Donkin correspondent à ces filtrations de \( H^{i}(\mu) \).
On obtiendra aussi comme corollaire une
autre démonstration de l'existence de la \( p \)-filtration de \( H^{0}(\mu) \) découvert
par Jantzen (\cite{Jan80}).
\begin{rmk}
  L'énoncé de ce théorème peut être généralisé à \( G \) arbitraire,
  qui fournira une conjecture de la structure pour tout
  \( H^{i}(\mu) \) en cas général. La filtration dans ce théorème est
  une version modifiée de la \( p \)-filtration de Jantzen même dans
  le cas où \(i=0 \), car on considère non seulement les fibrés en
  droites sur \( G/B \), mais aussi des fibrés vectoriels de rangs
  supérieurs (cette idée a été premièrement utilisée par Donkin dans \cite{Don02}). L'avantage de cette modification est claire dans le cas
  \( G=\SL_{3} \): on peut obtenir une description universelle pour
  tout \( i \) et \( \mu \) indépendamment de la position de
  \( \mu \). En particulier, cette nouvelle filtration explique les
  \og effacements\fg{} bizarres dans la \( p \)-filtration de Jantzen
  pour \( H^{0}(\mu) \) lorsque \( \mu \) est proche du mur
  (cf. \autoref{cor:pweylfiltrationjantzen}). Cette nouvelle idée
  devient encore plus intéressante après un contre-exemple de la
  \( p \)-filtration de Jantzen a récemment été trouvé dans l'article
  \cite{BNPS19}.
\end{rmk}

On montrera
aussi l'existence d'une filtration à deux étages de \( H^{1}(\mu) \)
et \( H^{2}(\mu) \) lorsque \( \mu \) est dans la région de Griffith
(\autoref{cor:2étages}). Cela fournira aussi des formules de
récurrence de \(\ch  H^{i}(\mu) \) pour tout \( i \) et \( \mu \), qui
sont complètement différentes de celles de
Donkin. Les formules de Donkin ont été utilisées par quelques travaux
récents (cf.\cite{AH19} et \cite{Har16}). Donc les nouvelles formules
de récurrence
obtenues par des résultats de cet article, qui sont plus simples que
celles de Donkin, seront utiles pour les autres chercheurs dans la
théorie des représentations géométrique. 

%\subsection{Structure de l'article}

\section{Notations et Préliminaires}
Dans cet article,  \( k \) désigne
un corps de caractéristique \( p>0 \),  \( G \)
désigne le \( k \)-schéma en groupes \( \SL_{3}\) sur \( k \), \( B \subset G\) est
le sous-groupe de Borel des matrices triangulaires inférieures, et \(
T\subset B \) est le tore maximal des matrices diagonales.

On note
  \( X(T) \) le groupe des
  caractères de \( T \) et \( Y(T)\) celui des cocaractères.  Notons
   \(  \langle \cdot,\cdot\rangle : X(T)\times Y(T) \to \mathbb{Z} \)
  le couplage naturel.
  Pour \( i\in\{1,2,3\} \), notons \(
  \epsilon_{i} \) l'élément de \( X(T) \) tel que \(
  \epsilon_{i}(\diag(a_{1},a_{2},a_{3}))=a_{i} \).

  Posons \(
  \alpha=\epsilon_{1}-\epsilon_{2} \), \(
  \beta=\epsilon_{2}-\epsilon_{3} \), \( \gamma=\alpha+\beta \), \(
  R^{+}=\{\alpha,\beta,\gamma\} \), et \( R^{-}=-R^{+} \).
  Alors \(R=\{\pm\alpha,\pm\beta,\pm\gamma\}\) est le système de
  racines de \( G \) par rapport à \( T \) 
  et le sous-groupe de Borel \( B \)
  correspond à \( R^{-} \).
  Notons \( \Delta=\{\alpha,\beta\} \)  l'ensemble des
  racines simples.  Définissons l'ordre partiel \( \leq \) sur \(
  X(T) \) par \( \mu\leq \lambda \) si et seulement si \(
  \lambda-\mu\in\mathbb{N}\alpha+\mathbb{N}\beta \).

  Pour tout \( \delta\in R \), notons \( \delta^{\vee}\in Y(T) \) la
  coracine correspondante. On désigne
  par \( \omega_{1},\omega_{2}\in X(T) \) les poids fondamentaux
  correspondant à \( \alpha^{\vee} \) et \( \beta^{\vee} \).
  Alors on a  \(
  X(T)=\mathbb{Z}\omega_{1}\oplus\mathbb{Z}\omega_{2} \). Pour tout \(
  a,b\in\mathbb{Z} \), notons \( (a,b)\) le poids \( a\omega_{1}+b\omega_{2}
  \). Posons \( \rho=\frac{1}{2}(\alpha+\beta+\gamma)=\gamma=(1,1)
  \). Notons \( X(T)^{+} \)
l'ensemble des poids dominants. Pour tout \( d\in\mathbb{N}^{*} \),
notons
\begin{displaymath}
  X_{d}(T)=\{\mu\in X(T)\mid 0\leq \langle\mu,\delta^{\vee}\rangle< p^{d}
  ,\forall \delta\in \Delta\}=\{(a,b)\in X(T)|0\leq a,b<p^{d}\}
\end{displaymath}
l'ensemble des poids dominants et \( p^{d} \)-restreints.

  Pour \( \delta\in R \), notons \( s_{\delta}\) la réflexion  par
  rapport à \( \delta\), c'est-à-dire, pour tout \( \mu\in X(T) \), 
  \( s_{\delta}(\mu)=\mu-\langle \mu,\delta^{\vee}\rangle \delta \).
  
 Soit
   \(  W\)
  le groupe de Weyl de \( R \), il est engendré par l'ensemble \( S \)
  des
  réflexions simples. La longueur \(
  \ell(w) \) d'un \( w\in W \) est le plus petit entier \( m \)
  tel que \( w \) s'écrive 
  \( s_{\alpha_{1}}s_{\alpha_{2}}\cdots s_{\alpha_{m}} \) avec \(
  \alpha_{i}\in S \). Soit \(
  w_{0}=s_{\alpha}s_{\beta}s_{\alpha}=s_{\beta}s_{\alpha}s_{\beta} \)
  l'unique élément de \( W \) de plus grande longueur.

  Pour \( \delta\in R \) et \( r\in\mathbb{Z} \), notons \(
  s_{\delta,r} \) la réflexion affine de \( X(T) \) définie par
   \(  s_{\delta,r}(\mu)=\mu-(\langle\mu,\delta^{\vee}\rangle-r)\delta \)
  pour tout \( \mu\in X(T) \). Désignons par \( W_{p} \) le groupe
  engendré par tous les \( s_{\delta,np} \) avec \( \delta\in R \) et
  \( n\in \mathbb{Z} \). Pour \( w\in W_{p} \), définissons l'action
  décalée par \( w\cdot \mu=w(\mu+\rho)-\rho \) pour tout \(
  \mu\in X(T) \). On note \( C=-\rho+X(T)^{+} \).

\medskip
  Tout \( G \)-module \( V \) est aussi un \( T \)-module de façon
  naturelle. Pour tout \( \mu\in X(T) \), on note
   \(  V_{\mu}\)
  l'espace de poids \( \mu \) de \( V \) et l'on dit que \( \mu \) est un poids de \( V \) si \( V_{\mu}\neq 0
  \). On dit que \( \mu \) est un plus haut poids de \( V \) si \( \mu
  \) est un poids de \( V \) qui est maximal par rapport à l'ordre \(
  \leq \) sur \( X(T) \). On définit le caractère de \( V \) par
   \(  \ch V=\sum_{\mu\in X(T)}\dim( V_{\mu})\, e_{\mu}\in \mathbb{Z}[X(T)] \).

  Soit \( H\subset G \) un sous-groupe fermé. Si \( V \) est un \( G \)-module, alors il admet naturellement une
  structure de \( H \)-module. On note \( \res_{H}^{G}(V) \) le \( H
  \)-module ainsi obtenu.
  
  Pour tout \( H
  \)-module \( N \), on note \( \ind_{H}^{G}(N) \) le \( G \)-module induit par
  \( N \). 
  Pour \( i\in\mathbb{N} \),
  on note
   \(  H^{i}(G/H, N)=H^{i}(G/H, \mathcal{L}_{G/H}(N)) \)
  où \( \mathcal{L}_{G/H}(N) \) est le fibré vectoriel \( G
  \)-équivariant  sur \( G/H \) associé à \( N \) (cf. \cite{Jan03} I.5). Alors on a
   \(  H^{i}(G/H, N)\cong R^{i}\ind_{H}^{G}(N) \). 
Pour un \( B \)-module \( N \), on note \( H^{i}(N)=H^{i}(G/B,N)
\). Si \( \mu\in X(T) \), alors \( \mu \) est aussi un caractère de \(
B \) par la composition \( B\twoheadrightarrow T\xrightarrow{\mu} \mathbb{G}_{m} \), et on désigne encore par \( \mu \) le \( B
\)-module de dimension \( 1 \) tel que \( g\in B\) agit comme le
scalaire \( \mu(g) \). Donc \( H^{i}(\mu) \) est défini comme
ci-dessus.

Pour \( \mu\in X(T)^{+} \), notons  \( L(\mu) \) le \( G \)-module
simple de plus haut poids \( \mu \). Notons aussi \(
V(\mu)=H^{3}(w_{0}\cdot \mu) \) le module de Weyl de plus haut poids \( \mu \).

Pour un \( G \)-module \( V \) de dimension finie, on note \( \KF(V) \) l'ensemble des
facteurs de composition de \( V \).

Pour \( i\in\{0,1,2,3\} \), on appelle \og \( H^{i} \)-chambre\fg{} tout
sous-ensemble de \( X(T) \) de la forme \( w\cdot C \) avec \(
\ell(w)=i \). Pour \( d\in\mathbb{N}^{*} \),
 une \og \( p^{d} \)-alcôve\fg{} est un ensemble de la forme
\begin{multline*}
  \{\mu\in X(T)\mid ap^{d}<\langle \mu+\rho,\alpha^{\vee}\rangle <(a+1)p^{d}, \quad
                bp^{d}<\langle \mu+\rho,\beta^{\vee}\rangle<(b+1)p^{d},\\
  cp^{d}<\langle \mu+\rho,\gamma^{\vee}\rangle<(c+1)p^{d}\}
\end{multline*}
pour certains \( a,b,c\in\mathbb{Z} \).

Pour tout \( G \)-module \( V \),  l'espace dual \(
\Hom_{k}(V,k) \)  est naturellement muni de la structure de \( G \)-module définie par \(
(g\cdot \phi )(v)=\phi(g^{-1}v)\). On le note \( V^{*} \) et on
l'appelle le dual de \( V \). La dualité de Serre sur \( G/B \) est compatible avec l'action de \( G
\), et donne
 \(  H^{i}(\mu)\cong H^{3-i}(-2\rho-\mu)^{*} \).

D'autre part, l'application \( g\mapsto {}^{t}g \) est un anti-automorphisme de \(
G=\SL_{3} \) qui est l'identité sur \( T \). On peut aussi  munir l'espace dual \(
\Hom_{k}(V,k) \) de la structure de \( G \)-module définie par \(
(g\cdot \phi)(v)=\phi(^{t}gv) \). On le note \( V^{t} \) et on
l'appelle \og le dual contravariant\fg{} de \( V \). Alors, \og la dualité de Serre
contravariante\fg{} s'écrit (cf. \cite{DS88} 2.1)
  \( H^{i}(\mu)\cong H^{3-i}(w_{0}\cdot \mu)^{t} \).

\medskip
  Soit \( F:G\to G \) le morphisme de Frobenius de \( G \). Pour tout \( r\in\mathbb{N}^{*} \), notons
\( G_{r}=\ker(F^{r}) \) le \( r \)-ième noyau de
Frobenius. Pour tout \( \mu\in X(T) \), notons \( \widehat{L}(\mu) \)
l'unique \( BG_{1} \)-module simple de plus haut poids \( \mu \), où \( BG_{1}=F^{-1}(B) \).  Si on écrit \( \mu=\mu^{0}+p\mu^{1} \)
avec \( \mu^{0}\in X_{1}(T) \) et \( \mu^{1}\in X(T) \), alors on a un
isomorphisme de \( BG_{1} \)-modules
  \( \widehat{L}(\mu)\cong \widehat{L}(\mu^{0})\otimes p\mu^{1} \).
De plus, si \( \mu\in X_{1}(T) \), alors on a un isomorphisme de \( BG_{1} \)-modules \(
\widehat{L}(\mu)\cong\res_{BG_{1}}^{G}(L(\mu)) \).

\section{Une filtration à deux étages}\label{chapitre:3étages}
\subsection{Énoncé du théorème pricipal}\label{subsection:2étagesénoncé}

 \begin{defi}[degré]\label{defi:degré}
   Soit \( n\in\mathbb{N} \). Si \( n\geq 1 \), on appelle degré de \(n\) l'unique
   \(d\in \mathbb{N}\) tel que \(p^{d}\leq n
   <p^{d+1}\). Si \( n=0 \), on
   dit que \( n \) est de degré \( -\infty \).

   Soit \(\mu\in X(T)\) tel que \( \mu\neq (-1,-1) \). Il existe un
   unique \( \lambda=(a,b)\in C\cap W\cdot \mu\). Le degré
   de \( \mu \) est défini comme le degré de \( a+b+1\in\mathbb{N} \). 
 \end{defi}

 \begin{rmk}
   Si \( \mu=(m,-n-2) \) avec \( m,n\in\mathbb{N} \), alors \(
   \mu=s_{\beta}\cdot (m-n-1,n)=s_{\beta}s_{\alpha}\cdot(n-m-1,m)
   \). Donc dans ce cas, le degré de \( \mu \) est celui de \( \max(m,n) \).
 \end{rmk}

 \begin{defi}[Condition de Griffith]\label{defi:Griffith}
   \begin{enumerate}
   \item On dit qu'un poids \(\mu\) vérifie la condition de Griffith
      s'il existe \(m,n,d\in \mathbb{N}^{*}\) et
     \mbox{\( a\in \{1,2,\cdots,p-1\} \)} tels que
     \begin{itemize}
     \item \(ap^{d}\leq m,n\leq (a+1)p^{d}-2\) ;
     \item \(\mu=(m,-n-2)\) ou \(\mu=(-n-2,m)\).
     \end{itemize}
  
     On appelle \og région de Griffith\fg{}, et l'on note \(\Gr\),
     l'ensemble des poids vérifiant la condition de Griffith.

    \item  On note \( \overline{\Gr} \) l'ensemble des poids \( \mu \) tels
     qu'il existe \(m,n,d\in \mathbb{N}^{*}\) et
     \mbox{\( a\in \{1,2,\cdots,p-1\} \)}  tels que
     \begin{itemize}
     \item \(ap^{d}-1\leq m,n\leq (a+1)p^{d}-1\) ;
     \item \(\mu=(m,-n-2)\) ou \(\mu=(-n-2,m)\).
     \end{itemize}

    \item  On note \( \widehat{\Gr} \) l'ensemble des poids \( \mu \) tels
     qu'il existe \(m,n,d\in \mathbb{N}^{*}\) et
     \mbox{\( a\in \{1,2,\cdots,p-1\} \)}  tels que
     \begin{itemize}
     \item \(ap^{d}\leq m,n\leq (a+1)p^{d}-1\) ;
     \item \(\mu=(m,-n-2)\) ou \(\mu=(-n-2,m)\).
     \end{itemize}
   \end{enumerate}

 \end{defi}
 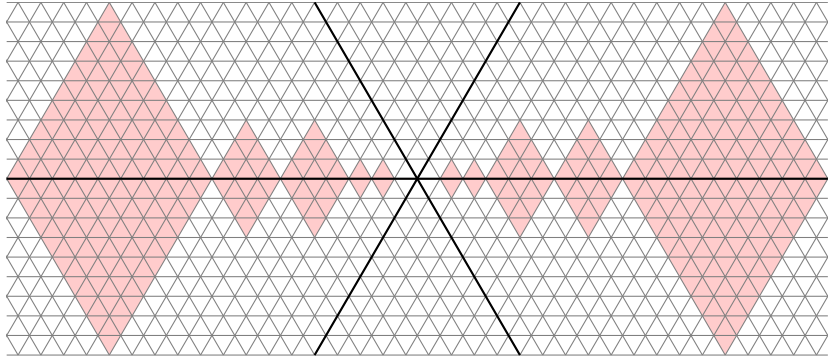
\begin{figure}[H]
    \centering
    \begin{tikzpicture}[scale=0.3]
      \clip (-18,-9*sin{60}) rectangle (18,9*sin{60});
  \draw (0,0)--(18,0);
  \fill [red, opacity=0.2] (1,0)--(1.5,
  sin{60})--(2,0)--(1.5,-sin{60})--cycle;
  \fill [red, opacity=0.2] (2,0)--(2.5,
  sin{60})--(3,0)--(2.5,-sin{60})--cycle;
  \fill [red,opacity=0.2]
  (3,0)--(4.5,3*sin{60})--(6,0)--(4.5,-3*sin{60})--cycle;
  \fill [red,opacity=0.2]
  (6,0)--(7.5,3*sin{60})--(9,0)--(7.5,-3*sin{60})--cycle;
   \fill [red,opacity=0.2]
   (9,0)--(13.5,9*sin{60})--(18,0)--(13.5,-9*sin{60})--cycle;

   \draw (0,0)--(-18,0);
  \fill [red, opacity=0.2] (-1,0)--(-1.5,
  sin{60})--(-2,0)--(-1.5,-sin{60})--cycle;
  \fill [red, opacity=0.2] (-2,0)--(-2.5,
  sin{60})--(-3,0)--(-2.5,-sin{60})--cycle;
  \fill [red,opacity=0.2]
  (-3,0)--(-4.5,3*sin{60})--(-6,0)--(-4.5,-3*sin{60})--cycle;
  \fill [red,opacity=0.2]
  (-6,0)--(-7.5,3*sin{60})--(-9,0)--(-7.5,-3*sin{60})--cycle;
   \fill [red,opacity=0.2]
   (-9,0)--(-13.5,9*sin{60})--(-18,0)--(-13.5,-9*sin{60})--cycle;
   \foreach \i in
    {-23,...,23}{\draw[gray] (\i -4.5, 9*sin{60})--(\i +4.5,-9*sin{60});
     \draw[gray] (\i -4.5, -9*sin{60})--(\i +4.5,9*sin{60});} \foreach \i
   in {-9,...,9}{\draw[gray] (-18, \i *sin{60})--(18, \i *sin{60});}

  \draw [line width=0.8] (-4.5,9*sin{60})--(4.5,-9*sin{60});
 \draw [line width=0.8] (-4.5,-9*sin{60})--(4.5,9*sin{60}); 
  \draw [line width=0.8] (-18,0)--(18,0);
\end{tikzpicture}
    \caption{Région de Griffith pour \( p=3 \)}
  \end{figure}
\begin{rmk}\label{label:griffith}
  Dans la Définition \ref{defi:Griffith}, le degré de \(\mu\) est \(d\).
\end{rmk}

\begin{rmk}\label{rmk:Griffithcondition}
D'après \cite{Gri80} Theorem 1.3 ou \cite{And79} Theorem 3.6, on sait
que  \( H^{1}(\mu)\) et \( H^{2}(\mu) \) sont tous les deux non
nuls  si et seulement si \(
\mu\in\Gr \). Si \( \mu \) est dans une \( H^{1} \)-chambre (resp. \( H^{2} \)-chambre) et \(
\mu\notin \Gr \), alors \( H^{2}(\mu)=0 \) (resp. \( H^{1}(\mu)=0 \)).
  
\end{rmk}

Le théorème principal de \S\ref{chapitre:3étages} est le suivant.
\begin{thm} \label{cor:2étages}
 Soit \(\mu=(m,-n-2)\in \overline{\Gr}\), où
  \(m=ap^{d}+r\) et \(n=ap^{d}+s\) avec \(d\geq 1\), \(0\leq a\leq
  p-1\) et \(-1\leq r,s\leq p^{d}-1\). Posons \(\mu'=(r,-s-2)\),
  \(\mu''=(-p^{d}+r,p^{d}-s-2)\), \(\lambda=(s,p^{d}-r-2)\) et
  \(^{t}\lambda=(r,p^{d}-s-2)\). Alors :

  \begin{enumerate}
  \item Il existe une suite exacte courte de \( G
    \)-modules:
    \[
  \begin{tikzcd}
    0\ar[r]&M\ar[r]&H^{2}(\mu)\ar[r] &L(0,a-1)^{(d)}\otimes V(\lambda)\ar[r]&0
  \end{tikzcd}
\]
telle que
\[
M\cong  L(0,a)^{(d)} \otimes H^2(\mu') \bigoplus L(0,a-2)^{(d)} \otimes H^2(\mu'').
\]
De plus, le quotient de \( H^{2}(\mu) \) par \( L(0,a)^{(d)}\otimes
H^{2}(\mu') \) est un quotient du module de Weyl \( V(s,ap^{d}-r-2) \).   
  \item Il existe une suite exacte courte de \( G
\)-modules:
\[
  \begin{tikzcd}
    0\ar[r]&L(0,a-1)^{(d)}\otimes
    H^{0}(^{t}\lambda)\ar[r]&H^{1}(\mu)\ar[r]&Q\ar[r]&0
  \end{tikzcd}
\]
telle que
\[Q\cong  L(0,a)^{(d)} \otimes H^1(\mu') \bigoplus L(0,a-2)^{(d)}
  \otimes H^1(\mu'').\]
De plus, le noyau de la projection \( H^{1}(\mu)\to
L(0,a)^{(d)}\otimes H^{1}(\mu') \) est un sous-module de \( H^{0}(r,ap^{d}-s-2) \).
\end{enumerate}

\end{thm} 

\begin{convention}\label{convention:nondominant}
  Si \(\eta\) n'est pas dominant, on pose \(L(\eta)=V(\eta)=0\). En
  particulier, si \( a=1 \) on a \( L(0,a-2)=0 \).
\end{convention}

Afin de démontrer le \autoref{cor:2étages} on a besoin de quelques
lemmes.

\begin{lemma}\label{lemma:poidsrestreints}
  Soient \( \mu'=(r,-s-2) \) et \(
  \mu''=(-p^{d}+r,p^{d}-s-2) \) avec \( -1\leq r,s\leq p^{d}-1
  \) et \( d\geq 1 \).
  \begin{enumerate}
  \item Si \( L(\eta) \) est un facteur de composition de \(H^{i}(\mu')\) ou
    \(H^{i}(\mu'')\), alors \(\eta\) est \(p^{d}\)-restreint.
  \item Si \( s\leq r+1 \) et si  \(
  L(\eta) \) est un facteur de composition de \(V(\lambda)= V(s,p^{d}-r-2) \),
  alors \( \eta \) est \( p^{d} \)-restreint.
  \end{enumerate}
\end{lemma}
\begin{proof}
  Soit \(\zeta\in -\rho+X(T)^{+}\). On sait, d'après le \og Strong
  Linkage Principle\fg{} (\cite{Jan03} II.6.13), que pour tout facteur de composition
  \(L(\eta)\) d'un \(H^{i}(w\cdot \zeta)\) on a \(\eta\leq
  \zeta\). Comme \(\gamma\) est dominant, on a donc:
  \[
\langle\eta,\alpha^{\vee}\rangle\leq \langle\eta,\gamma^{\vee}\rangle\leq \langle\zeta,\gamma^{\vee}\rangle
\]
et de même pour \(\langle\eta,\beta^{\vee}\rangle\).

Pour \(\mu'=(r,-s-2)\), le \(\zeta\) correspondant est \((r-s-1,s)\)
si \(r\geq s\) et \((s-r-1,r)\) si \(s\geq r\). Dans les deux cas on a
\(\langle \zeta,\gamma^{\vee}\rangle=\max(r,s)-1<p^{d}\).

De même, pour \(\mu''=(r-p^{d},p^{d}-s-2)\), le poids \(\zeta\)
correspondant est \((p^{d}-r-2,r-s-1)\) si \(r\geq s\) et
\((p^{d}-s-2,s-r-1)\) si \(s\geq r\). Dans les deux cas on a
\(\langle\zeta,\gamma^{\vee}\rangle=p^{d}-\min(r,s)-3<p^{d}\).

Si \( s\leq r+1 \) et \( L(\eta) \) est un facteur de composition de
\(V(\lambda) =V(s,p^{d}-r-2)\cong H^{3}(w_{0}\cdot \lambda) \), alors
dans ce cas
\( \langle \zeta,\gamma^{\vee}\rangle=p^{d}+s-r-2 \leq p^{d}-1\).
 \end{proof}

\begin{lemma}\label{lemma:extension}
 Soit \( d\in \mathbb{N}^{*} \) et  soient \(\lambda,\mu\in X_{d}(T)\). Alors on a
  \[
\Ext^{1}_{G}(L(0,a)^{(d)}\otimes L(\lambda), L(0,a-2)^{(d)}\otimes L(\mu))=0.
  \]
\end{lemma}
\begin{proof}
  Raisonnons par récurrence sur \(d\).

  Si \(d=1\), alors \(\lambda,\mu\in X_{1}(T)\). Si \(\lambda=\mu\),
  alors d'après \cite{Jan03} II.10.17(2), 
  \[\Ext^{1}_{G}\Big(L(0,a)^{(1)}\otimes L(\lambda), L(0,a-2)^{(1)}\otimes
    L(\lambda)\Big)\cong \Ext^{1}_{G}\Big(L(0,a),L(0,a-2)\Big)=0
  \]
  car \((0,a-2)\notin W_{p}\cdot (0,a)\). Si \(\lambda\neq \mu\),
  d'après \cite{Ye82} Proposition 4.1.1, on sait que si \(\Ext^{1}_{G}(L(0,a)^{(1)}\otimes L(\lambda), L(0,a-2)^{(1)}\otimes
    L(\mu))\) est non nul, alors si \( p\neq 3 \) il est
    parmi les trois possibilités suivantes (et est leur somme directe si \( p=3 \)):
      \begin{align*}
        \Hom_{G}(L(0,a),L(0,a-2)),\\
        \Hom_{G}(L(0,a),L(0,1)\otimes
        L(0,a-2)),\\
        \Hom_{G}(L(0,a),L(1,0)\otimes L(0,a-2)).
      \end{align*}

\noindent Or ceux-ci sont tous nuls car \((0,a)\not\leq \nu_{0}+(0,a-2)\) pour
\(\nu_{0}\in \{(0,0),(0,1),(1,0)\}\).

\smallskip
Supposons que l'énoncé est vrai pour \(d\geq 1\). Soient
\(\lambda,\mu\in X_{d+1}(T)\). Écrivons
\(\lambda=p\lambda^{1}+\lambda^{0}\) et \(\mu=p\mu^{1}+\mu^{0}\) avec
\(\lambda^{0}, \mu^{0}\in X_{1}(T)\). Si \(\lambda_{0}=\mu_{0}\),
alors
\begin{multline*}
  \Ext^{1}_{G}\Big(L(0,a)^{(d+1)}\otimes L(\lambda), L(0,a-2)^{(d+1)}\otimes
  L(\mu)\Big)\\
  \cong \Ext^{1}_{G}\Big(L(0,a)^{(d)}\otimes
  L(\lambda^{1}),L(0,a-2)^{(d)}\otimes L(\mu^{1})\Big)=0
\end{multline*}
d'après \cite{Jan03} II.10.17 (2) et  l'hypothèse de récurrence.

Si \(\lambda_{0}\neq \mu_{0}\), alors d'après \cite{Ye82} Proposition 4.1.1, on sait que si \(\Ext^{1}_{G}(L(0,a)^{(d+1)}\otimes L(\lambda), L(0,a-2)^{(d+1)}\otimes
    L(\mu))\) est non nul, alors  si \( p\neq 3 \) il est
    parmi les trois possibilités suivantes (et est leur somme directe si \( p=3 \)):
    
    \begin{equation}\label{eq:possibleExt}
    \begin{aligned}
      \Hom_{G}\Big(L((0,a)p^{d}+\lambda^{1}), L((0,a-2)p^{d}+\mu^{1})\Big),\\
      \Hom_{G}\Big(L((0,a)p^{d}+\lambda^{1}),
      L((0,a-2)p^{d}+\mu^{1})\otimes L(0,1)\Big),\\
      \Hom_{G}\Big(L((0,a)p^{d}+\lambda^{1}),
      L((0,a-2)p^{d}+\mu^{1})\otimes L(1,0)\Big).
    \end{aligned}
  \end{equation}

Soit \(L(\eta)\)  un facteur de composition de \(L((0,a-2)p^{d}+\mu^{1})\otimes L(\nu_{0})\), où
\(\nu_{0}\in\{(0,0),(0,1),(1,0)\}\). Alors on a
\[\eta\leq (0,a-2)p^{d}+\mu^{1}+\nu_{0}.\]
Donc, comme \(\mu^{1}\) est \(p^{d}\)-restreint, 
\[
\langle \eta,\gamma^{\vee}\rangle \leq
\langle(0,a-2)p^{d}+\mu^{1}+\nu_{0},\gamma^{\vee}\rangle\leq (a-2)p^{d}+2(p^{d}-1)+1=ap^{d}-1.
\]
 Donc comme \(\lambda^{1}\) est dominant, on ne peut pas avoir \(\eta=
(0,a)p^{d}+\lambda^{1}\). Par
conséquent, tous les \( \Hom \) de \eqref{eq:possibleExt} sont nuls, d'où le résultat.

\end{proof}

\subsection{Démonstration du 
 \texorpdfstring{ \autoref{cor:2étages} : réduction au \autoref{thm2}}{}}\label{subsection:2etages}
Dans ce paragraphe, on va montrer que le \autoref{cor:2étages} découle
du théorème un peu plus faible suivant:
\begin{thm}\label{thm2}
  Soit \(\mu=(m,-n-2)\), où
  \(m=ap^{d}+r\) et \(n=ap^{d}+s\) avec \(d\geq 1\), \(1\leq a\leq
  p-1\) et \(-1\leq r,s\leq p^{d}-1\) (c'est-à-dire, \(
  \mu\in\overline{\Gr} \) de degré \( d \)). Posons \(\mu'=(r,-s-2)\),
  \(\mu''=(-p^{d}+r,p^{d}-s-2)\), \(\lambda=(s,p^{d}-r-2)\) et
  \(^{t}\lambda=(r,p^{d}-s-2)\). Alors :
  \begin{enumerate}
  \item Il existe des suites exactes courtes de
\(G\)-modules:
\[
  \begin{tikzcd}
    0\ar[r]&M\ar[r]&H^{1}(\mu)\ar[r]&L(0,a-2)^{(d)}\otimes H^{1}(\mu''),
  \end{tikzcd}
\]
\[
  \begin{tikzcd}
    0\ar[r]&L(0,a-1)^{(d)}\otimes
    H^{0}(^{t}\lambda)\ar[r]&M\ar[r]&L(0,a)^{(d)}\otimes H^{1}(\mu')\ar[r]&0.
  \end{tikzcd}
\]

\item Il existe des suites exactes courtes de
\(G\)-modules:
\[
  \begin{tikzcd}
    0\ar[r]&L(0,a)^{(d)}\otimes H^{2}(\mu')\ar[r]&H^{2}(\mu)\ar[r] &Q\ar[r]&0,
  \end{tikzcd}
\]
\[
  \begin{tikzcd}
    0\ar[r]&L(0,a-2)^{(d)}\otimes
    H^{2}(\mu'')\ar[r]&Q\ar[r]&L(0,a-1)^{(d)}\otimes V(\lambda)\ar[r]&0. 
  \end{tikzcd}
\]
De plus, \(Q\) est un quotient du module de Weyl \(V(s,ap^{d}-r-2)\).
\end{enumerate}
\end{thm}

Montrons que le \autoref{cor:2étages} découle du \autoref{thm2}.

On pose \( w=s_{\gamma}s_{\beta}=s_{\beta}s_{\alpha} \). Notons \(
\overline{\Gr}_{\alpha}=\overline{\Gr}\cap s_{\alpha}\cdot C \), \( \overline{\Gr}_{\beta}=\overline{\Gr}\cap
s_{\beta}\cdot C \) et \( \overline{\Gr}_{w}=\overline{\Gr}\cap w\cdot C \).

 Posons 
\[
\widetilde{\mu} = w_{0}\cdot \mu = (ap^d + s, -ap^d - r- 2).
\]
Alors \(\widetilde{\mu}\) appartient à \(\overline{\Gr}_w\) (resp. à \(\overline{\Gr}_\beta\)) si et seulement si \(\mu\) appartient à \(\overline{\Gr}_\beta\) 
(resp. à \(\overline{\Gr}_w\)). D'autre part, comme \(\widetilde{\mu}\) se déduit de \(\mu\) en échangeant \(r\) et \(s\), alors 
le poids \((\widetilde{\mu})'\) associé à \( \widetilde{\mu} \) est \((s,-r-2) = s_\gamma \cdot \mu'\); on le notera 
\(\widetilde{\mu}'\). 
De même, le poids \(\widetilde{\mu}''\) associé à \(\widetilde{\mu}\) est 
\[
(-p^d + s, p^d - r-2) = s_\gamma \cdot \mu''. 
\]
 
Par dualité de Serre contravariante, on a:
\[
H^1(\mu) \simeq H^2(\widetilde{\mu})^t \qquad \text{et }\qquad H^2(\mu) \simeq H^1(\widetilde{\mu})^t 
\]
et de même \(H^i(\mu') \simeq H^{3-i}(\widetilde{\mu}')^t\) et \(H^i(\mu'')\simeq H^{3-i}(\widetilde{\mu}'')^t\) pour \(i= 1,2\). 
Comme les modules simples \(L(0,i)\) sont auto-duaux pour la dualité contravariante, 
on obtient que \(H^1(\mu)\) a aussi la filtration à trois étages suivante: 
\[
H^1(\mu) \simeq H^2(\widetilde{\mu})^t = 
\begin{array}{|c|}
\hline 
L(0,a)^{(d)} \otimes H^1(\mu') \vphantom{\int_A^B}
\\
\hline 
L(0,a-2)^{(d)} \otimes H^1(\mu'') \vphantom{\int_A^B}
\\
\hline 
L(0,a-1)^{(d)} \otimes H^0(^{t}\lambda) \vphantom{\int_A^B}
\\
\hline 
\end{array} 
\]
où les deux étages inférieurs sont un sous-module de \(H^0(r, ap^d - s - 2)\), 
et \(H^2(\mu)\) a aussi la filtration à trois étages suivante: 
\[
H^2(\mu) = H^1(\widetilde{\mu})^t = 
\begin{array}{|c|}
\hline 
L(0,a-1)^{(d)} \otimes V(\lambda)  \vphantom{\int_A^B}
\\
\hline 
L(0,a)^{(d)} \otimes H^2(\mu') \vphantom{\int_A^B}
\\
\hline 
L(0,a-2)^{(d)} \otimes H^2(\mu'') \vphantom{\int_A^B}
\\
\hline 
\end{array} .
\]

Donc pour montrer le \autoref{cor:2étages},
  il suffit de
  montrer que  pour \(i\in\{1,2\}\), on a:
  \[
    \Ext_{G}^{1}(L(0,a)^{(d)}\otimes H^{i}(\mu'),L(0,a-2)^{(d)}\otimes
  H^{i}(\mu'') )=0.
  \]
  Or ceci résulte des lemmes
  \ref{lemma:poidsrestreints} et  \ref{lemma:extension}. Ceci montre que le \autoref{cor:2étages}
  découle du \autoref{thm2}. On va montrer le \autoref{thm2} dans le paragraphe \ref{subsection:preuvethm2}.

\subsection{Preuve du Théorème \texorpdfstring{\ref{thm2}}{1}}\label{subsection:preuvethm2}

Commençons par le lemme suivant.
\begin{lemma}\label{lemma:filtweyl}
  Soit \(\lambda = (0,a) \in X^+\) tel que \(1\leq a\leq p-1\).  
 Soit \(K\) le sous-\( B \)-module de \(L(\lambda)\) engendré par le
 vecteur de poids
\((a,-a)\). 
Alors \(L(\lambda)/K\) est isomorphe comme \(B\)-module à \(L(0,a-1) \otimes (0,1)\). 
\end{lemma} 

\begin{proof}
  On sait que \( L(0,a)\cong k[x,y,z]_{a} \) l'espace des
  polynômes homogènes de degré \( a \) avec l'action naturelle de \(
  \SL_{3} \). Alors, on a un morphisme surjectif de \( B \)-modules
  \begin{displaymath}
    L(0,a)\to L(0,a-1)\otimes (0,1),\quad x^{i}y^{j}z^{a-i-j}\mapsto \frac{a-i-j}{a}x^{i}y^{j}z^{a-i-j-1}
  \end{displaymath}
  dont le noyau est \( K \).
\end{proof} 
\subsubsection{Trois  suites exactes de \texorpdfstring{\(B\)}{B}-modules }\label{subsection:L(ab)}
Appliquons le \autoref{lemma:filtweyl} à  \(L(0,a)\)  et 
 notons \(K_{a}\) le sous-module engendré par 
 le vecteur de poids \((a,-a)\); il est isomorphe comme \(B\)-module
 au \(P_\alpha\)-module simple 
\(L_\alpha(a,-a)\) de plus haut poids \((a,-a)\) et \(L(0,a)/K_{a}\) est
isomorphe à \(L(0,a-1)\otimes (0,1)\).
On a donc une suite exacte 
\begin{equation}
\xymatrix{0 \ar[r] & K_{a} \ar[r] & L(0,a) \ar[r] & L(0,a-1)\otimes (0,1) \ar[r] & 0. 
}\label{eq:b0c2fdba8e6675fd}
\end{equation}
Notons \(M_{a}\) le sous-module de \(K_{a}\) tel qu'on ait une suite exacte 
\begin{equation}
\xymatrix{0 \ar[r] & M_{a} \ar[r] & K_{a} \ar[r] & (a,-a) \ar[r] & 0.
}\label{eq:0b4aa8ab1e47bee9}
\end{equation}
Comme \(a<p\) on voit que \(M_a \simeq K_{a-1} \otimes (-1,0)\) et donc on a une suite exacte 
\begin{equation}
\xymatrix{ 0\ar[r] & M_a \ar[r] & L(0,a-1) \otimes (-1,0) \ar[r] & L(0,a-2) \otimes (-1,1) \ar[r] & 0 
.}\label{eq:f8ccc21dc40b71a7}
\end{equation}

\subsubsection{Suites exactes longues induites par le foncteur d'induction}\label{subsubsection:foncteurinduction}
Appliquons la \(d\)-ième puissance 
du Frobenius aux suites exactes courtes du paragraphe précédent et tensorisons par le poids \(\mu'=(r,-s-2)\). Posons aussi  
\( \lambda_{0} =  (s, ap^d - r-2) \) et \( \nu=(-p^{d}+r,-s-2) \)
et remarquons que \(w_0\cdot \lambda_{0} = w_0\lambda_{0} - 2\rho = (r-ap^d, -s-2)\). 
On obtient alors des suites exactes:
  \begin{equation}
\xymatrix{0 \ar[r] & \widetilde{K}_{a} \ar[r] & L(0,a)^{(d)} \otimes (r,-s-2)
  \ar[r] & L(0,a-1)^{(d)} \otimes (r,p^d - s -2) \ar[r] & 0}
\label{suite:main1}
\end{equation}
  \begin{equation}
  \xymatrix{0 \ar[r] & \widetilde{M}_{a} \ar[r] & \widetilde{K}_{a} \ar[r] & (m,-n-2) \ar[r] & 0
}\label{suite:main2}
\end{equation}
\begin{equation}
\xymatrix{0 \ar[r] & \widetilde{M}_a \ar[r] & L(0,a-1)^{(d)} \otimes \nu \ar[r] & L(0,a-2)^{(d)} \otimes (-p^d+r, p^d-s-2) \ar[r] & 0} .
\label{suite:star}
\end{equation}

Appliquons le foncteur \(H^0\) à ces suites exactes. Comme \( \nu\in
w_{0}\cdot C \), on a \( H^{i}(\nu)=0 \) pour \( i<3 \). Par
conséquent, \eqref{suite:star} donne l'isomorphisme
\begin{equation}
L(0,a-2)^{(d)} \otimes H^1(\mu'') \simeq
H^2(\widetilde{M}_a)\label{suite:iota}
\end{equation}
et la suite exacte: 
  \begin{equation}
      0 \to L(0,a-2)^{(d)} \otimes H^2(\mu'') \to
      H^3(\widetilde{M}_a) \to 
      L(0,a-1)^{(d)} \otimes V(s, p^d - r-
      2) \to 0
\label{suite:main8}
\end{equation}
où l'on a posé, comme dans le \autoref{thm2},
\(\mu''=(-p^{d}+r,p^{d}-s-2)\).

Comme
\(^{t}\lambda=(r,p^d - s -2)\) appartient à \( C\)  et comme 
\((r,-s-2)\) n'a de la cohomologie qu'en degré \(1\) et \(2\), alors \eqref{suite:main1} donne, 
en utilisant l'identité tensorielle (\cite{Jan03} I.4.8): l'égalité \(H^0(\widetilde{K}_{a}) = 0\), la suite exacte 
\begin{equation}\hskip-4mm
      0 \to L(0,a-1)^{(d)} \otimes H^0(r,p^d - s -2) \to
      H^1(\widetilde{K}_{a}) \to
      L(0,a)^{(d)} \otimes H^1(r,- s -2) \to 0,
\label{suite:main4}
\end{equation}
l'isomorphisme 
  \begin{equation}
 H^2(\widetilde{K}_{a}) \simeq   L(0,a)^{(d)} \otimes H^2(r,- s -2)\label{eq:7}
\end{equation}
 et l'égalité  \(H^3(\widetilde{K}_{a}) = 0\). 

\smallskip Considérons maintenant la suite exacte \eqref{suite:main2}. Comme on a vu que \(H^1(\widetilde{M}_{a}) = 0\), on obtient la suite 
exacte: 
\begin{equation}
    0 \to H^1(\widetilde{K}_{a}) \to H^1(m,-n-2) \to
    H^2(\widetilde{M}_{a}) \xrightarrow{f}
    H^2(\widetilde{K}_{a}) \to H^2(m,-n-2)
    \to H^3(\widetilde{M}_{a}) \to 0.
\label{suite:main7}
\end{equation}

\subsubsection{Annulation de \texorpdfstring{\(f\)}{f}}
\begin{lemma}
  Le morphisme  \(f\) dans la suite exacte \eqref{suite:main7}
  est nul. 
\end{lemma}
\begin{proof}
  Par \eqref{eq:7}, on sait que \(H^{2}(\widetilde{K}_{a})\cong
  L(0,a)^{(d)}\otimes H^{2}(\mu')\). Donc par le
  \autoref{lemma:poidsrestreints}, 
  si \(L(\eta)\) est un facteur
  de composition de \(H^{2}(\widetilde{K}_{a})\), alors \(\eta=(0,ap^{d})+\eta_{0}\)
  où \(\eta_{0}\) est un poids dominant \(p^{d}\)-restreint. De même,
  comme \(H^{2}(\widetilde{M}_{a})\cong L(0,a-2)^{(d)}\otimes H^{1}(\mu'')\)
  par \eqref{suite:iota}, si \(L(\eta)\) est un facteur de composition
  de \(H^{2}(\widetilde{M}_{a})\), alors \(\eta=(0,(a-2)p^{d})+\eta_{0}\) où
  \(\eta_{0}\) est dominant et \(p^{d}\)-restreint.

  Par conséquent, \(H^{2}(\widetilde{K}_{a})\) et \(H^{2}(\widetilde{M}_{a})\) n'ont
  pas de facteur de composition commun. Donc le morphisme \( f \)
  de \(H^{2}(\widetilde{M}_{a})\) vers \(H^{2}(\widetilde{K}_{a})\)  dans
  \eqref{suite:main7} est  nul. 
\end{proof}

Par conséquent,  la suite exacte \eqref{suite:main7} 
se coupe en deux suites exactes courtes: 
  \begin{equation}
\xymatrix{
0 \ar[r] & L(0,a)^{(d)} \otimes H^2(\mu') \ar[r] &H^2(m,-n-2) \ar[r] &H^3(\widetilde{M}_{a}) \ar[r] &0 
}\label{suite:main9}
\end{equation}
\begin{equation}
\xymatrix{
0 \ar[r] &H^1(\widetilde{K}_{a}) \ar[r] &H^1(m,-n-2) \ar[r]
&L(0,a-2)^{(d)}\otimes H^{1}(\mu'') \ar[r] & 0 
.}\label{suite:main10}
\end{equation}

Celles-ci, avec la suite exacte \eqref{suite:main8} et la suite exacte
\eqref{suite:main4}, terminent la preuve du \autoref{thm2}.

\subsection{Description de \( H^{2}(\mu) \) et \( H^{1}(\mu) \) pour \(
  \mu \) sur le mur}\label{subsection:surlemur}
Lorsque  \(\mu\) se situe sur le mur entre une
\(H^{1}\)-chambre et une \(H^{2}\)-chambre, c'est-à-dire,
\(\mu=(n,-n-2)\) ou \((-n-2,n)\) pour un \(n\in \mathbb{N}\), on peut
donner une version plus précise du \autoref{cor:2étages}. Par la
symétrie entre \(\alpha\) et \(\beta\), il suffit de considérer le cas
où \(\mu=(n,-n-2)\).

Remarquons d'abord que si \(0\leq n\leq p-1\), on peut appliquer le théorème de
Borel-Weil-Bott (cf. \cite{Jan03} II.5.5) à  \(\mu=(n,-n-2)=s_{\beta}\cdot
(-1,n)\). Donc on a
\(H^{i}(\mu)\cong H^{i-1}(-1,n)=0\) pour tout \(i\) dans ce cas.

Si \(n\geq p\), on a  le théorème suivant:

\begin{thm}\label{thm:murfilt}
  Soit \(\mu=(n,-n-2)\)  de degré \(d\geq 1\)
  (c'est-à-dire, \(n\geq p\)). Alors il existe une filtration 
  \(
0=V_{0}\subset V_{1}\subset \cdots\subset V_{\ell-1}\subset V_{\ell}=H^{2}(\mu)
\),
avec \(\ell\leq d\)  telle que pour tout \(i\in\{1,2,\cdots,\ell\}\), on ait
\[V_{i}/V_{i-1}\cong \bigoplus_{j=1}^{q_{i}}
  L(\nu_{ij})^{(d_{ij})}\otimes V(\lambda_{ij}),\]
avec \(q_{i}\leq 2^{\ell-i}\). De plus, 
\(p^{d_{ij}}\nu_{ij}\) est \(p^{d+1}\)-restreint et  \(\lambda_{ij}\) est
\(p^{d_{ij}}\)-restreint  pour tout \(i,j\).

\smallskip
Comme \( H^{1}(\mu)\cong H^{2}(\mu)^{t} \), on obtient aussi une
filtration duale de \( H^{1}(\mu) \).
\end{thm}

\begin{proof}
  Raisonnons par récurrence sur \(d\).
  Si \(d=1\), alors \(n=ap+r\) avec \(1\leq a,r\leq p-1\). D'après le
  \autoref{cor:2étages}, il existe une filtration à deux étages:
\[
    H^2(\mu) =
    \begin{array}{|c|}
      \hline 
      L(0,a-1)^{(1)} \otimes V(\lambda)  \vphantom{\int_A^B}
      \\
      \hline 
      L(0,a)^{(1)} \otimes H^2(\mu') \bigoplus L(0,a-2)^{(1)} \otimes H^2(\mu'') \vphantom{\int_A^B}
      \\
      \hline 
    \end{array}
  \]
  où \(\lambda=(r,p-r-2)\), \(\mu'=(r,-r-2)\) et
  \(\mu''=(-p+r,p-r-2)\). Comme \( r \) et \( p-r-2 \) sont \( \leq p-1 \), on a
  \(H^{2}(\mu')=H^{2}(\mu'')=0\), d'où \(H^{2}(\mu)\cong
  L(0,a-1)^{(1)}\otimes V(\lambda)\). Donc l'énoncé est vrai dans ce
  cas.

  Supposons l'énoncé  vrai pour tout \( n \) de degré \(\leq d\), et
  soit
  \(n=ap^{d+1}+r\) avec \(1\leq a\leq p-1\) et \(0\leq r\leq
  p^{d}-1\). D'après le
  \autoref{cor:2étages}, on a  une filtration à deux étages:
  \[
    H^2(\mu) =
    \begin{array}{|c|}
      \hline 
      L(0,a-1)^{(d+1)} \otimes V(\lambda)  \vphantom{\int_A^B}
      \\
      \hline 
      L(0,a)^{(d+1)} \otimes H^2(\mu') \bigoplus L(0,a-2)^{(d+1)} \otimes H^2(\mu'') \vphantom{\int_A^B}
      \\
      \hline 
    \end{array} 
  \]
où \(\lambda=(r,p^{d+1}-r-2)\), \(\mu'=(r,-r-2)\) et
\(\mu''=(-p^{d+1}+r,p^{d+1}-r-2)=(-m-2,m)\), où \( m=p^{d+1}-r-2 \). Donc
\(\mu'\) et \(\mu''\) sont tous les deux encore sur le mur et de degré
\(\leq d\). D'après l'hypothèse de récurrence, il
existe une filtration de \( H^{2}(\mu')\):
\[
0=M_{0}\subset M_{1}\subset\cdots\subset
M_{\ell'}= H^{2}(\mu')
\]
avec  \(\ell'\leq d\) telle que  pour tout \(i\in\{1,2,\cdots,\ell'\}\), on ait
\[M_{i}/M_{i-1}\cong \bigoplus_{j=1}^{q'_{i}}
  L(\nu'_{ij})^{(d'_{ij})}\otimes V(\lambda'_{ij}),\]
avec \(q'_{i}\leq 2^{\ell'-i}\). De plus, 
\(p^{d'_{ij}}\nu'_{ij}\) est \(p^{d+1}\)-restreint et  \(\lambda'_{ij}\) est
\(p^{d'_{ij}}\)-restreint  pour tout \(i,j\). Pour \( i>\ell' \), posons \(M_{i}=M_{\ell'}=H^{2}(\mu')\) et
\(q'_{i}=0\).

De même, on a une filtration de \(H^{2}(\mu'')\):
\[
0=N_{0}\subset N_{1}\subset\cdots\subset
N_{\ell''}= H^{2}(\mu'')
\]
avec \(\ell''\leq d\) telle que  pour tout \(i\in\{1,2,\cdots,\ell''\}\), on ait
\[N_{i}/N_{i-1}\cong \bigoplus_{j=1}^{q''_{i}}
  L(\nu''_{ij})^{(d''_{ij})}\otimes V(\lambda''_{ij}),\]
avec \(q''_{i}\leq 2^{\ell''-i}\). De plus, 
\(p^{d''_{ij}}\nu''_{ij}\) est \(p^{d+1}\)-restreint et  \(\lambda''_{ij}\) est
\(p^{d''_{ij}}\)-restreint  pour tout \(i,j\). Pour \( i>\ell'' \), posons \(N_{i}=N_{\ell''}=H^{2}(\mu'')\)
et \(q''_{i}=0\).

Posons maintenant \(\ell=\max(\ell',\ell'')+1\leq d+1\). Pour \(0\leq
i\leq \ell-1\), posons
\begin{align*}
V_{i}&=L(0,a)^{(d+1)}\otimes M_{i}\bigoplus
       L(0,a-2)^{(d+1)}\otimes N_{i}\\
  &\subset L(0,a)^{(d+1)}\otimes
    H^{2}(\mu')\bigoplus L(0,a-2)^{(d+1)}\otimes H^{2}(\mu'')\subset H^{2}(\mu).
\end{align*}
Posons aussi \(V_{\ell}=H^{2}(\mu)\).

  Alors pour \( 1\leq i\leq \ell-1 \), on a :
\begin{align*}
  V_{i}/V_{i-1}\cong &L(0,a)^{(d+1)}\otimes \bigoplus_{j=1}^{q'_{i}}
                       L(\nu'_{ij})^{(d'_{ij})}\otimes V(\lambda'_{ij})
  \oplus L(0,a-2)^{(d+1)}\otimes \bigoplus_{j=1}^{q''_{i}}
    L(\nu''_{ij})^{(d''_{ij})}\otimes V(\lambda''_{ij})\\
  \cong &\bigoplus_{j=1}^{q'_{i}}
                       L(\nu'_{ij}+(0,a)p^{d+1-d'_{ij}})^{(d'_{ij})}\otimes
          V(\lambda'_{ij})
  \oplus \bigoplus_{j=1}^{q''_{i}}
    L(\nu''_{ij}+(0,a-2)p^{d+1-d''_{ij}})^{(d''_{ij})}\otimes
    V(\lambda''_{ij}).\\
\end{align*}

Pour \(1\leq i\leq \ell-1\), posons \(q_{i}=q'_{i}+q''_{i}\). Pour
\(1\leq j \leq q'_{i}\), posons
\(\nu_{ij}=\nu'_{ij}+(0,a)p^{d+1-d'_{ij}}\), \(d_{ij}=d'_{ij}\) et
\(\lambda_{ij}=\lambda'_{ij}\). Pour \(q'_{i}<j\leq q_{i}\), posons
\(\nu_{ij}=\nu''_{i,j-q'_{i}}+(0,a-2)p^{d+1-d''_{i,j-q'_{i}}}\),
\(d_{ij}=d''_{i,j-q_{i}}\) et
\(\lambda_{ij}=\lambda''_{i,j-q_{i}}\). Alors l'isomorphisme précédent
se réécrit
\[V_{i}/V_{i-1}\cong \bigoplus_{j=1}^{q_{i}}L(\nu_{ij})^{(d_{ij})}\otimes V(\lambda_{ij}).\]

De plus, on a
\(q_{i}=q'_{i}+q''_{i}\leq 2^{\ell'-i}+2^{\ell''-i}\leq 2\cdot
2^{\max(\ell',\ell'')-i}=2^{\ell-i}\) et \(\lambda_{ij}\) est
\(p^{d_{ij}}\)-restreint par définition. D'après le
\autoref{lemma:poidsrestreints},  \(p^{d_{ij}}\nu_{ij}\) est
\(p^{d+2}\)-restreint puisque \( L(p^{dij}\nu_{ij}+\lambda_{ij}) \)
est un facteur de composition de \( H^{2}(n,-n-2) \), avec \( n=ap^{d+1}+r \).

Enfin, si \(i=\ell\), on a \(V_{i}/V_{i-1}\cong L(0,a-1)^{(d+1)}\otimes
V(r,p^{d+1}-r-2)\).

Donc l'énoncé est  vrai pour \(\mu\). Ceci termine la preuve du \autoref{thm:murfilt}. 

\end{proof}

D'autre part, si \( \mu=(n,-n-2) \) avec \(
n=ap^{d}+r \), où \( 0\leq a\leq p-1 \) et \( 0\leq r\leq p^{d}-1 \),
alors d'après le
\autoref{thm2}, il existe une suite exacte
courte de \( G \)-modules :
\begin{displaymath}
  \begin{tikzcd}
 0\ar[r]&L(0,a)^{(d)}\otimes H^{2}(r,-r-2)\ar[r]& H^{2}(\mu)\ar[r]&W(r,n-2r-2)\ar[r]&0,   
  \end{tikzcd}
\end{displaymath}
où \( W(r,n-2r-2) \) est un quotient du module de Weyl \(
V(r,n-2r-2) \). On a le corollaire suivant :
\begin{cor}\label{cor:autrefiltrationmur}
  Soit \( n=a_{d}p^{d}+a_{d-1}p^{d-1}+\cdots+a_{0} \) avec \( 0\leq
  a_{i}\leq p-1 \). Pour \( k\in\{0,1,\cdots,d\} \), notons \(
  r_{k}=\sum_{i=0}^{k}a_{i}p^{i}  \) (donc \( n=r_{d} \)). Alors \(
  H^{2}(n,-n-2) \) admet une filtration:
  \begin{displaymath}
    0=M_{0}\subset M_{1}\subset\cdots\subset M_{d-1}\subset M_{d}=H^{2}(n,-n-2)
  \end{displaymath}
  telle que
  \begin{displaymath}
    M_{i}/M_{i-1}\cong L(0,n-r_{i})\otimes W(r_{i-1},r_{i}-2r_{i-1}-2)
  \end{displaymath}
  où \( W(r_{i-1},r_{i}-2r_{i-1}-2) \) est un quotient du module de
  Weyl \( V(r_{i-1},r_{i}-2r_{i-1}-2) \).
\end{cor}
\begin{rmk}
  On utilise toujours la convention que \( V(a,b)=0 \) si \( (a,b) \)
  n'est pas dominant. Donc si \( a_{i}=0 \) pour un \(
  i\in\{1,2,\cdots,d\} \), alors \( r_{i-1}=r_{i} \) et \(
  W(r_{i-1},r_{i}-2r_{i-1}-2)=0 \). Donc \( M_{i}=M_{i-1} \) dans ce
  cas. 
\end{rmk}
\begin{proof}
  Raisonnons par récurrence sur \( d \). Si \( d=1 \), alors \( n=ap+r
  \) avec \( 0\leq a,r\leq p-1 \). Avec les notations ci-dessus, on a
  \( r_{0}=r \) et \( r_{1}=n \). Comme \( H^{2}(r,-r-2)\cong
  H^{0}(-1,r)=0 \), d'après le \autoref{thm2}, \( H^{2}(n,-n-2)\cong
  W(r,n-2r-2)\cong L(0,n-r_{1})\otimes W(r_{0},r_{1}-2r_{0}-2) \) où
  \( W(r,n-2r-2)=W(r_{0},r_{1}-2r_{0}-2) \) est un quotient du module
  de Weyl \( V(r_{0},r_{1}-2r_{0}-2) \). Donc l'énoncé est vrai dans
  ce cas.

  Supposons l'énoncé  vrai pour tout \( n \) de degré \( \leq d \) pour un
  \( d\geq 1 \). Soit \( n=a_{d+1}p^{d+1}+a_{d}p^{d}+\cdots+a_{0}
  \). Alors d'après le \autoref{thm2}, on a une suite exacte courte de
  \( G \)-modules:
  \begin{displaymath}
 0\to L(0,a_{d+1})^{(d+1)}\otimes H^{2}(r_{d},-r_{d}-2)\to
 H^{2}(n,-n-2)\to W(r_{d},n-2r_{d}-2)\to0,   
\end{displaymath}
où \( W(r_{d},n-2r_{d}-2) \) est un quotient de \( V(r_{d},n-2r_{d}-2)
\). En appliquant l'hypothèse de récurrence à \(
r_{d}=a_{d}p^{d}+\cdots+a_{0} \), on obtient une filtration:
\begin{displaymath}
    0=M'_{0}\subset M'_{1}\subset\cdots\subset M'_{d-1}\subset M'_{d}=H^{2}(r_{d},-r_{d}-2)
  \end{displaymath}
  telle que pour \( i=1,2,\cdots,d \),
  \begin{displaymath}
    M'_{i}/M'_{i-1}\cong L(0,r_{d}-r_{i})\otimes W(r_{i-1},r_{i}-2r_{i-1}-2)
  \end{displaymath}
  où \( W(r_{i-1},r_{i}-2r_{i-1}-2) \) est un quotient de  \(
  V(r_{i-1},r_{i}-2r_{i-1}-2) \). Posons \(
  M_{i}=L(0,a_{d+1})^{(d+1)}\otimes M'_{i} \) pour \( i=0,1,\cdots,d
  \) et \( M_{d}=H^{2}(n,-n-2) \), alors on obtient une filtration de \( H^{2}(n,-n-2) \)
 \begin{displaymath}
    0=M_{0}\subset M_{1}\subset\cdots\subset M_{d}\subset M_{d+1}=H^{2}(n,-n-2)
  \end{displaymath} 
  telle que
  \begin{align*}
    M_{d+1}/M_{d}&\cong W(r_{d},n-2r_{d}-2)\cong L(0,n-r_{d+1})\otimes
                   W(r_{d},r_{d+1}-2r_{d}-2),\\
    M_{i}/M_{i-1}&\cong L(0,a_{d+1})^{(d+1)}\otimes
                   (M'_{i}/M'_{i-1})\\
                 &\cong L(0,a_{d+1}p^{d+1})\otimes
                   L(0,r_{d}-r_{i})\otimes
                   W(r_{i-1},r_{i}-2r_{i-1}-2)\\
    &\cong L(0,n-r_{i})\otimes  W(r_{i-1},r_{i}-2r_{i-1}-2)
      \qquad\text{si }i\leq d.
  \end{align*}

  Ceci termine la preuve du \autoref{cor:autrefiltrationmur}.
\end{proof}

\section{Une \texorpdfstring{\(p\)}{p}-\texorpdfstring{\(H^{i}\)}{Hi}-D-filtration}\label{chapitre:pHifiltration}

La filtration obtenue dans le chapitre \ref{chapitre:3étages} ne donne pas
d'informations sur la structure de \(H^{1}(\mu)\) et \(H^{2}(\mu)\) si
\(\mu\) n'est pas dans la région de Griffith. Mais  Jantzen a montré (\cite{Jan80}) que pour \(G = \SL_3\), tout 
module de Weyl  \(V(\lambda)\) possède une \(p\)-Weyl-filtration,
c'est-à-dire une filtration dont les quotients sont de la forme 
\(V(\nu^1)^{(1)} \otimes L(\nu^0)\), où \(\nu^0\) est \(p\)-restreint et l'exposant \(^{(1)}\) désigne la torsion par 
le morphisme de Frobenius.

Dualement, pour \(G = \SL_3\) 
tout module induit \(H^0(\lambda)\) possède une
\(p\)-\(H^0\)-filtration. Il est naturel de se demander si
\(H^{1}(\mu)\) et \(H^{2}(\mu)\) possèdent aussi une filtration
analogue. Pour cela, comme dans \cite{Jan80}, on commence par étudier la
structure du \(BG_{1}\)-module \(\widehat{Z}(\mu)=\ind
_{B}^{BG_{1}}(\mu)\).

Tandis que Jantzen utilise une suite de composition arbitraire de
\(\widehat{Z}(\mu)\) pour induire une \(p\)-filtration de \(H^{0}(\mu)\)
(et de  \(H^{3}(\mu)\) par dualité), j'utiliserai la notion de   \og
D-filtration\fg{} (en l'honneur de Donkin, cf. \cite{Don06}) de
\(\widehat{Z}(\mu)\), qui sera définie dans le  paragraphe
\ref{subsection:Zhat}. On va voir que cette filtration non
seulement redonne la  \(p\)-filtration de Jantzen pour
\(H^{0}(\mu)\) et \( H^{3}(\mu) \) (\autoref{thm:weyl}) si
\(\mu\) est dominant ou anti-dominant,
mais donne aussi une filtration analogue pour \(H^{1}(\mu)\) et
\(H^{2}(\mu)\) si \(\mu\notin C\cup w_{0}\cdot C\). 

\subsection{ \texorpdfstring{\og D-filtration\fg{} de
    \(\widehat{Z}(\mu)=\ind_{B}^{BG_{1}}(\mu)\)}{D-filtration de Z(mu)}}\label{subsection:Zhat}
Pour tout \( \mu\in X(T) \), notons
\(\widehat{Z}(\mu)=\ind_{B}^{BG_{1}}(\mu)\).

Dans ce paragraphe, je vais considérer une filtration de \(
\widehat{Z}(\mu) \) qui se comportera bien pour le foncteur \(
H^{0}(G/BG_{1},\bullet) \). Ce n'est pas une suite de composition
comme \( BG_{1} \)-module car certains facteurs font apparaître des \(
B \)-extensions de dimension \( 2 \), tordues par le Frobenius. Ces
extensions apparaissent, au moins au niveau des formules de caractère,
dans l'article \cite{Don06} de Donkin. Pour cette raison, j'appelle
cette filtration de \( \widehat{Z}(\mu) \) la D-filtration.

\begin{rmk}
  Notre \( \widehat{Z}(\mu) \) est noté \( \widehat{Z}'_{1}(\mu) \) dans
  \cite{Jan03} II.9.
\end{rmk}

 Notons \(E_{\alpha}(\mu)\) l'unique sous-\(B\)-module de dimension \(
 2 \) de \( L(0,1)\otimes (\mu+(-1,1)) \). Donc il existe une
suite exacte non scindée de \( B \)-modules:
\[
\xymatrix{0\ar[r]&\mu-\alpha\ar[r]&E_{\alpha}(\mu)\ar[r]&\mu\ar[r]&0}.
\]
De même, notons \(E_{\beta}(\mu)\) l'unique sous-\(B\)-module de dimension \(
 2 \) de \( L(1,0)\otimes (\mu+(1,-1)) \). Donc il existe une
suite exacte non scindée de \( B \)-modules :
\[
\xymatrix{0\ar[r]&\mu-\beta\ar[r]&E_{\beta}(\mu)\ar[r]&\mu\ar[r]&0}.
\]
Posons aussi \(E_{0}(\mu)=\mu\).

On sait que  \(\widehat{Z}(\mu+p\mu')\cong \widehat{Z}(\mu)\otimes p\mu'\)
comme \(BG_{1}\)-modules (cf. \cite{Jan03} II.9.2), donc il suffit de
considérer six cas pour \(\mu\in X_{1}(T)\), cf. la figure et la
définition ci-dessous.

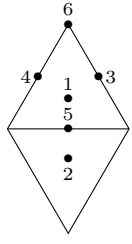
\begin{figure}[H]
  \centering
  \begin{tikzpicture}[scale=0.4]
    \draw (0,0)--(2,4*sin{60})--(0,8*sin{60})--(-2,4*sin{60})--cycle;
    \draw (-2,4*sin{60})--(2,4*sin{60}); \node[font=\tiny] at (0,4*sin{60}+1)
    {\( \bullet \)}; \node[font=\tiny] at (0,4*sin{60}+1.5) {\( 1 \)}; \node[font=\tiny] at
    (0,4*sin{60}-1) {\( \bullet \)}; \node[font=\tiny] at (0,4*sin{60}-1.5)
    {\( 2 \)}; \node[font=\tiny] at (1,6*sin{60}) {\( \bullet \)}; \node[font=\tiny] at
    (1.4,6*sin{60}) {\( 3 \)}; \node[font=\tiny] at (-1,6*sin{60})
    {\( \bullet \)}; \node[font=\tiny] at (-1.4,6*sin{60}) {\( 4 \)}; \node[font=\tiny] at
    (0,4*sin{60}) {\( \bullet \)}; \node[font=\tiny] at (0,4*sin{60}+0.5)
    {\( 5 \)}; \node[font=\tiny] at (0,8*sin{60}){\( \bullet \)}; \node[font=\tiny] at
    (0,8*sin{60}+0.5){\( 6 \)};
  \end{tikzpicture}
 
  \caption{Six cas dans $X_1(T)$}
  \label{figure:sixcas}
\end{figure}

\begin{defi}\label{defi:positions}
    Soit \(\mu=(x,y)\in X(T)\). Écrivons
    \(x=x^{1}p+r\) et
    \(y=y^{1}p+s\) avec
    \(r,s\in\{0,1,\cdots,p-1\}\). On rappelle la terminologie suivante
    (voir par exemple \cite{KH84} 1.1).
    \begin{enumerate}
    \item On dit que \(\mu\) est de type \(\Delta\) si \(r<p-1\),
      \(s<p-1\) et \(r+s>p-2\);
    \item On dit que \(\mu\) est de type \(\nabla\) si
       \(r+s<p-2\);
    \item On dit que \(\mu\) est \(\alpha\)-singulier si \(r=p-1\) et \(s<p-1\);
    \item On dit que \(\mu\) est \(\beta\)-singulier si \(s=p-1\) et \(r<p-1\);
    \item On dit que \(\mu\) est \(\gamma\)-singulier si
      \(r<p-1\), \(s<p-1\) et \(r+s=p-2\);
    \item On dit que \(\mu\) est \(\alpha\)-\(\beta\)-singulier si
      \(r=s=p-1\).
    \end{enumerate}

  \end{defi}

  Pour \( 0\leq r,s\leq p-2 \), on pose \( \overline{r}=p-r-2 \) et \( \overline{s}=p-s-2 \).

\medskip

D'abord, si  \(\mu = (p-1)\rho\) (correspondant au cas  \( 6 \) dans la
\autoref{figure:sixcas}) alors
\(\widehat{Z}(\mu) = L((p-1)\rho)\). Dans ce cas, la D-filtration
est juste la filtration triviale.

\smallskip
Comme \( \hat{Z}(\mu) \) est un \( BG_{1} \)-module de longueur finie,
dont la multiplicité de chaque facteur simple est \( 1 \), la
structure de sous-modules de \( \hat{Z}(\mu) \) peut se décrire par un
graphe, cf. \cite{Irv86} 2.5.

\subsubsection{Cas singulier pour une seule racine} 
\underline{Si \(\mu\) est \(\gamma\)-singulier} (correspondant au cas \( 5 \) dans la \autoref{figure:sixcas})  alors \(\mu = (r,p-2-r)\) avec \(0\leq r \leq p-2\). 
Alors \(s_\alpha\cdot \mu = \mu - (r+1)\alpha = (-r-2,p-1)\) et \(s_\beta\cdot \mu = \mu - (p-1-r)\beta = (p-1, -p + r)\). 
Et \(s_\gamma \cdot \mu = \mu - p\gamma = (-p + r, -r-2)\). 
Alors d'après \cite{Irv86} 3.3, le graphe de \( \widehat{Z}(r,p-2-r) \) comme  \( TG_{1} \)-module, 
 est donné par: 
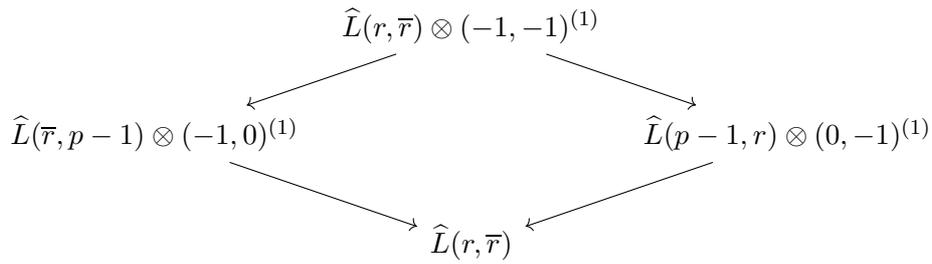
\[
  \begin{tikzcd}[column sep=tiny]
    & \ar[ld] \widehat{L}(r,\overline{r}) \otimes (-1,-1)^{(1)} \ar[rd] &
    \\
    \widehat{L}(\overline{r},p-1) \otimes (-1,0)^{(1)} \ar[rd] & & \ar[ld] \widehat{L}(p-1,r)
    \otimes (0,-1)^{(1)}
    \\
    & \widehat{L}(r,\overline{r}) &
  \end{tikzcd}
\]
De plus, on a
\[
\Ext_{BG_{1}}^{1}\big(\widehat{L}(\overline{r},p-1)\otimes (-1,0)^{(1)},
\widehat{L}(p-1,r)\otimes (0,-1)^{(1)}\big)=0
\]
et
\[
\Ext_{BG_{1}}^{1}\big(\widehat{L}(p-1,r)\otimes (0,-1)^{(1)}, \widehat{L}(\overline{r},p-1)\otimes (-1,0)^{(1)}
\big)=0\]
d'après \cite{Jan03} II.9.21. Donc le graphe ci-dessus est aussi le graphe de \( \widehat{Z}(r,p-2-r) \) comme  \( BG_{1} \)-module.

Dans ce cas, \underline{une D-filtration} est  n'importe
  quelle suite
  de composition de \(\widehat{Z}(\mu)\).
  \medskip
  
\underline{Si \(\mu\) est \(\alpha\)-singulier} (correspondant au cas   \( 3 \) dans la \autoref{figure:sixcas}) alors \(\mu = (p-1, s)\) avec \(0\leq s \leq p-2\). 
On a 
\[
\begin{cases} 
\mu_3 = s_\beta\cdot \mu = \mu - (s+1)\beta = (p+s, -s-2) 
\\
\mu_4 = s_{\alpha,p}\cdot \mu_3 = \mu_3 - (s+1) \alpha = (p-2-s, -1)
\\
\mu_2 = s_\alpha\cdot \mu_4 = \mu_4 - (p-1-s)\alpha = \mu_3 - p\alpha = (-p+s, p-s-2)
\end{cases}
\]

Alors d'après \cite{Irv86} 5.2, le graphe de \(\widehat{Z}(p-1,s)\) comme  \( BG_{1} \)-module
 est donné par: 
\[
\xymatrix{ \widehat{L}(\overline{s},p-1) \otimes (0,-1)^{(1)} \ar[d] 
\\
\widehat{L}(s,\overline{s}) \otimes (1,-1)^{(1)} \ar@{=>}[d]^-{-p\alpha} 
\\
\widehat{L}(s,\overline{s}) \otimes (-1,0)^{(1)} \ar[d] 
\\
\widehat{L}(p-1,s)  
}
\]
où la flèche \(\implies\) indique une extension non scindée de \(
\widehat{L}(s,\overline{s})\otimes (1,-1)^{(1)} \) par \(
\widehat{L}(s,\overline{s})\otimes (-1,0)^{(1)} \). Or on a
\[
\Ext_{BG_{1}}^{1}(\widehat{L}(s,\overline{s})\otimes (1,-1)^{(1)}
,\widehat{L}(s,\overline{s})\otimes (-1,0)^{(1)})\cong k
\]
d'après \cite{Jan03} II.9.21 et on sait qu'il existe une extension non
scindée
\[
  \begin{tikzcd}
    0\ar[r]&\widehat{L}(s,\overline{s})\otimes (-1,0)^{(1)}
    \ar[r]&\widehat{L}(s,\overline{s})\otimes E_{\alpha}(1,-1)^{(1)}\ar[r]
    & \widehat{L}(s,\overline{s})\otimes (1,-1)^{(1)} \ar[r]&0,
  \end{tikzcd}
\]
donc la flèche \(\implies \) indique l'extension non scindée isomorphe
à \( \widehat{L}(s,\overline{s})\otimes E_{\alpha}(1,-1)^{(1)} \).

Dans ce cas, \underline{la D-filtration} est la suivante:
\begin{equation}
  \label{eq:deltafiltrationalpha}
  \begin{aligned}
    0=N_{0}\subset N_{1}&\subset N_{2}\subset N_{3}=\widehat{Z}(p-1,s)\\
    N_{1}&\cong \widehat{L}(p-1,s)\\
    N_{2}/N_{1}&\cong \widehat{L}(s,\overline{s})\otimes E_{\alpha}(1,-1)^{(1)}\\
    N_{3}/N_{2}&\cong \widehat{L}(\overline{s},p-1)\otimes (0,-1)^{(1)}.
  \end{aligned}
\end{equation}

\underline{De même, si \(\mu\) est \(\beta\)-singulier} (correspondant
au cas  \( 4 \) dans
la \autoref{figure:sixcas}) alors \(\mu = (r,p-1)\) avec \(0\leq r \leq p-2\). 
On a 
\[
\begin{cases} 
\mu_3 = s_\alpha\cdot \mu = \mu - (r+1)\alpha = (-r-2, p+r) 
\\
\mu_4 = s_{\beta,p}\cdot \mu_3 = \mu_3 - (r+1) \beta = (-1, p-2-r)
\\
\mu_2 = s_\beta\cdot \mu_4 = \mu_4 - (p-1-r)\beta = \mu_3 - p\beta = (p-2-r,-p+r)
\end{cases}
\]
Alors  le graphe de 
\(\widehat{Z}(r,p-1)\) comme \( BG_{1} \)-module est donné par: 
\[
\xymatrix{ \widehat{L}(p-1,p-2-r) \otimes (-1,0)^{(1)} \ar[d] 
\\
\widehat{L}(p-2-r,r) \otimes (-1,1)^{(1)} \ar@{=>}[d]^-{-p\beta} 
\\
\widehat{L}(p-2-r,r) \otimes (0,-1)^{(1)} \ar[d] 
\\
\widehat{L}(r,p-1)  
}
\]
où la flèche \(\implies\) indique l'extension non scindée isomorphe
à \( \widehat{L}(p-2-r,r)\otimes E_{\beta}(-1,1)^{(1)} \).

Dans ce cas, \underline{la D-filtration} est la suivante:
\begin{equation}
  \label{eq:deltafiltrationbeta}
  \begin{aligned}
    0=N_{0}\subset N_{1}&\subset N_{2}\subset N_{3}=\widehat{Z}(r,p-1)\\
    N_{1}&\cong \widehat{L}(r,p-1)\\
    N_{2}/N_{1}&\cong \widehat{L}(\overline{r},r)\otimes E_{\beta}(-1,1)^{(1)}\\
    N_{3}/N_{2}&\cong \widehat{L}(p-1,\overline{r})\otimes (-1,0)^{(1)}.
  \end{aligned}
\end{equation}

\subsubsection{Cas de l'alcôve supérieure \texorpdfstring{\(\Delta\)}{Delta}}

 Soient  \(r, s\geq 0\) tels que \(r+s \leq p-3\) et soit 
 \(\mu = (\overline{r}, \overline{s})\). 

Alors d'après
\cite{Irv86} 5.3,  le graphe de \(\widehat{Z}(\overline{r},\overline{s})\)
(correspondant au cas \( 1
\) dans la \autoref{figure:sixcas}) comme \( BG_{1} \)-module est donné par: 
\[
\xymatrix{ 
& \widehat{L}(s,r) \otimes (-1,-1)^{(1)} \ar[ld] \ar[rd] & 
\\
\widehat{L}(\overline{s},r+s+1) \otimes (-1,0)^{(1)} \ar[d] \ar[rdd] \ar[rrdd] && \widehat{L}(r+s+1,\overline{r}) \otimes (0,-1)^{(1)}\ar[d] \ar[ldd] \ar[lldd] 
\\
\widehat{L}(r,p-3-r-s) \otimes (-1,1)^{(1)}\quad  \ar@{=>}[d]^-{-p\beta} && \quad\widehat{L}(p-3-r-s,s) \otimes (1,-1)^{(1)}  \ar@{=>}[d]^-{-p\alpha}
\\
\widehat{L}(r,p-3-r-s) \otimes (0,-1)^{(1)} \ar[rd] & \widehat{L}(s,r) \ar[d] &  \widehat{L}(p-3-r-s,s) \otimes (-1,0)^{(1)}  \ar[ld] 
\\
& \widehat{L}(\overline{r},\overline{s}) & 
}
\]
où la flèche \(\implies\) à gauche indique une extension non scindée
de
\( \widehat{L}(r,p-3-r-s) \otimes (-1,1)^{(1)} \) par \(
\widehat{L}(r,p-3-r-s) \otimes (0,-1)^{(1)} \).
 Or d'après \cite{Jan03} II.9.21, il existe une unique telle extension
à isomorphisme près,
donc cette flèche \( \implies \) indique l'extension isomorphe à \(
\widehat{L}(r,p-3-r-s) \otimes E_{\beta}(-1,1)^{(1)} \). De même, la
flèche \( \implies \) à droite indique une extension isomorphe à \( \widehat{L}(p-3-r-s,s) \otimes E_{\alpha}(1,-1)^{(1)} \).

Dans ce cas, \underline{une D-filtration} est une filtration
induite par le graphe suivant:
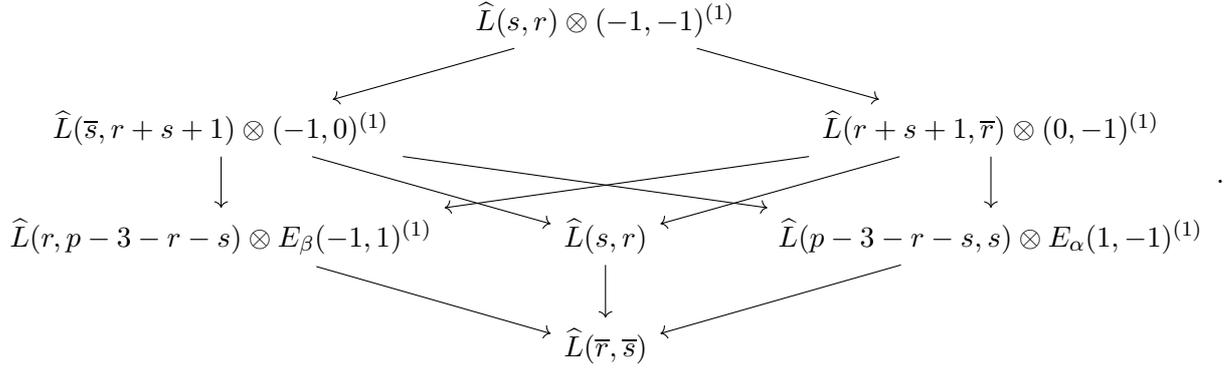
\begin{equation}
  \label{eq:deltagraphe}
  \begin{tikzcd}[column sep=tiny]
    &\widehat{L}(s,r)\otimes (-1,-1)^{(1)}\ar[dl]\ar[dr]&\\
    \widehat{L}(\overline{s},r+s+1) \otimes (-1,0)^{(1)} \ar[d] \ar[rd] \ar[rrd] && \widehat{L}(r+s+1,\overline{r}) \otimes (0,-1)^{(1)}\ar[d] \ar[ld] \ar[lld] 
    \\
    \widehat{L}(r,p-3-r-s)\otimes
    E_{\beta}(-1,1)^{(1)}\ar[rd]&\widehat{L}(s,r)\ar[d]&\widehat{L}(p-3-r-s,s)\otimes
    E_{\alpha}(1,-1)^{(1)}\ar[ld]\\
    &\widehat{L}(\overline{r},\overline{s})&
  \end{tikzcd}.
\end{equation}

Par exemple, la filtration suivante est une D-filtration:
\begin{equation}
  \label{eq:deltafiltrationdelta}
  \begin{aligned}
    0=N_{0}\subset N_{1}\subset &\cdots\subset N_{6}\subset
    N_{7}=\widehat{Z}(\overline{r},\overline{s})\\
    N_{1}&\cong \widehat{L}(\overline{r},\overline{s})\\
    N_{2}/N_{1}&\cong \widehat{L}(s,r)\\
    N_{3}/N_{2}&\cong \widehat{L}(p-3-r-s,s)\otimes E_{\alpha}(1,-1)^{(1)}\\
    N_{4}/N_{3}&\cong \widehat{L}(r,p-3-r-s)\otimes E_{\beta}(-1,1)^{(1)}\\
    N_{5}/N_{4}&\cong \widehat{L}(r+s+1,\overline{r})\otimes (0,-1)^{(1)}\\
    N_{6}/N_{5}&\cong \widehat{L}(\overline{s},r+s+1)\otimes (-1,0)^{(1)}\\
    N_{7}/N_{6}&\cong \widehat{L}(s,r)\otimes (-1,-1)^{(1)}.\\
  \end{aligned}
\end{equation}

\subsubsection{Cas de l'alcôve inférieure \texorpdfstring{\(\nabla\)}{nabla}}

Soit \(\mu = (r,s)\)  avec \(r,s\geq 0\) et \(r+ s \leq p-3\)
(correspondant au cas \( 2 \) dans la \autoref{figure:sixcas}).  

Alors d'après \cite{Irv86} 5.3, 
le graphe de \(\widehat{Z}(r,s)\) comme \( BG_{1} \)-module est donné par: 
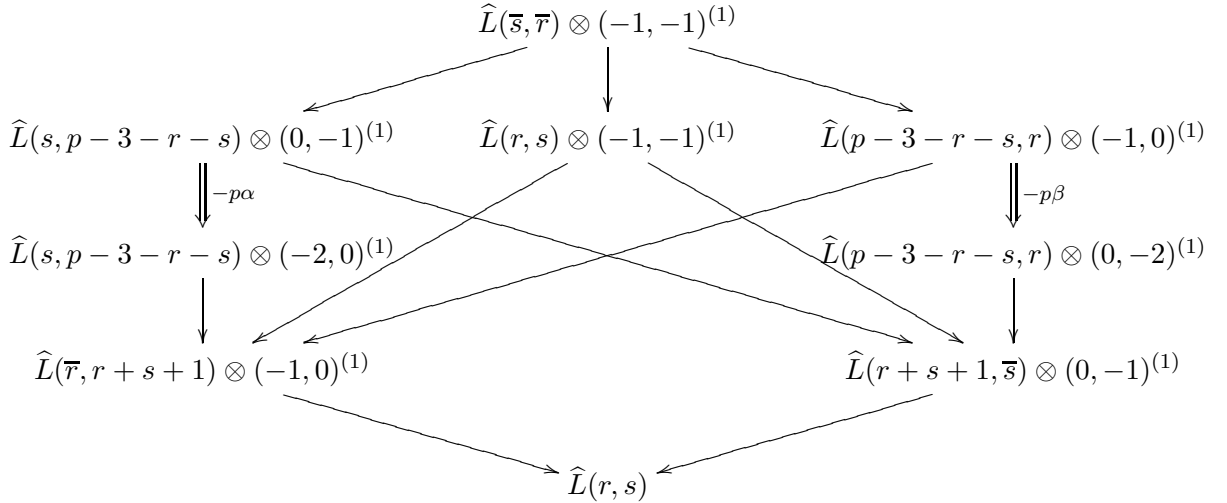
\[
\xymatrix{
& \widehat{L}(\overline{s},\overline{r}) \otimes (-1,-1)^{(1)} \ar[ld] \ar[rd] \ar[d] & 
\\
\widehat{L}(s,p-3-r-s) \otimes (0,-1)^{(1)} \ar@{=>}[d]^-{-p\alpha} \ar[rrdd] & \widehat{L}(r,s) \otimes (-1,-1)^{(1)} \ar[ldd] \ar[rdd] & \widehat{L}(p-3-r-s,r) \otimes (-1,0)^{(1)} \ar@{=>}[d]^-{-p\beta} \ar[lldd] 
\\
\widehat{L}(s,p-3-r-s) \otimes (-2,0)^{(1)} \ar[d] & & \widehat{L}(p-3-r-s,r) \otimes (0,-2)^{(1)} \ar[d]
\\
\widehat{L}(\overline{r},r+s+1) \otimes (-1,0)^{(1)} \ar[rd] & & \widehat{L}(r+s+1, \overline{s}) \otimes (0,-1)^{(1)} \ar[ld] 
\\
& \widehat{L}(r,s) & 
}
\]
où à nouveau la flèche \(\implies\) à gauche indique l'extension non
scindée   \(\widehat{L}(s,p-3-r-s)\otimes
E_\alpha(0,-1)^{(1)}\) et la flèche \( \implies \) à droite indique
l'extension non scindée \(\widehat{L}(p-3-r-s,r)\otimes E_\beta(-1,0)^{(1)}\)
comme dans le cas de l'alcôve \( \Delta \).

Dans ce cas, \underline{une D-filtration} est une filtration
induite par le graphe suivant:
\begin{equation}
  \label{eq:nablagraphe}
  \begin{tikzcd}[column sep=tiny]
 & \widehat{L}(\overline{s},\overline{r}) \otimes (-1,-1)^{(1)} \ar[ld] \ar[rd] \ar[d] & 
\\
\widehat{L}(s,p-3-r-s) \otimes E_{\alpha}(0,-1)^{(1)} \ar[rrd]\ar[d] & \widehat{L}(r,s) \otimes (-1,-1)^{(1)} \ar[ld] \ar[rd] & \widehat{L}(p-3-r-s,r) \otimes E_{\beta}(-1,0)^{(1)} \ar[lld] \ar[d]
\\
\widehat{L}(\overline{r},r+s+1) \otimes (-1,0)^{(1)} \ar[rd] & & \widehat{L}(r+s+1, \overline{s}) \otimes (0,-1)^{(1)} \ar[ld] 
\\
& \widehat{L}(r,s) &    
  \end{tikzcd}.
\end{equation}

Par exemple, la filtration suivante est une D-filtration:
\begin{equation}
  \label{eq:deltafiltrationnabla}
  \begin{aligned}
    0=N_{0}\subset N_{1}\subset &\cdots\subset N_{6}\subset
    N_{7}=\widehat{Z}(r,s)\\
    N_{1}&\cong \widehat{L}(r,s)\\
    N_{2}/N_{1}&\cong \widehat{L}(\overline{r},r+s+1)\otimes (-1,0)^{(1)}\\
    N_{3}/N_{2}&\cong \widehat{L}(r+s+1,\overline{s})\otimes (0,-1)^{(1)}\\
    N_{4}/N_{3}&\cong \widehat{L}(s,p-3-r-s)\otimes E_{\alpha}(0,-1)^{(1)}\\
    N_{5}/N_{4}&\cong \widehat{L}(p-3-r-s,r)\otimes E_{\beta}(-1,0)^{(1)}\\
    N_{6}/N_{5}&\cong \widehat{L}(r,s)\otimes (-1,-1)^{(1)}\\
    N_{7}/N_{6}&\cong \widehat{L}(\overline{s},\overline{r})\otimes (-1,-1)^{(1)}.\\
  \end{aligned}
\end{equation}

\subsection{Sur la cohomologie des \( B \)-modules \( E_{\alpha}(\mu) \)
  et \( E_{\beta}(\mu) \)}
Pour montrer les résultats principaux, il faut d'abord établir
quelques propriétés des modules  \(H^{i}( E_{\alpha}(\mu) )\) et \( H^{i}(E_{\beta}(\mu)) \).
\begin{lemma}\label{lemmaE0}
  On a \(H^{i}(E_{\alpha}(0,y))=H^{i}(E_{\beta}(x,0))=0\) pour tout
  \(i\in\mathbb{N}\) et \( x,y\in\mathbb{Z} \).
\end{lemma}
\begin{proof}
  On a \(E_{\alpha}(0,y)\cong L_{\alpha}(1,0)\otimes (-1,y),\)
  donc d'après l'identité tensorielle (cf. \cite{Jan03} I.3.6) \(H^{i}(P_{\alpha}/B,E_{\alpha}(0,y))\cong
  L_{\alpha}(1,0)\otimes H^{i}(P_{\alpha}/B,(-1,y))=0\) pour tout
  \(i\). Et de  même  pour \(E_{\beta}(x,0)\).
\end{proof}

\begin{proposition}\label{prop:Eparticulier}
 Supposons que \(\mu_{1}=(ap^{d}+r,-ap^{d})\) et
\(\mu_{2}=((a+1)p^{d},-ap^{d}-s-2)\) avec \(d\geq 0\),
\(a\in\{1,2,\cdots, p-1\}\), \(r\geq -1\) et
\(s\leq p^{d}-1\). Alors
\[H^{2}(E_{\beta}(\mu_{1}))=H^{2}(E_{\alpha}(\mu_{2}))=0.\]
\end{proposition}
\begin{proof}
Raisonnons par récurrence sur \(d\).
Lorsque \(d=0\), on a
\(\mu_{1}=(a+r,-a)\) et \(\mu_{2}=(a+1,-a-s-2)\) avec
\(r\geq -1\) et \(s\leq 0\). Donc
\(H^{2}(\mu_{1})=H^{2}(\mu_{1}-\beta)=H^{2}(\mu_{2})=H^{2}(\mu_{2}-\alpha)=0\)
d'après la \autoref{rmk:Griffithcondition}. Par conséquent,
\(H^{2}(E_{\beta}(\mu_{1}))=H^{2}(E_{\alpha}(\mu_{2}))=0\).

Supposons le résultat établi au cran \( d \) et soient
 \(\mu_{1}=(ap^{d+1}+r,-ap^{d+1})\) et
\(\mu_{2}=((a+1)p^{d+1},-ap^{d+1}-s-2)\) avec \(r\geq -1\) et \(s\leq
p^{d+1}-1\).

\smallskip
1) Montrons d'abord que \(E_{\beta}(\mu_{1})=0\).
Notons \(\mu_{1}'=(r,0)\) et 
\(\mu_{1}''=(-p^{d+1}+r,p^{d+1})\). Comme \(\mu''_{1}=(-(p-1)p^{d}-(p^{d}-r-2)-2, p\cdot
p^{d})\), alors,  en échangeant les
rôles de \(\alpha\) et \(\beta\) et en appliquant l'hypothèse de
récurrence à \(\alpha\), on a \(H^{2}(E_{\beta}(\mu_{1}''))=0\).

Rappelons les trois  suites exactes  du paragraphe
\ref{subsection:L(ab)}:
 \begin{equation}\label{def-pfiltKEbeta} 
\xymatrix{0 \ar[r] & K_{a} \ar[r] & L(0,a) \ar[r] & L(0,a-1)\otimes (0,1) \ar[r] & 0, 
}
\end{equation} 

\begin{equation}\label{def-pfiltMEbeta} 
\xymatrix{0 \ar[r] & M_{a} \ar[r] & K_a \ar[r] & (a,-a) \ar[r] & 0,
}
\end{equation} 
et
\begin{equation}\label{def-pfiltQEbeta} 
\xymatrix{ 0\ar[r] & M_a \ar[r] & L(0,a-1) \otimes (-1,0) \ar[r] & L(0,a-2) \otimes (-1,1) \ar[r] & 0 . 
}
\end{equation}

Appliquons la \((d+1)\)-ième puissance du morphisme de 
Frobenius à \eqref{def-pfiltKEbeta},\eqref{def-pfiltMEbeta},
\eqref{def-pfiltQEbeta} et tensorisons par
\(E_{\beta}(r,0)\). Désignons encore les
modules ainsi obtenus  par \(\widetilde{K}_{a}\),\(\widetilde{M}_{a}\) et
\(\widetilde{Q}_{a}\). On obtient les suites exactes:
\begin{equation}
  \label{eq:betapartiK}
\xymatrix{0\ar[r]&\widetilde{K}_{a}\ar[r]&L(0,a)^{(d+1)}\otimes
  E_{\beta}(\mu_{1}')\ar[r]& L(0,a-1)^{(d+1)}\otimes E_{\beta}(r,p^{d+1})\ar[r]&0.}
\end{equation}
\begin{equation}
  \label{eq:betapartiM}
  \xymatrix{0\ar[r]&\widetilde{M}_{a}\ar[r]&\widetilde{K}_{a}\ar[r]&E_{\beta}(\mu_{1})\ar[r]&0.}
\end{equation}
\begin{equation}
  \label{eq:betapartiM2}
 \xymatrix{0\ar[r]&\widetilde{M}_{a}\ar[r]&L(0,a-1)^{(d+1)}\otimes
   E_{\beta}(r-p^{d+1},0)\ar[r]&L(0,a-2)^{(d+1)}\otimes E_{\beta}(\mu_{1}'')\ar[r]&0.}  
\end{equation}

Comme \(H^{2}(E_{\beta}(\mu''_{1}))=0\) d'après l'hypothèse de
récurrence et
\(H^{3}(E_{\beta}(r-p^{d+1},0))=0\) d'après le \autoref{lemmaE0}, alors (\ref{eq:betapartiM2})
donne
\(H^{3}(\widetilde{M}_{a})=0\).

Comme \((r,p^{d+1})\) et \((r,p^{d+1})-\beta=(r+1,p^{d+1}-2)\) sont
dominants donc  n'ont pas de \(H^{1}\),
on a \(H^{1}(E_{\beta}(r,p^{d+1}))=0\). Par ailleurs
\(H^{2}(E_{\beta}(\mu_{1}'))=0\) d'après le \autoref{lemmaE0}, donc (\ref{eq:betapartiK}) donne
\(H^{2}(\widetilde{K}_{a})=0\).

D'après (\ref{eq:betapartiM}), on a une suite exacte
\(\xymatrix{H^{2}(\widetilde{K}_{a})\ar[r]&H^{2}(E_{\beta}(\mu_{1}))\ar[r]&H^{3}(\widetilde{M}_{a})}\),
d'où \(H^{2}(E_{\beta}(\mu_{1}))=0\).

\smallskip
2) Montrons maintenant que \(H^{2}(E_{\alpha}(\mu_{2}))=0\).
Notons \(\mu_{2}'=(p^{d+1},-s-2)\) et \(\mu_{2}''=(0,p^{d+1}-s-2)\).

 Comme  \(\mu_{2}'=(p\cdot p^{d}, -(p-1)p^{d}-(s-p^{d+1}+p^{d})-2)\)
avec \(s-p^{d+1}+p^{d}\leq p^{d}-1\), alors d'après
l'hypothèse de récurrence, on obtient \(H^{2}(E_{\alpha}(\mu_{2}'))=0\).

Appliquons la \((d+1)\)-ième puissance du morphisme de 
Frobenius à (\ref{def-pfiltKEbeta}),(\ref{def-pfiltMEbeta}),
(\ref{def-pfiltQEbeta}) et tensorisons par \(E_{\alpha}(\mu_{2}')\). On obtient les suites exactes suivantes:
\begin{equation}
  \label{eq:alphapartiK}
\xymatrix{0\ar[r]&\widetilde{K}_{a}\ar[r]&L(0,a)^{(d+1)}\otimes
  E_{\alpha}(\mu_{2}')\ar[r]& L(0,a-1)^{(d+1)}\otimes E_{\alpha}(p^{d+1},p^{d+1}-s-2)\ar[r]&0.}
\end{equation}
\begin{equation}
  \label{eq:alphapartiM}
  \xymatrix{0\ar[r]&\widetilde{M}_{a}\ar[r]&\widetilde{K}_{a}\ar[r]&E_{\alpha}(\mu_{2})\ar[r]&0.}
\end{equation}
\begin{equation}
  \label{eq:alphapartiM2}
 \xymatrix{0\ar[r]&\widetilde{M}_{a}\ar[r]&L(0,a-1)^{(d+1)}\otimes
   E_{\alpha}(0,-s-2)\ar[r]&L(0,a-2)^{(d+1)}\otimes E_{\alpha}(\mu_{2}'')\ar[r]&0.}  
\end{equation}

Comme \(H^{2}(E_{\alpha}(\mu''_{2}))=H^{2}(E_{\alpha}(0,p^{d+1}-s-2))=0\) et
\(H^{3}(E_{\alpha}(0,-s-2))=0\) d'après le \autoref{lemmaE0}, on
déduit de  \eqref{eq:alphapartiM2}  que \(H^{3}(\widetilde{M}_{a})=0\).

Comme \((p^{d+1},p^{d+1}-s-2)\) et
\((p^{d+1},p^{d+1}-s-2)-\alpha=(p^{d+1}-2,p^{d+1}-s-1)\) sont
dominants donc n'ont pas de \(H^{1}\),
on a \(H^{1}(E_{\alpha}(p^{d+1},p^{d+1}-s-2))=0\). Par ailleurs
\(H^{2}(E_{\alpha}(\mu_{2}'))=0\), donc d'après (\ref{eq:alphapartiK}) on a
\(H^{2}(\widetilde{K}_{a})=0\).

Enfin, par (\ref{eq:alphapartiM})  on a une suite exacte
\(\xymatrix{H^{2}(\widetilde{K}_{a})\ar[r]&H^{2}(E_{\alpha}(\mu_{2}))\ar[r]&H^{3}(\widetilde{M}_{a})}\),
 ce qui donne \(H^{2}(E_{\alpha}(\mu_{2}))=0\).
Ceci termine la preuve de la \autoref{prop:Eparticulier}.
\end{proof}

On  déduit de la symétrie entre \( \alpha \) et \( \beta \) le corollaire suivant:
\begin{cor}\label{cor:Eparticulier}
  Soient \( a\in\{1,2,\cdots,p-1\}\), \( d\geq 0 \), \( m\geq ap^{d}-1
  \) et \( n\leq ap^{d}-1 \). Alors
  \[
H^{2}(E_{\beta}(-n-2,ap^{d}))=0, \quad H^{2}(E_{\alpha}(-ap^{d},m))=0.
\]
\end{cor}

\begin{proposition}\label{thm:Edelta}
  Soit \(\mu=(ap^{d}+r,-ap^{d}-s-2)\) avec \(a\in \{1,2,\cdots,p-1\}\)
  et \(d\geq 0\).
  \begin{enumerate}[(i)]
  \item Si \(1\leq r\leq p^{d}\) et \(0\leq s\leq p^{d}-1\),
  alors \( H^{2}(E_{\alpha}(\mu)) \) admet la filtration à trois
  étages suivante:
  \begin{equation}
    \label{eq:reccuEalpha}
     H^{2}(E_{\alpha}(\mu))=
     \begin{array}{|c|}
       \hline
     L(0,a-1)^{(d)}\otimes 
       H^{3}(E_{\alpha}(\mu+(-a-1,a)p^{d}))\vphantom{\int_{A}^{B}}\\
       \hline
       L(0,a-2)^{(d)}\otimes H^{2}(E_{\alpha}(\mu+(-a-1,a+1)p^{d}))\vphantom{\int_{A}^{B}}\\
       \hline
        L(0,a)^{(d)}\otimes
       H^{2}(E_{\alpha}(\mu+(-a,a)p^{d}))\vphantom{\int_{A}^{B}}\\
       \hline
     \end{array}.
  \end{equation}
\item Si \(-1\leq r\leq p^{d}-2\) et \(-2\leq s\leq p^{d}-3\),
  alors \( H^{2}(E_{\beta}(\mu) )\) admet la filtration à trois
  étages suivante:
  \begin{equation}\label{reccuEbeta}
    H^{2}(E_{\beta}(\mu))=
     \begin{array}{|c|}
       \hline
     L(0,a-1)^{(d)}\otimes 
       H^{3}(E_{\beta}(\mu+(-a-1,a)p^{d}))\vphantom{\int_{A}^{B}}\\
       \hline
       L(0,a-2)^{(d)}\otimes H^{2}(E_{\beta}(\mu+(-a-1,a+1)p^{d}))\vphantom{\int_{A}^{B}}\\
       \hline
        L(0,a)^{(d)}\otimes
       H^{2}(E_{\beta}(\mu+(-a,a)p^{d}))\vphantom{\int_{A}^{B}}\\
       \hline
     \end{array}.
  \end{equation}
  \end{enumerate}
\end{proposition}
\begin{proof}
\fbox{Montrons d'abord (i).} Écrivons
\(E_{\alpha}(\mu)=E(\mu)\) pour abréger.

Soit \(\mu=(ap^{d}+r,-ap^{d}-s-2)\) avec \(1\leq r\leq p^{d}\) et
\(0\leq s\leq p^{d}-1\).

Notons \(\mu'=(r,-s-2)=\mu\otimes (-a,a)p^{d}\) et
\(\mu''=(-p^{d}+r,p^{d}-s-2)=\mu\otimes (-a-1,a+1)p^{d}\).
Alors \(E(\mu')\cong
E(\mu)\otimes (-a,a)p^{d}\) et \(E(\mu'')\cong E(\mu)\otimes
(-a-1,a+1)p^{d}\).
Appliquons la \(d\)-ième puissance du morphisme de 
Frobenius à (\ref{def-pfiltKEbeta}),(\ref{def-pfiltMEbeta}),
(\ref{def-pfiltQEbeta}) et tensorisons par \(E(\mu')\). Désignons les
modules ainsi obtenus  par \(\widetilde{K}_{a}\), \(\widetilde{M}_{a}\) et
\(\widetilde{Q}_{a}\). On obtient des suites exactes:
\begin{equation}\label{KEalpha} 
\xymatrix{0 \ar[r] & \widetilde{K}_{a} \ar[r] & L(0,a)^{(d)}\otimes
  E(\mu') \ar[r] & L(0,a-1)^{(d)}\otimes E(r,p^{d}-s-2) \ar[r] & 0. 
}
\end{equation} 
\begin{equation}\label{MEalpha} 
\xymatrix{0 \ar[r] & \widetilde{M}_{a} \ar[r] & \widetilde{K}_{a} \ar[r] & E(\mu) \ar[r] & 0
}
\end{equation} 
\begin{equation}\label{MKEalpha} 
\xymatrix{0\ar[r] & \widetilde{M}_a \ar[r] & L(0,a-1)^{(d)} \otimes
   E(r-p^{d},-s-2) \ar[r] & L(0,a-2)^{(d)} \otimes
  E(\mu'') \ar[r] & 0. 
}
\end{equation}

Si \( r\leq p^{d}-1 \), alors \((r-p^{d},-s-2)\) et
\((r-p^{d},-s-2)-\alpha=(r-2-p^{d},-s-1)\) sont dans \( w_{0}\cdot C
\) donc n'ont  de la
cohomologie qu'en degré 3. Si \( r=p^{d} \), alors \( H^{i}(E(r-p^{d},-s-2))=0 \)
pour tout \( i \) d'après le \autoref{lemmaE0}. Donc dans tous les cas, on
a \(H^{i}(E(r-p^{d},-s-2))=0\) si
\(i\neq 3\). De plus, comme \( s\leq p^{d}-1\), alors \(
\mu''=(-p^{d}+r,p^{d}-s-2) \) et \(
\mu''-\alpha=(-p^{d}+r-2,p^{d}-s-1) \) n'ont pas de cohomologie en
degré \( 3 \), donc \( H^{3}(E(\mu''))=0 \).
Donc  d'après \eqref{MKEalpha} on obtient l'isomorphisme
\begin{equation*}
  H^{2}(\widetilde{M}_{a})\cong L(0,a-2)^{(d)}\otimes H^{1}(E(\mu''))
\end{equation*}
et la suite exacte
\begin{equation}
  \label{eq:M3Ealpha}
    0\to L(0,a-2)^{(d)}\otimes
    H^{2}(E(\mu''))\to H^{3}(\widetilde{M}_{a})\to
    L(0,a-1)^{(d)}\otimes
    H^{3}(E(r-p^{d},-s-2))\to 0.
\end{equation}

Comme \((r,p^{d}-s-2)\) et \((r,p^{d}-s-2)-\alpha=(r-2,p^{d}-s-1)\) n'ont de la
cohomologie qu'en degré \(0\) car \( r\geq 1 \) et \( s\leq p^{d}-1
\), on a  \(H^{i}(E(r,p^{d}-s-2))=0\) si \(i\neq 0\). De plus, comme \(
\mu'=(r,-s-2) \) et \( \mu'-\alpha=(r-2,-s) \) n'ont pas de
cohomologie en degré \( 3 \), on a \( H^{3}(E(\mu'))=0 \). Donc d'après
(\ref{KEalpha}) on a \(H^{2}(\widetilde{K}_{a})\cong L(0,a)^{(d)}\otimes
H^{2}(E(\mu'))\) et \(H^{3}(\widetilde{K}_{a})\cong L(0,a)^{(d)}\otimes
H^{3}(E(\mu'))=0\).
 D'après (\ref{MEalpha}), on a
\begin{equation}\hskip-9mm
  \label{eq:MKEalphalong}
\xymatrix{H^{2}(\widetilde{M}_{a})\ar[r]^{f}&H^{2}(\widetilde{K}_{a})\ar[r]&H^{2}(E(\mu))\ar[r]&H^{3}(\widetilde{M}_{a})\ar[r]&0}
\end{equation}
car \( H^{3}(\widetilde{K}_{a})=0 \).

Par ailleurs, si \( r=p^{d} \), alors \(
H^{i}(E(\mu''))=H^{i}(E(0,p^{d}-s-2))=0 \) pour tout \( i \) d'après
le \autoref{lemmaE0}. Si \( r\leq p^{d}-1 \), alors
  \begin{align*}
\KF(H^{1}(E(\mu'')))&\subset \KF(H^{1}(\mu''))\cup
    \KF(H^{1}(\mu''-\alpha))\\
   & =\KF(H^{1}(-p^{d}+r,p^{d}-s-2))\cup \KF(H^{1}(-p^{d}+r-2,p^{d}-s-1)),
  \end{align*}
donc tout plus haut poids d'un facteur de
composition de \(H^{1}(E(\mu''))\) est \(p^{d}\)-restreint d'après le
\autoref{lemma:poidsrestreints}. De même,
\begin{align*}
\KF(H^{2}(E(\mu')))&\subset \KF(H^{2}(\mu''))\cup
    \KF(H^{2}(\mu'-\alpha))\\
   & =\KF(H^{2}(r,-s-2))\cup \KF(H^{2}(r-2,-s-1)),
\end{align*}
 donc tout plus haut poids d'un facteur de composition de \(H^{2}(E(\mu'))\)
est \(p^{d}\)-restreint d'après le \autoref{lemma:poidsrestreints} (en
fait, d'après la preuve du \autoref{lemma:poidsrestreints}, on peut
voir que tout plus haut poids de \( H^{i}(p^{d},-s-2) \) est aussi \(
p^{d} \)-restreint si \( s\geq -1 \)). Donc
\[\KF(H^{2}(\widetilde{M}_{a}))\cap \KF(H^{2}(\widetilde{K}_{a}))=\varnothing\]
car \(H^{2}(\widetilde{M}_{a})\cong L(0,a-2)^{(d)}\otimes
H^{1}(E(\mu''))\) et \(H^{2}(\widetilde{K}_{a})\cong L(0,a)^{(d)}\otimes
H^{2}(E(\mu'))\). Donc \(f=0\) dans (\ref{eq:MKEalphalong}).

En conclusion, si \( \mu=(ap^{d}+r,-ap^{d}-s-2) \) avec
\(1\leq r\leq p^{d}\) et \( 0\leq s\leq p^{d}-1 \), 
alors on a une filtration à trois étages de \(H^{2}(E(\mu))\), donnée
par \eqref{eq:M3Ealpha} et par:
\begin{equation}
  \label{eq:filtEalpha1}
  \xymatrix{0\ar[r]&L(0,a)^{(d)}\otimes H^{2}(E(\mu'))\ar[r]&H^{2}(E(\mu))\ar[r]&H^{3}(\widetilde{M}_{a})\ar[r]&0.}
\end{equation}
Ceci preuve (i).

\fbox{Montrons maintenant (ii).} Écrivons \(E_{\beta}(\mu)=E(\mu)\) pour abréger.
Supposons que \(\mu=(ap^{d}+r,-ap^{d}-s-2)\) avec \(-1\leq r\leq p^{d}-2\) et
\(-1\leq s\leq p^{d}-3\) (On traitera le cas \( s=-2 \) à la fin).

Notons \(\mu'=(r,-s-2)=\mu\otimes (-a,a)p^{d}\) et
\(\mu''=(-p^{d}+r,p^{d}-s-2)=\mu\otimes (-a-1,a+1)p^{d}\).
Alors \(E(\mu')\cong
E(\mu)\otimes (-a,a)p^{d}\) et \(E(\mu'')\cong E(\mu)\otimes
(-a-1,a+1)p^{d}\).
Appliquons la \(d\)-ième puissance du morphisme de 
Frobenius à (\ref{def-pfiltKEbeta}),(\ref{def-pfiltMEbeta}),
(\ref{def-pfiltQEbeta}) et tensorisons par \(E(\mu')\). Désignons les
modules ainsi obtenus par \(\widetilde{K}_{a}\), \(\widetilde{M}_{a}\) et
\(\widetilde{Q}_{a}\). On obtient les suites exactes
\begin{equation}\label{KEbeta} 
\xymatrix{0 \ar[r] & \widetilde{K}_{a} \ar[r] & L(0,a)^{(d)}\otimes
  E(\mu') \ar[r] & L(0,a-1)^{(d)}\otimes E(r,p^{d}-s-2) \ar[r] & 0. 
}
\end{equation} 

\begin{equation}\label{MEbeta} 
\xymatrix{0 \ar[r] & \widetilde{M}_{a} \ar[r] & \widetilde{K}_{a} \ar[r] & E(\mu) \ar[r] & 0
}
\end{equation} 

\begin{equation}\label{QEbeta} 
\xymatrix{0 \ar[r] & E(-ap^{d}+r,-s-2) \ar[r] & \widetilde{M}_{a} \ar[r] & \widetilde{Q}_{a} \ar[r] & 0. 
}
\end{equation}
\begin{equation}\label{MKEbeta} 
\xymatrix{0\ar[r] & \widetilde{M}_a \ar[r] & L(0,a-1)^{(d)} \otimes
   E(r-p^{d},-s-2) \ar[r] & L(0,a-2)^{(d)} \otimes
  E(\mu'') \ar[r] & 0. 
}
\end{equation}

Comme \((r-p^{d},-s-2)\) et \((r-p^{d},-s-2)-\beta=(r+1-p^{d},-s-4)\)
sont dans \( w_{0}\cdot C \) donc n'ont  de la
cohomologie qu'en degré \(3\), on a \(H^{i}(E(r-p^{d},-s-2))=0\) si
\(i\neq 3\) et la suite exacte:
\begin{equation}
  \label{eq:Erps2}
  \xymatrix{0\ar[r]&V(s+2,p^{d}-r-3)\ar[r]&H^{3}(E(r-p^{d},-s-2))\ar[r]&V(s,p^{d}-r-2)\ar[r]&0.}
\end{equation}
Donc par (\ref{MKEbeta}) on obtient
\begin{equation*}
  H^{2}(\widetilde{M}_{a})\cong L(0,a-2)^{(d)}\otimes H^{1}(E(\mu''))
\end{equation*}
et la suite exacte
\begin{equation}\hskip-9mm
  \label{eq:M3Ebeta}
  \begin{tikzcd}
    0\ar[r]&L(0,a-2)^{(d)}\otimes
    H^{2}(E(\mu''))\ar[r]&H^{3}(\widetilde{M}_{a})\ar[overlay,out=-10,in=170]{dl}\\
   & L(0,a-1)^{(d)}\otimes H^{3}(E(r-p^{d},-s-2))\ar[r]&0.
  \end{tikzcd}
\end{equation}

Comme \((r,p^{d}-s-2)\) et \((r,p^{d}-s-2)-\beta=(r+1,p^{d}-s-4)\) n'ont de la
cohomologie qu'en degré 0 car \( r\geq -1 \) et \( s\leq p^{d}-3 \), alors \(H^{i}(E(r,p^{d}-s-2))=0\) si \(i\neq 0\). Donc par
(\ref{KEbeta}) on a \(H^{2}(\widetilde{K}_{a})\cong L(0,a)^{(d)}\otimes
H^{2}(E(\mu'))\) et \(H^{3}(\widetilde{K}_{a})\cong L(0,a)^{(d)}\otimes
H^{3}(E(\mu'))=0\) car \(\mu'\) et \(\mu'-\beta\) n'ont pas de
cohomologie en degré \( 3 \).

Comme \(\mu\) et \(\mu-\beta\) n'ont pas de \(H^{0}\) ni de
\(H^{3}\), on a
\(H^{0}(E(\mu))=H^{3}(E(\mu))=0\). Donc par (\ref{MEbeta}), on a une
suite exacte:
\begin{equation}
  \label{eq:MKEbatalong}
\xymatrix{H^{2}(\widetilde{M}_{a})\ar[r]^{f}&H^{2}(\widetilde{K}_{a})\ar[r]&H^{2}(E(\mu))\ar[r]&H^{3}(\widetilde{M}_{a})\ar[r]&0.}
\end{equation}

Par ailleurs, \(\KF(H^{1}(E(\mu'')))\subset \KF(H^{1}(\mu''))\cup
\KF(H^{1}(\mu''-\beta))\), donc tout plus haut poids d'un facteur de
composition de \(H^{1}(E(\mu''))\) est \(p^{d}\)-restreint d'après le \autoref{lemma:poidsrestreints}. De même,
tout plus haut poids d'un facteur de composition de \(H^{2}(E(\mu'))\)
est \(p^{d}\)-restreint. Donc
\[\KF(H^{2}(\widetilde{M}_{a}))\cap \KF(H^{2}(\widetilde{K}_{a}))=\varnothing\]
car \(H^{2}(\widetilde{M}_{a})\cong L(0,a-2)^{(d)}\otimes
H^{1}(E(\mu''))\) et \(H^{2}(\widetilde{K}_{a})\cong L(0,a)^{(d)}\otimes
H^{2}(E(\mu'))\). Donc \(f=0\) dans (\ref{eq:MKEbatalong}).

En conclusion, si \( \mu=(ap^{d}+r,-ap^{d}-s-2) \) avec
\(-1\leq r\leq p^{d}-2\) et \( -1\leq s\leq p^{d}-3 \),  alors on a
une filtration à trois étages pour \(H^{2}(E(\mu))\) donnée par
\eqref{eq:M3Ebeta} et par:
\begin{equation}
  \label{eq:filtEbeta1}
  \xymatrix{0\ar[r]&L(0,a)^{(d)}\otimes H^{2}(E(\mu'))\ar[r]&H^{2}(E(\mu))\ar[r]&H^{3}(\widetilde{M}_{a})\ar[r]&0}.
\end{equation}

Cette filtration implique \eqref{reccuEbeta} pour \( -1\leq r\leq
p^{d}-2 \) et \( -1\leq s\leq p^{d}-3 \).

Il reste à montrer  \eqref{reccuEbeta} pour \( s=-2 \) et \( -1\leq r\leq
p^{d}-2 \). Dans ce cas, \( \mu=(ap^{d}+r,-ap^{d}) \), donc d'après la
\autoref{prop:Eparticulier}, on a \( H^{2}(E_{\beta}(\mu))=0 \).
Comme \( \mu+(-a-1,a)p^{d}=(r-p^{d},0) \) et \( \mu+(-a,a)p^{d}=(r,0)
\), on a \[
  H^{3}(E_{\beta}(\mu+(-a-1,a)p^{d}))=H^{2}(E_{\beta}(\mu+(-a,a)p^{d}))=0 \]
d'après le \autoref{lemmaE0}. 
En outre, posons \[ \mu''=(\mu+(-a-1,a+1)p^{d})=(r-p^{d},p^{d}), \]
alors \( H^{2}(E_{\beta}(\mu''))=0 \) d'après
le \autoref{cor:Eparticulier} car \( r\geq -1 \). Donc les deux
membres de \eqref{reccuEbeta} sont nuls.  Ceci termine la preuve de
(ii) et donc de la \autoref{thm:Edelta}.

\end{proof}

\subsection{La \texorpdfstring{\(p\)}{p}-filtration de Jantzen} \label{subsection:Jantzen}

Tandis que Jantzen utilise une suite de composition arbitraire de
\(\widehat{Z}(\mu)\) pour induire une \(p\)-filtration de \(H^{0}(\mu)\)
(et de \(H^{3}(w_{0}\cdot \mu)\) par dualité) pour \( \mu \) dominant,
je vais utiliser une
D-filtration de \(\widehat{Z}(\mu)\). 

\begin{lemma}\label{lemma:filtrationinduite}
  Soient \( G \) un schéma en groupes réductif déployé sur un corps \( k \) et
  \( H \) un sous-groupe fermé. Soit \( N \) un \( H
  \)-module qui admet une filtration : 
  \(0=N_{0}\subset N_{1}\subset\cdots\subset N_{\ell}=N.\)
  Posons \( L_{i}=N_{i}/N_{i-1} \) pour \( i\in\{1,2,\cdots,\ell\}
  \). Si pour un \( n\in\mathbb{N} \) on a 
    \( \ch R^{n}\ind_{H}^{G}(N)=\sum_{i=1}^{\ell}\ch R^{n}\ind_{H}^{G}(L_{i}) \),
  alors pour \( i=1,2,\cdots,\ell\), \( R^{n}\ind_{H}^{G}(N_{i-1})\)
  est un sous-module de \( R^{n}\ind_{H}^{G}(N_{i}) \) et l'on a :
  \[R^{n}\ind_{H}^{G}(N_{i})/R^{n}\ind_{H}^{G}(N_{i-1})\cong
    R^{n}\ind_{H}^{G}(L_{i}).\]
\end{lemma}

La preuve est standard et laissée au lecteur.

\begin{proposition}\label{thm:weyl}
  Soit \(\lambda=(x,y)\) un poids tel que \(x,y \geq -1\).
   D'après le paragraphe \ref{subsection:Zhat}, il existe une D-filtration
\( 0=N_{0}\subset N_{1}\subset N_{2}\subset\cdots \subset
N_{\ell}=\widehat{Z}(w_{0}\cdot \lambda)=\widehat{Z}(-y-2,-x-2) \)
 telle que \(N_{i}/N_{i-1}\cong
\widehat{L}(\nu_{i}^{0})\otimes E_{\delta_{i}}(\nu_{i}^{1})^{(1)}\) où
\(\delta_{i}\in \{0,\alpha,\beta\}\) (donc \(\ell=1,3,4,\) ou \(7\)).
 
Alors il existe une
filtration 
\( 0=\widetilde{N_{0}}\subset \widetilde{N_{1}}\subset\cdots \subset
\widetilde{N_{\ell}}=V(\lambda) \)
 de \(H^{3}(-y-2,-x-2)\cong V(\lambda)\) telle que
\( \widetilde{N_{i}}/\widetilde{N_{i-1}}\cong L(\nu_{i}^{0})\otimes H^{3}(E_{\delta_{i}}(\nu_{i}^{1}))^{(1)} \)
pour tout \(i\in\{1,2,\cdots,\ell\}\).

De plus, pour tout \(i\in\{1,2,\cdots,\ell\}\) et tout \( j\neq 3 \), on a
\(H^{j}(E_{\delta_{i}}(\nu_{i}^{1}))=0\). 
\end{proposition}
\begin{proof}
Posons
\(\widetilde{N_{i}}=H^{3}(G/BG_{1},N_{i})\).  D'après le \autoref{lemma:filtrationinduite}, il suffit de montrer
l'égalité suivante:
   \begin{equation}
    \label{eq:pfiltWeyl}
\ch H^{3}(-y-2,-x-2)=\sum_{i=1}^{\ell}\ch L(\nu_{i}^{0})\ch H^{3}(E_{\delta_{i}}(\nu_{i}^{1}))^{(1)}.
\end{equation}

  La caractéristique d'Euler-Poincaré \( \chi(\cdot)=\sum_{i\geq
    0}(-1)^{i}\ch H^{i}(\cdot) \) est additive,  donc 
  \[
 \chi(-y-2,-x-2)=\sum_{i=1}^{\ell}\ch L(\nu_{i}^{0})\chi(E_{\delta_{i}}(\nu_{i}^{1}))^{(1)}.
\]

 Comme \(x,y\geq -1\), on a \(\chi(-y-2,-x-2)=-\ch
 H^{3}(-y-2,-x-2)\). Si \(-y-2=ap+r\) et \(-x-2=bp+s\) avec \(r,s\in
 \{0,1,2,\cdots,p-1\}\), alors \(a,b\leq -1\). D'après le paragraphe
 \ref{subsection:Zhat}, les
 \(E_{\delta_{i}}(\nu_{i}^{1})\) possibles sont:
 \begin{itemize}
 \item
   \((a,b), (a-1,b),(a,b-1),(a-1,b-1)\)
\item \(E_{\alpha}(a+1,b-1),E_{\alpha}(a,b-1), E_{\beta}(a-1,b+1),E_{\beta}(a-1,b)\).
\end{itemize}
Tout poids de la première ligne n'a de la cohomologie qu'en degré \(3\). Pour la
deuxième ligne: \(E_{\alpha}(a,b-1)\) et \(E_{\beta}(a-1,b)\) n'ont
de la cohomologie qu'en degré \(3\). Et \(E_{\alpha}(a+1,b-1)\) n'a de
la cohomologie qu'en degré \(3\) si
\(a\leq -2\); si \(a=-1\), \((a+1,b-1)=(0,b-1)\), donc
\(H^{2}(E_{\alpha}(a+1,b-1))=0\) par le \autoref{lemmaE0}. Donc
\(E_{\alpha}(a+1,b-1)\) n'a de la cohomologie qu'en degré \(3\)
. De même pour \(E_{\beta}(a-1,b+1)\).

Donc on a toujours \( H^{j}(E_{\delta_{i}}(\nu_{i}^{1}))=0 \) si \( j\neq
3 \). Par conséquent, pour tout \( i\in\{1,2,\cdots,\ell\} \)  on a
pour \( j\neq 3 \)
\[H^{j}(G/BG_{1},N_{i}/N_{i-1})\cong L(\nu_{i}^{0})\otimes
  H^{j}(E_{\delta_{i}}(\nu_{i}^{1}))^{(1)}=0\]
 (cf. \cite{Jan03} II.9.13)  et  
\(\chi(E_{\delta_{i}}(\nu_{i}^{1}))=-\ch
H^{3}(E_{\delta_{i}}(\nu_{i}^{1}))\), d'où l'égalité
\eqref{eq:pfiltWeyl}.

\end{proof}

En utilisant la dualité de Serre contravariante, on obtient la
proposition suivante. En fait, on peut
aussi la montrer directement par une preuve analogue.
\begin{proposition}\label{thm:H0pfilt}
  Soit \(\lambda=(x,y)\) un poids tel que \(x,y \geq -1\).
   D'après le paragraphe \ref{subsection:Zhat}, il existe une D-filtration
\( 0=N_{0}\subset N_{1}\subset N_{2}\subset\cdots \subset
N_{\ell}=\widehat{Z}(\lambda) \)
 telle que \(N_{i}/N_{i-1}\cong
\widehat{L}(\nu_{i}^{0})\otimes E_{\delta_{i}}(\nu_{i}^{1})^{(1)}\) où
\(\delta_{i}\in \{0,\alpha,\beta\}\) (donc \(\ell=1,3,4,\) ou \(7\)). 

Alors il existe une
filtration \( 0=\widetilde{N_{0}}\subset \widetilde{N_{1}}\subset\cdots \subset
\widetilde{N_{\ell}}=H^{0}(\lambda) \) de \(H^{0}(\lambda)\)
telle que
\( \widetilde{N_{i}}/\widetilde{N_{i-1}}\cong L(\nu_{i}^{0})\otimes H^{0}(E_{\delta_{i}}(\nu_{i}^{1}))^{(1)} \)
pour tout \(i\in\{1,2,\cdots,\ell\}\).

De plus, pour tout \(i\in\{1,2,\cdots,\ell\}\) et tout \( j\neq 0 \), on a \(H^{j}(E_{\delta_{i}}(\nu_{i}^{1}))=0\).
\end{proposition}

Avec cette filtration, on peut redémontrer l'existence d'une \( p
\)-Weyl-filtration pour tout \( \lambda\in X(T)^{+} \) si \( G=\SL_{3}
\) (cf. \cite{Jan80} 3.13).

Plus précisément, 
supposons \( \lambda=(a,b)\in X(T)^{+} \) et écrivons \(
a=a^{1}p+r \), \( b=b^{1}p+s \) avec \( 0\leq r,s\leq p-1 \). Pour \( \mu=p\mu^{1}+\mu^{0} \), notons \(
\nabla_{p}(\mu)=L(\mu^{0})\otimes H^{0}(\mu^{1})^{(1)} \).  Distinguons les cas suivants.

1) Si \( \lambda \) est de type \( \Delta \), alors les plus hauts poids
des facteurs de composition de \( \widehat{Z}(\lambda) \) sont donnés
par la figure suivante, où \( \lambda_{1}=\lambda \) et les triangles
équilatéraux sont des \( p \)-alcôves:
\begin{figure}[H]
    \centering
    \begin{tikzpicture}[scale=0.7]
      \draw (0,0)--(3,0)--(1.5,-3*sin{60})--cycle; \draw
      (0.5,-sin{60})--(1.5,sin{60})--(2.5,-sin{60})--cycle; \draw
      (1,0)--(2.5,-3*sin{60})--(3,-2*sin{60})--(0,-2*sin{60})--(0.5,-3*sin{60})--(2,0);
      \node[font=\tiny] at (1.5,0.3){\( \lambda_{1} \)}; \node[font=\tiny] at (0.5,-0.3){\( \lambda_{6} \)}; \node[font=\tiny] at
      (1.5,-0.3){\( \lambda_{2} \)}; \node[font=\tiny] at (2.5,-0.3){\( \lambda_{4} \)}; \node[font=\tiny] at
      (1,-sin{60}+0.3){\( \lambda_{7} \)}; \node[font=\tiny] at (2,-sin{60}+0.3){\( \lambda_{8} \)}; \node[font=\tiny] at
      (1,-sin{60}-0.3){\( \lambda_{3} \)}; \node[font=\tiny] at (2,-sin{60}-0.3){\( \lambda_{5} \)}; \node[font=\tiny] at
      (1.5,-2*sin{60}-0.3){\( \lambda_{9} \)};
    \end{tikzpicture}.  
    \caption{type $\Delta$}
    \label{fig:19c27edb97a6ba80}
  \end{figure}
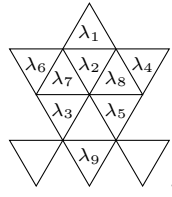

  \begin{rmk}
Si \( \widehat{L}(\lambda') \) est un facteur de composition de \(
\widehat{Z}(\lambda) \), alors \( \lambda'\in W_{p}\cdot \lambda \), donc
il suffit d'indiquer la \( p \)-alcôve contenant \( \lambda' \).
\end{rmk}

 Écrivons \( \lambda_{i}=p\lambda_{i}^{1}+\lambda_{i}^{0} \) avec \(
\lambda_{i}^{0}\in X_{1}(T) \). On sait que \(
\lambda_{5}^{0}=\lambda_{6}^{0} \) et \( \widehat{L}(\lambda_{5}) \) et \(
\widehat{L}(\lambda_{6}) \) forment le facteur \(
\widehat{L}(\lambda_{6}^{0})\otimes E_{\beta}(\lambda_{6}^{1})^{(1)}
\). De même, \( \lambda_{3}^{0}=\lambda_{4}^{0} \) et \( \widehat{L}(\lambda_{3}) \) et \(
\widehat{L}(\lambda_{4}) \) forment le facteur \(
\widehat{L}(\lambda_{4}^{0})\otimes E_{\beta}(\lambda_{4}^{1})^{(1)}
\). Appliquons le foncteur \( \ind_{BG_{1}}^{G}(\bullet) \) aux suites
exactes suivantes:
\begin{displaymath}
  \begin{tikzcd}
    0\ar[r]& \widehat{L}(\lambda_{3})\ar[r]&
    \widehat{L}(\lambda_{4}^{0})\otimes
    E_{\alpha}(\lambda_{4}^{1})^{(1)}\ar[r]&
    \widehat{L}(\lambda_{4})\ar[r]&0,
  \end{tikzcd}
\end{displaymath}

\begin{displaymath}
  \begin{tikzcd}
     0\ar[r]& \widehat{L}(\lambda_{5})\ar[r]&
    \widehat{L}(\lambda_{6}^{0})\otimes
    E_{\beta}(\lambda_{6}^{1})^{(1)}\ar[r]&
    \widehat{L}(\lambda_{6})\ar[r]&0.
  \end{tikzcd}
\end{displaymath}
On obtient
\begin{equation}
  \label{eq:bdd86dd833a6c4fb}
    0\to \nabla_{p}(\lambda_{3})\to L(\lambda_{4}^{0})\otimes
    H^{0}(E_{\alpha}(\lambda_{4}^{1}))^{(1)}\to\nabla_{p}(\lambda_{4})\xrightarrow{\partial_{\alpha}}
    L(\lambda_{3}^{0})\otimes
  H^{1}(\lambda_{3}^{1})^{(1)}\to\cdots,
\end{equation}
et
\begin{equation}
  \label{eq:b544b61499768f3d}
    0\to \nabla_{p}(\lambda_{5})\to L(\lambda_{6}^{0})\otimes
    H^{0}(E_{\beta}(\lambda_{6}^{1}))^{(1)}\to
    \nabla_{p}(\lambda_{6})\xrightarrow{\partial_{\beta}} L(\lambda_{5}^{0})\otimes
  H^{1}(\lambda_{5}^{1})^{(1)}\to\cdots.
\end{equation}

Mais on a \( \lambda_{3}^{1}=(a^{1}-1,b^{1}) \) et \(
\lambda_{5}^{1}=(a^{1},b^{1}-1) \), donc \(
\lambda_{3}^{1},\lambda_{5}^{1} \in C\). Par conséquent, on a \(
H^{1}(\lambda_{3}^{1})=H^{1}(\lambda_{5}^{1})=0 \), d'où \(
\partial_{\alpha}=\partial_{\beta}=0 \). C'est-à-dire, \( L(\lambda_{3}^{0})\otimes
    H^{0}(E_{\alpha}(\lambda_{4}^{1}))^{(1)}  \) est juste une
    extension de \( L(\lambda_{4}^{0})\otimes
    H^{0}(\lambda_{4}^{1})^{(1)} \) par
    \( L(\lambda_{3}^{0})\otimes H^{0}(\lambda_{3}^{1})^{(1)} \), et
    \( L(\lambda_{5}^{0})\otimes
    H^{0}(E_{\beta}(\lambda_{6}^{1}))^{(1)}  \) est juste une
    extension de \( L(\lambda_{6}^{0})\otimes
    H^{0}(\lambda_{6}^{1})^{(1)} \) par
    \( L(\lambda_{5}^{0})\otimes H^{0}(\lambda_{5}^{1})^{(1)} \).

Donc d'après la \autoref{thm:H0pfilt},   il existe dans ce cas
une filtration de \( H^{0}(\lambda) \) dont les quotients sont les  \(
L(\nu_{i}^{0})\otimes H^{0}(\lambda_{i}^{1})^{(1)} \) pour \(
i\in\{1,2,\cdots,9\} \). Certains d'entre eux peuvent être nuls si
l'alcôve en question n'est pas dans \( C \), mais à part cela il n'y a
pas d'effacement.

\smallskip
2) Si \( \lambda \) est de type \( \nabla \), alors les plus hauts poids
des facteurs de composition de \( \widehat{Z}(\lambda) \) sont donnés
par la figure suivante, où \( \lambda_{1}=\lambda \):
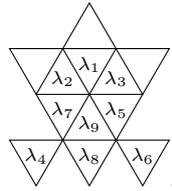
\begin{figure}[H]
  \centering
  \begin{tikzpicture}[scale=0.7]
    \draw (0,0)--(3,0)--(1.5,-3*sin{60})--cycle; \draw
    (0.5,-sin{60})--(1.5,sin{60})--(2.5,-sin{60})--cycle; \draw
    (1,0)--(2.5,-3*sin{60})--(3,-2*sin{60})--(0,-2*sin{60})--(0.5,-3*sin{60})--(2,0);
    \node[font=\tiny] at (1.5,-0.3){\( \lambda_{1} \)}; \node[font=\tiny] at (1,-sin{60}+0.3){\( \lambda_{2} \)}; \node[font=\tiny] at
    (2,-sin{60}+0.3){\( \lambda_{3} \)}; \node[font=\tiny] at (1,-sin{60}-0.3){\( \lambda_{7} \)}; \node[font=\tiny] at
    (2,-sin{60}-0.3){\( \lambda_{5} \)}; \node[font=\tiny] at (1.5,-2*sin{60}+0.3){\( \lambda_{9} \)}; \node[font=\tiny] at
    (1.5,-2*sin{60}-0.3){\( \lambda_{8} \)}; \node[font=\tiny] at (0.5,-2*sin{60}-0.3){\( \lambda_{4} \)}; \node[font=\tiny]
    at (2.5,-2*sin{60}-0.3){\( \lambda_{6} \)};
  \end{tikzpicture}.
  
  \caption{type $\nabla$}
  \label{fig:66d09edc930bac8a}
\end{figure}

Écrivons \( \lambda_{i}=p\lambda_{i}^{1}+\lambda_{i}^{0} \) avec \(
\lambda_{i}^{0}\in X_{1}(T) \). On sait que \(
\lambda_{6}^{0}=\lambda_{7}^{0} \) et \( \widehat{L}(\lambda_{6}) \) et \(
\widehat{L}(\lambda_{7}) \) forment le facteur \(
\widehat{L}(\lambda_{7}^{0})\otimes E_{\beta}(\lambda_{7}^{1})^{(1)}
\). De même, \( \lambda_{4}^{0}=\lambda_{5}^{0} \) et \( \widehat{L}(\lambda_{4}) \) et \(
\widehat{L}(\lambda_{5}) \) forment le facteur \(
\widehat{L}(\lambda_{5}^{0})\otimes E_{\beta}(\lambda_{5}^{1})^{(1)}
\). Appliquons le foncteur \( \ind_{BG_{1}}^{G}(\bullet) \) aux suites
exactes suivantes:
\begin{displaymath}
  \begin{tikzcd}
    0\ar[r]& \widehat{L}(\lambda_{4})\ar[r]&
    \widehat{L}(\lambda_{5}^{0})\otimes
    E_{\alpha}(\lambda_{5}^{1})^{(1)}\ar[r]&
    \widehat{L}(\lambda_{5})\ar[r]&0,
  \end{tikzcd}
\end{displaymath}

\begin{displaymath}
  \begin{tikzcd}
     0\ar[r]& \widehat{L}(\lambda_{6})\ar[r]&
    \widehat{L}(\lambda_{7}^{0})\otimes
    E_{\beta}(\lambda_{7}^{1})^{(1)}\ar[r]&
    \widehat{L}(\lambda_{7})\ar[r]&0.
  \end{tikzcd}
\end{displaymath}
On obtient
\begin{equation}
    0\to \nabla_{p}(\lambda_{4})\to L(\lambda_{5}^{0})\otimes
    H^{0}(E_{\alpha}(\lambda_{5}^{1}))^{(1)}\to \nabla_{p}(\lambda_{5})\xrightarrow{\partial_{\alpha}}
  L(\lambda_{4}^{0})\otimes
  H^{1}(\lambda_{4}^{1})^{(1)}\to\cdots,
\label{eq:d4ddce0a686cbd28}
\end{equation}
et
\begin{equation}
    0\to \nabla_{p}(\lambda_{6})\to L(\lambda_{7}^{0})\otimes
    H^{0}(E_{\beta}(\lambda_{7}^{1}))^{(1)}\to
    \nabla_{p}(\lambda_{7})\xrightarrow{\partial_{\beta}}
  L(\lambda_{6}^{0})\otimes
  H^{1}(\lambda_{6}^{1})^{(1)}\to\cdots.
\label{eq:5631c02386712dbe}
\end{equation}

De plus, on a \( \lambda_{4}^{1}=(a^{1}-2,b^{1}) \) et \(
\lambda_{6}^{1}=(a^{1},b^{1}-2) \).

\underline{Si \( a^{1}\geq 1 \) et \(
b^{1}\geq 1 \),} on a \(
\lambda_{4}^{1},\lambda_{6}^{1} \in C\) et \(
H^{1}(\lambda_{4}^{1})=H^{1}(\lambda_{6}^{1})=0 \), d'où \(
\partial_{\alpha}=\partial_{\beta}=0 \). C'est-à-dire, \( L(\lambda_{4}^{0})\otimes
    H^{0}(E_{\alpha}(\lambda_{5}^{1}))^{(1)}  \) est juste une
    extension de \( L(\lambda_{5}^{0})\otimes
    H^{0}(\lambda_{5}^{1})^{(1)} \) par
    \( L(\lambda_{4}^{0})\otimes H^{0}(\lambda_{4}^{1})^{(1)} \), et
    \( L(\lambda_{6}^{0})\otimes
    H^{0}(E_{\beta}(\lambda_{7}^{1}))^{(1)}  \) est juste une
    extension de \( L(\lambda_{7}^{0})\otimes
    H^{0}(\lambda_{7}^{1})^{(1)} \) par
    \( L(\lambda_{6}^{0})\otimes H^{0}(\lambda_{6}^{1})^{(1)} \).

Donc d'après la \autoref{thm:H0pfilt},   il existe dans ce cas
une filtration de \( H^{0}(\lambda) \) dont les quotients sont \(
L(\nu_{i}^{0})\otimes H^{0}(\lambda_{i}^{1})^{(1)} \) pour \(
i\in\{1,2,\cdots,9\} \) (certains peuvent être nuls).

\underline{Si \( a^{1}=0 \)}, alors
\(\lambda_{5}^{1}=(a^{1},b^{1}-1)=(0,b^{1}-1) \), d'où \(
H^{i}(E_{\alpha}(\lambda_{5}^{1}))=0 \) pour tout \( i \) d'après le
\autoref{lemmaE0}. Donc le morphisme de bord \( \partial_{\alpha} \) dans
\eqref{eq:d4ddce0a686cbd28} est un isomorphisme de \(
L(\lambda_{5}^{0})\otimes H^{0}(\lambda_{5}^{1})^{(1)} \) sur
\( L(\lambda_{4}^{0})\otimes H^{1}(\lambda_{4}^{1})^{(1)} \).
Donc dans ce cas, non seulement le facteur correspondant à \(
\lambda_{4} \) n'apparaît pas, mais le facteur correspondant à \(
\lambda_{5}\) est \og
effacé\fg{} dans \( H^{0}(\lambda) \).

\smallskip
\underline{De même, si \( b^{1}=0 \)}, alors le facteur \(
\lambda_{7}\) est \og
effacé\fg{} dans \( H^{0}(\lambda) \).

\smallskip
3) Si \( \lambda \) est \( \alpha \)-singulier, alors les plus hauts poids
des facteurs de composition de \( \widehat{Z}(\lambda) \) sont donnés
par la figure suivante, où \( \lambda_{1}=\lambda \):
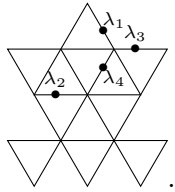
\begin{figure}[H]
  \centering
  \begin{tikzpicture}[scale=0.7]
    \draw (0,0)--(3,0)--(1.5,-3*sin{60})--cycle; \draw
    (0.5,-sin{60})--(1.5,sin{60})--(2.5,-sin{60})--cycle; \draw
    (1,0)--(2.5,-3*sin{60})--(3,-2*sin{60})--(0,-2*sin{60})--(0.5,-3*sin{60})--(2,0);
    \node[font=\tiny] at (1.8,0.4*sin{60}){\( \bullet \)}; \node[font=\tiny] at
    (2,0.4*sin{60}+0.2){\( \lambda_{1} \)}; \node[font=\tiny] at (2.4,0){\( \bullet \)}; \node[font=\tiny] at
    (2.4,0.3){\( \lambda_{3} \)}; \node[font=\tiny] at (1.8,-0.4*sin{60}){\( \bullet \)}; \node[font=\tiny] at
    (2,-0.4*sin{60}-0.2){\( \lambda_{4} \)}; \node[font=\tiny] at (0.9,-sin{60}){\( \bullet \)};
    \node[font=\tiny] at (0.9,-sin{60}+0.3){\( \lambda_{2} \)};
  \end{tikzpicture}.  
  \caption{$\alpha$-singulier}
  \label{fig:c6fb49c3d7ebc0b1}
\end{figure}

Écrivons \( \lambda_{i}=p\lambda_{i}^{1}+\lambda_{i}^{0} \) avec \(
\lambda_{i}^{0}\in X_{1}(T) \). On sait que \(
\lambda_{2}^{0}=\lambda_{3}^{0}=(s,\overline{s}) \) et \( \widehat{L}(\lambda_{2}) \) et \(
\widehat{L}(\lambda_{3}) \) forment le facteur \(
\widehat{L}(\lambda_{3}^{0})\otimes E_{\alpha}(\lambda_{3}^{1})^{(1)}
\).  Appliquons le foncteur \( \ind_{BG_{1}}^{G}(\bullet) \) à la suite
exacte suivante:
\begin{displaymath}
  \begin{tikzcd}
    0\ar[r]& \widehat{L}(\lambda_{2})\ar[r]&
    \widehat{L}(\lambda_{3}^{0})\otimes
    E_{\alpha}(\lambda_{3}^{1})^{(1)}\ar[r]&
    \widehat{L}(\lambda_{3})\ar[r]&0.
  \end{tikzcd}
\end{displaymath}
On obtient
\begin{equation}
  \label{eq:ef983ade747f7c18}
    0\to \nabla_{p}(\lambda_{2})\to L(\lambda_{3}^{0})\otimes
    H^{0}(E_{\alpha}(\lambda_{3}^{1}))^{(1)}\to
    \nabla_{p}(\lambda_{3})\xrightarrow{\partial_{\alpha}} L(\lambda_{2}^{0})\otimes
  H^{1}(\lambda_{2}^{1})^{(1)}\to \cdots.
\end{equation}

De plus, on a \( \lambda_{2}^{1}=(a^{1}-1,b^{1})\in C \) car \(
a^{1},b^{1}\geq 0 \), d'où \( H^{1}(\lambda_{2}^{1})=0 \). C'est-à-dire, \( L(\lambda_{2}^{0})\otimes
    H^{0}(E_{\alpha}(\lambda_{3}^{1}))^{(1)}  \) est juste une
    extension de \( L(\lambda_{3}^{0})\otimes
    H^{0}(\lambda_{3}^{1})^{(1)} \) par
    \( L(\lambda_{2}^{0})\otimes H^{0}(\lambda_{2}^{1})^{(1)} \).

Donc d'après la \autoref{thm:H0pfilt},   il existe dans ce cas
une filtration de \( H^{0}(\lambda) \) dont les quotients sont les \(
L(\nu_{i}^{0})\otimes H^{0}(\lambda_{i}^{1})^{(1)} \) pour \(
i\in\{1,2,\cdots,4\} \).

\smallskip
4) Si \( \lambda \) est \( \beta \)-singulier, alors les plus hauts poids
des facteurs de composition de \( \widehat{Z}(\lambda) \) sont donnés
par la figure suivante, où \( \lambda_{1}=\lambda \):
\begin{figure}[H]
  \centering
  \begin{tikzpicture}[scale=0.7]
    \draw (0,0)--(3,0)--(1.5,-3*sin{60})--cycle; \draw
    (0.5,-sin{60})--(1.5,sin{60})--(2.5,-sin{60})--cycle; \draw
    (1,0)--(2.5,-3*sin{60})--(3,-2*sin{60})--(0,-2*sin{60})--(0.5,-3*sin{60})--(2,0);
    \node[font=\tiny] at (1.2,0.4*sin{60}){\( \bullet \)}; \node[font=\tiny] at
    (1,0.4*sin{60}+0.2){\(\lambda_{1}  \)}; \node[font=\tiny] at (0.6,0){\( \bullet \)}; \node[font=\tiny] at
    (0.6,0.3){\( \lambda_{3} \)}; \node[font=\tiny] at (1.2,-0.4*sin{60}){\( \bullet \)}; \node[font=\tiny] at
    (1,-0.4*sin{60}-0.2){\( \lambda_{4} \)}; \node[font=\tiny] at (2.1,-sin{60}){\( \bullet \)};
    \node[font=\tiny] at (2.1,-sin{60}+0.3){\( \lambda_{2} \)};
  \end{tikzpicture}. 
  \caption{$\beta$-singulier}
  \label{fig:0dd758a606ef7b71}
\end{figure}
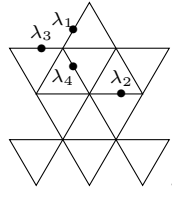

Comme dans le cas \( \alpha \)-singulier,  il existe dans ce cas
une filtration de \( H^{0}(\lambda) \) dont les quotients sont les \(
L(\nu_{i}^{0})\otimes H^{0}(\lambda_{i}^{1})^{(1)} \) pour \(
i\in\{1,2,\cdots,4\} \).

5) Si \( \lambda \) est \( \gamma \)-singulier, alors les plus hauts poids
des facteurs de composition de \( \widehat{Z}(\lambda) \) sont donnés
par la figure suivante, où \( \lambda_{1}=\lambda \):
\begin{figure}[H]
  \centering
  \begin{tikzpicture}[scale=0.7]
    \draw (0,0)--(3,0)--(1.5,-3*sin{60})--cycle; \draw
    (0.5,-sin{60})--(1.5,sin{60})--(2.5,-sin{60})--cycle; \draw
    (1,0)--(2.5,-3*sin{60})--(3,-2*sin{60})--(0,-2*sin{60})--(0.5,-3*sin{60})--(2,0);
    \node[font=\tiny] at (1.5,0){\( \bullet \)}; \node[font=\tiny] at (1.5,0.3){\( \lambda_{1} \)}; \node[font=\tiny] at
    (2.25,-0.5*sin{60}){\( \bullet \)}; \node[font=\tiny] at
    (2.05,-0.5*sin{60}-0.2){\( \lambda_{3} \)}; \node[font=\tiny] at
    (1.5,-2*sin{60}){\( \bullet \)}; \node[font=\tiny] at (1.5,-2*sin{60}+0.3){\( \lambda_{4} \)};
    \node[font=\tiny] at (0.75,-0.5*sin{60}){\( \bullet \)}; \node[font=\tiny] at
    (0.95,-0.5*sin{60}-0.2){\( \lambda_{2} \)};
  \end{tikzpicture}.
  
  \caption{$\gamma$-singulier}
  \label{fig:747f3465e550aee7}
\end{figure}
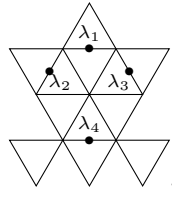

Comme il n'y a pas de facteur \( E_{\alpha}(\nu) \) ou \(
E_{\beta}(\nu) \) dans ce cas, alors d'après la \autoref{thm:H0pfilt} il existe
une filtration de \( H^{0}(\lambda) \) dont les quotients sont les \(
L(\nu_{i}^{0})\otimes H^{0}(\lambda_{i}^{1})^{(1)} \) pour \(
i\in\{1,2,3,4\} \).

6) Si \( \lambda \) est \( \alpha \)-\( \beta \)-singulier, alors \[
\widehat{Z}(\lambda)= \widehat{L}(p-1,p-1)\otimes
(a^{1},b^{1})^{(1)}=\widehat{L}(\lambda^{0})\otimes p\lambda^{1}. \]
\begin{figure}[H]
  \centering
  \begin{tikzpicture}[scale=0.7]
    \draw (0,0)--(3,0)--(1.5,-3*sin{60})--cycle; \draw
    (0.5,-sin{60})--(1.5,sin{60})--(2.5,-sin{60})--cycle; \draw
    (1,0)--(2.5,-3*sin{60})--(3,-2*sin{60})--(0,-2*sin{60})--(0.5,-3*sin{60})--(2,0);
    \node[font=\tiny] at (1.5,-sin{60}){\( \bullet \)};
    \node[font=\tiny] at (1.9,-sin{60}+0.25){\( \lambda \)};
  \end{tikzpicture}
  
  \caption{$\alpha$-$\beta$-singulier}
  \label{fig:6a2e4c2bd6db7dd8}
\end{figure}
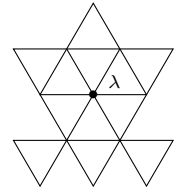

Dans ce cas, on a \[ H^{0}(\lambda)\cong L(\lambda^{0})\otimes H^{0}(\lambda^{1})^{(1)}. \]

En conclusion, on obtient comme corollaire une autre démonstration du
résultat suivant de  Jantzen (\cite{Jan80} 3.13, voir aussi
\cite{KH84} 2.4):

\begin{cor}[Jantzen]\label{cor:pweylfiltrationjantzen}
  Soit \( \lambda=(a,b)\in X(T)^{+} \). Écrivons \( a=a^{1}p+r \)
  et \( b=b^{1}p+s \) avec \( 0\leq r,s\leq p-1 \). Soit
\( 0=N_{0}\subset N_{1}\subset\cdots\subset N_{\ell-1}\subset N_{\ell}=\widehat{Z}(\lambda) \)
une suite de composition de \( \widehat{Z}(\lambda) \) induite par une D-filtration. Notons \( N_{i}/N_{i-1}\cong
\widehat{L}(\nu_{i}^{0})\otimes p\nu_{i}^{1} \) pour \(
i\in\{1,2,\cdots,\ell\} \) et \( \nu_{i}=\nu_{i}^{0}+p\nu_{i}^{1} \). Posons
  \( \widetilde{N_{i}}=\ind_{BG_{1}}^{G}(N_{i})\cong H^{0}(G/BG_{1},N_{i}) \).
Alors \( H^{0}(\lambda) \) possède une filtration
\( 0=\widetilde{N_{0}}\subset \widetilde{N_{1}}\subset\cdots\subset \widetilde{N_{\ell-1}}\subset\widetilde{N_{\ell}}=H^{0}(\lambda) \)
telle que \( \widetilde{N_{i}}/\widetilde{N_{i-1}}\cong L(\nu_{i}^{0})\otimes
M_{i}^{(1)} \) où
\begin{displaymath}
  M_{i}=
  \begin{cases}
    0 & \text{si }\nu_{i}^{1}\notin X(T)^{+}, \\
    0 & \text{si } \lambda \text{ est de type } \nabla,\, a^{1}=0 \text{
      et } \nu_{i}=\lambda' \text{ dans la figure
      \ref{fig:d78dc2295e2ed985}},\\
    0 & \text{si } \lambda \text{ est de type } \nabla,\,
    b^{1}=0 \text{
      et } \nu_{i}=\lambda' \text{ dans la figure
      \ref{fig:2b380fca10c86e4c}},\\
    H^{0}(\nu_{i}^{1}) & \text{sinon.}
  \end{cases}
\end{displaymath}

\begin{figure}[H]
  \centering
  \begin{tikzpicture}[scale=0.7]
    \draw [line width=0.5mm] (-0.8,1.6*sin{60})--(0.8,-1.6*sin{60});
    \draw (-0.5,sin{60})--(0.5,sin{60})--(0,0)--(1,0)--(0.5,-sin{60});
    \draw (0.5,sin{60})--(1,0); \node[font=\tiny] at (0,sin{60}-0.3)
    {\( \lambda \)}; \node[font=\tiny] at (0.5,-0.3) {\( \lambda' \)};
  \end{tikzpicture}
  
  \caption{Alcôve $\nabla$ touchant le mur pour \( \alpha \)}
  \label{fig:d78dc2295e2ed985}
\end{figure}
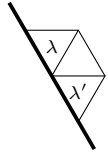

\begin{figure}[H]
  \centering
   \begin{tikzpicture}[scale=0.7]
    \draw [line width=0.5mm] (0.8,1.6*sin{60})--(-0.8,-1.6*sin{60});
    \draw (0.5,sin{60})--(-0.5,sin{60})--(0,0)--(-1,0)--(-0.5,-sin{60});
    \draw (-0.5,sin{60})--(-1,0); \node[font=\tiny] at (0,sin{60}-0.3)
    {\( \lambda \)}; \node[font=\tiny] at (-0.5,-0.3) {\( \lambda' \)};
  \end{tikzpicture}
  \caption{Alcôve $\nabla$ touchant le mur pour \( \beta \)}
  \label{fig:2b380fca10c86e4c}
\end{figure}
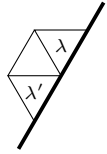
\end{cor}

Par dualité, on obtient aussi une \( p
\)-Weyl-filtration pour le module de Weyl \( V(\lambda) \).

\subsection{Existence d'une \( p \)-\( H^{i} \)-D-filtration}

Supposons maintenant que \(\mu\notin C\cup w_{0}\cdot C\). Alors
\(\mu=(m,-n-2)\) ou \((-n-2,m)\) avec \(m,n\in \mathbb{N}\). D'après
la symétrie entre \(\alpha\) et \(\beta\), on peut supposer que
\(\mu=(m,-n-2)\) sans perte de  généralité. 

Écrivons \(m=m^{1}p+r\) et  \(n=n^{1}p+s\) avec
  \(0\leq s,r<p\). D'après le paragraphe \ref{subsection:Zhat}, il existe
  une D-filtration 
\( 0=N_{0}\subset N_{1}\subset N_{2}\subset\cdots \subset N_{\ell}=\widehat{Z}(\mu) \)
 telle que \(N_{i}/N_{i-1}\cong
\widehat{L}(\nu_{i}^{0})\otimes E_{\delta_{i}}(\nu_{i}^{1})^{(1)}\) où
\(\delta_{i}\in \{0,\alpha,\beta\}\).   Listons tous les 
  \(\widehat{L}(\nu_{i}^{0})\otimes E_{\delta_{i}}(\nu_{i}^{1})\) possibles:
  \begin{enumerate}[(I)]
  \item  si \(\mu\) est de type \(\Delta\), il y a les sept facteurs suivants: 
    \begin{equation}\label{liste1}\hskip-7mm
      \begin{aligned}
       & \widehat{L}(r,\overline{s})\otimes (m^{1},-n^{1}-1)^{(1)},\quad
        \widehat{L}(s,\overline{r})\otimes
        (m^{1},-n^{1}-1)^{(1)},\quad\widehat{L}(s,\overline{r})\otimes (m^{1}-1,-n^{1}-2)^{(1)}\\
        &\widehat{L}(r-s-1,s)\otimes
        E_{\alpha}(m^{1}+1,-n^{1}-2)^{(1)},\quad
        \widehat{L}(\overline{r},r-s-1)\otimes
        E_{\beta}(m^{1}-1,-n^{1})^{(1)},\\
        &\widehat{L}(\overline{r}+s+1,r)\otimes (m^{1},-n^{1}-2)^{(1)}, \quad
        \widehat{L}(\overline{s},\overline{r}+s+1)\otimes (m^{1}-1,-n^{1}-1)^{(1)};
      \end{aligned}
    \end{equation}
   
  \item  si \(\mu\) est de type \(\nabla\), il y a les sept facteurs suivants:
    \begin{equation}\label{liste2}\hskip-7mm
      \begin{aligned}
        &\widehat{L}(r,\overline{s})\otimes (m^{1},-n^{1}-1)^{(1)},\quad
        \widehat{L}(\overline{r},r+\overline{s}+1)\otimes
        (m^{1}-1,-n^{1}-1)^{(1)},
        \\
        &\widehat{L}(r+\overline{s}+1,s)\otimes (m^{1},-n^{1}-2)^{(1)},\quad
        \widehat{L}(\overline{s},s-r-1)\otimes
        E_{\alpha}(m^{1},-n^{1}-2)^{(1)},\\
        &\widehat{L}(s-r-1,r)\otimes E_{\beta}(m^{1}-1,-n^{1}-1)^{(1)},\quad
        \widehat{L}(r,\overline{s})\otimes (m^{1}-1,-n^{1}-2)^{(1)},\\
        &\widehat{L}(s,\overline{r})\otimes (m^{1}-1,-n^{1}-2)^{(1)};
     \end{aligned}
    \end{equation}
   
  \item si \(\mu\) est \(\alpha\)-singulier, il y a les trois facteurs
    suivants:
    \begin{equation}
      \label{liste3}
      \begin{aligned}
        &\widehat{L}(p-1,\overline{s})\otimes (m^{1},-n^{1}-1)^{(1)},\quad
        \widehat{L}(\overline{s},s)\otimes E_{\alpha}(m^{1}+1,-n^{1}-2)^{(1)},\\
        &\widehat{L}(s,p-1)\otimes (m^{1},-n^{1}-2)^{(1)};
      \end{aligned}
    \end{equation}
    
 \item  si \(\mu\) est \(\beta\)-singulier, il y a les trois facteurs suivants:
   \begin{equation}
     \label{liste4}
     \begin{aligned}
      & \widehat{L}(r,p-1)\otimes
       (m^{1},-n^{1}-2)^{(1)},\quad \widehat{L}(\overline{r},r)\otimes
       E_{\beta}(m^{1}-1,-n^{1}-1)^{(1)},\\
       &\widehat{L}(p-1,\overline{r})\otimes
       (m^{1}-1,-n^{1}-2)^{(1)};
     \end{aligned}
   \end{equation}
  
\item si \(\mu\) est \(\gamma\)-singulier, il y a les quatre facteurs suivants:
  \begin{equation}
    \label{liste5}
    \begin{aligned}
      &\widehat{L}(r,\overline{r})\otimes (m^{1},-n^{1}-1)^{(1)},\quad
      \widehat{L}(p-1,r)\otimes  (m^{1},-n^{1}-2)^{(1)},\\
      &\widehat{L}(\overline{r},p-1)\otimes (m^{1}-1,-n^{1}-1)^{(1)},\quad
      \widehat{L}(r,\overline{r})\otimes (m^{1}-1,-n^{1}-2)^{(1)};
    \end{aligned}
\end{equation}

\item si \(\mu\) est \(\alpha\)-\(\beta\)-singulier, il n'y a que le facteur:
  \begin{equation}
    \label{liste6}
    \widehat{L}(p-1,p-1)\otimes (m^{1},-n^{1}-2)^{(1)}.
  \end{equation}

\end{enumerate}

Donc pour la partie à tordre par le Frobenius, il n'y a que les huit
possibilités suivantes:
 \begin{multline}
      \label{liste0}
      (m^{1},-n^{1}-1),(m^{1}-1,-n^{1}-1),(m^{1},-n^{1}-2),(m^{1}-1,-n^{1}-2)\\
      E_{\alpha}(m^{1}+1,-n^{1}-2),E_{\alpha}(m^{1},-n^{1}-2),
      E_{\beta}(m^{1}-1,-n^{1}), E_{\beta}(m^{1}-1,-n^{1}-1).
    \end{multline}

    Énonçons maintenant le théorème principal de cette section:
\begin{thm}[Existence d'une \(p\)-\(H^{i}\)-D-filtration]\label{thm:pHifiltration}
  Supposons que  \(\mu \notin C\cup w_{0}\cdot C\). Soit
\( 0=N_{0}\subset N_{1}\subset N_{2}\subset\cdots \subset N_{\ell}=\widehat{Z}(\mu) \)
une D-filtration de \(\widehat{Z}(\mu)\) (cf. le paragraphe
\ref{subsection:Zhat})
telle que \(N_{i}/N_{i-1}\cong
\widehat{L}(\nu_{i}^{0})\otimes E_{\delta_{i}}(\nu_{i}^{1})^{(1)}\) où
\(\delta_{i}\in \{0,\alpha,\beta\}\). Alors  \( H^{1}(\mu) \)
possède une filtration
\( 0=\widetilde{N_{0}}\subset \widetilde{N_{1}}\subset
\widetilde{N_{1}}\subset\cdots\subset \widetilde{N_{\ell}}=H^{1}(\mu) \)
où \( \widetilde{N_{i}}\cong H^{1}(G/BG_{1},N_{i}) \) et  l'on a \(\widetilde{N_{i}}/\widetilde{N_{i-1}}\cong L(\nu_{i}^{0})\otimes
H^{1}(E_{\delta_{i}}(\nu_{i}^{1}))^{(1)}.\)

De même,
 \( H^{2}(\mu) \) possède une filtration
\( 0=\widetilde{N_{0}}\subset \widetilde{N_{1}}\subset
\widetilde{N_{1}}\subset\cdots\subset \widetilde{N_{\ell}}=H^{2}(\mu) \)
où \( \widetilde{N_{i}}\cong H^{2}(G/BG_{1},N_{i}) \) et l'on a  \(\widetilde{N_{i}}/\widetilde{N_{i-1}}\cong L(\nu_{i}^{0})\otimes
H^{2}(E_{\delta_{i}}(\nu_{i}^{1}))^{(1)}.\)

De plus, si \(\mu=(m,-n-2)\) avec \(m=m^{1}p+r\), \(n=n^{1}p+s\) et
\(0\leq r,s\leq p-1\), alors la liste des \(\nu_{i}^{0},\nu_{i}^{1}\) se trouve dans
\eqref{liste1},\eqref{liste2},\eqref{liste3},\eqref{liste4},\eqref{liste5}
et \eqref{liste6}.

On appelle cette filtration de \(H^{i}(\mu)\) {\bf \og une
  \(p\)-\(H^{i}\)-D-filtration\fg{} }.

\end{thm}

Avant de démontrer ce théorème, prouvons d'abord le lemme suivant:
  \begin{lemma}\label{lemma:H0H3}
    Soit \(\mu=(m,-n-2)\) avec \(m,n\in \mathbb{N}\).
    Utilisons les notations du
    \autoref{thm:pHifiltration}. Alors
\( H^{0}(G/BG_{1}, N_{i}/N_{i-1})=H^{3}(G/BG_{1}, N_{i}/N_{i-1})=0 \)
 pour \(1\leq i\leq \ell\) et 
\( H^{0}(G/BG_{1}, N_{i})=H^{3}(G/BG_{1}, N_{i})=0 \)
pour \(0\leq i\leq \ell\).

  \end{lemma}
  \begin{proof}
    Écrivons \(m=m^{1}p+r\) et \(n=n^{1}p+s\) avec
  \(0\leq s,r<p\). 
    
    Comme \(H^{j}(G/BG_{1}, N_{i}/N_{i-1})\cong L(\nu_{i}^{0})\otimes
    H^{j}(E_{\delta_{i}}(\nu_{i}^{1}))^{(1)}\) pour tout \(i,j\)
    (cf. \cite{Jan03} II.9.13), pour
    montrer la première assertion il suffit de montrer que
    \(H^{0}(E)=H^{3}(E)=0\) pour tout
    \(E\) dans \eqref{liste0}.
    
    Comme \(m^{1},n^{1}\geq 0\), on a \(m^{1}-1\geq -1\),
    \(-n^{1}-1\leq -1\) et \(-n^{1}-2\leq -2\). Donc aucun poids dans
    la première ligne de \eqref{liste0} n'a de \(H^{0}\) ou
    \(H^{3}\).
    Les deux poids de \(E_{\alpha}(m^{1}+1,-n^{1}-2)\) sont
    \((m^{1}+1,-n^{1}-2)\) et \((m^{1}-1,-n^{1}-1)\), qui n'ont pas de
    \(H^{0}\) ou \(H^{3}\), d'où l'assertion pour
    \(E_{\alpha}(m^{1}+1,-n^{1}-2)\).

    Les deux poids de \(E_{\alpha}(m^{1},-n^{1}-2)\) sont
    \((m^{1},-n^{1}-2)\) et \((m^{1}-2,-n^{1}-1)\), qui n'ont jamais
    de \(H^{0}\) car \(-n^{1}-1\leq -1\). Si \(m^{1}\geq 1\), ils n'ont pas de \(H^{3}\) non
    plus. Si \(m^{1}=0\), le poids \((m^{1}-2,-n^{1}-1)\) peut avoir
    un \(H^{3}\) non nul. Mais dans ce cas, on a encore que
    \(H^{3}(E_{\alpha}(m^{1},-n^{1}-2)=H^{3}(E_{\alpha}(0,-n^{1}-2))=0\)
    par le \autoref{lemmaE0}.

    Les deux poids de \(E_{\beta}(m^{1}-1,-n^{1})\) sont
    \((m^{1}-1,-n^{1})\) et \((m^{1},-n^{1}-2)\), qui n'ont jamais de
    \(H^{3}\) car \(m^{1}-1\geq -1\). Si \(n^{1}\geq 1\), ils n'ont pas
    de \(H^{0}\) non plus. Si \(n^{1}=0\), on a encore que
    \(H^{0}(E_{\beta}(m^{1}-1,-n^{1}))=H^{0}(E_{\beta}(m^{1}-1,0))=0\)
    par le \autoref{lemmaE0}.

    En conclusion, on a que
    \(
H^{0}(G/BG_{1}, N_{i}/N_{i-1})=H^{3}(G/BG_{1}, N_{i}/N_{i-1})=0
\) pour tout \(1\leq i\leq \ell\).
La deuxième assertion s'en déduit par récurrence sur
\(i\).

\end{proof}

\begin{proof}[Démonstration du \autoref{thm:pHifiltration}]
  Par dualité de Serre contravariante, on a \( H^{i}(m,-n-2) \cong
  H^{3-i}(-m-2,n)\), donc il suffit de traiter le cas où \(
  m\geq n \).

  Pour tout \( BG_{1} \)-module \( M \), notons
  \[\chi_{1}(M)=\sum_{i\geq 0}(-1)^{i}\ch
    H^{i}(G/BG_{1},M).\]
  Comme le foncteur \( \ind_{B}^{BG_{1}} \) est exact (\cite{Jan03}
  II.9.12), alors pour tout \( B \)-module \( M \), on a
 \( \chi(M)=\chi_{1}(\ind_{B}^{BG_{1}}(M)) \).
  Comme la caractéristique  d'Euler-Poincaré \( \chi_{1}(\cdot) \) est additive sur les suites exactes,   on a
  \begin{equation}
    \label{eq:fd1c0255d82f9b29}
    \chi(\mu)=\chi_1
    (\widehat{Z}(\mu))=\sum_{i=1}^{\ell}\chi_{1}(N_{i}/N_{i-1}).
  \end{equation}

  Comme \( \mu\notin C\cup w_{0}\cdot C\), on a
  \begin{equation}
    \label{eq:6a6e73d407e6d757}
    \chi(\mu)=-\ch H^{1}(\mu)+\ch H^{2}(\mu).
 \end{equation}

  En outre, d'après le \autoref{lemma:H0H3}, on a
  \begin{equation}
    \label{eq:f1a2edac09576e70}
    \chi_{1}(N_{i}/N_{i-1})=-\ch H^{1}(G/BG_{1},N_{i}/N_{i-1})+\ch H^{2}(G/BG_{1},N_{i}/N_{i-1})
  \end{equation}
  pour tout \( i \).

  Donc d'après \eqref{eq:fd1c0255d82f9b29},
  \eqref{eq:6a6e73d407e6d757} et \eqref{eq:f1a2edac09576e70}, on a
  \begin{equation}
    \label{eq:6fb12e23f3d1bd71}
    \ch H^{1}(\mu)-\sum_{i=1}^{\ell}\ch
    H^{1}(G/BG_{1},N_{i}/N_{i-1})=\ch H^{2}(\mu)-\sum_{i=1}^{\ell}\ch H^{2}(G/BG_{1},N_{i}/N_{i-1}).
  \end{equation}

  Comme on a 
  \begin{displaymath}
   H^{j}(G/BG_{1},N_{i}/N_{i-1})\cong
   H^{j}(G/BG_{1},\widehat{L}(\nu_{i}^{0})\otimes
   E_{\delta_{i}}(\nu_{i}^{1})^{(1)})\cong L(\nu_{i}^{0})\otimes H^{j}(E_{\delta_{i}}(\nu_{i}^{1}))^{(1)},
 \end{displaymath}
 alors le \autoref{thm:pHifiltration} découle du
 \autoref{lemma:filtrationinduite} du paragraphe \ref{subsection:Jantzen} et de la
 proposition suivante.

\end{proof}

\begin{proposition}\label{thm:presquepfiltration1}
 Soit \(\mu=(m,-n-2)\) avec \(m,n\in \mathbb{N}\). Soit
\( 0=N_{0}\subset N_{1}\subset N_{2}\subset\cdots \subset N_{\ell}=\widehat{Z}(\mu) \)
une D-filtration de \(\widehat{Z}(\mu)\) (cf. le paragraphe
\ref{subsection:Zhat})
telle que \(N_{i}/N_{i-1}\cong
\widehat{L}(\nu_{i}^{0})\otimes E_{\delta_{i}}(\nu_{i}^{1})^{(1)}\) où
\(\delta_{i}\in \{0,\alpha,\beta\}\).
Si \(m\geq n\), alors on a
\begin{equation}
  \label{eq:nablapfilt1}
  \ch H^{2}(\mu)=\sum_{i=1}^{\ell}\ch L(\nu_{i}^{0})\ch H^{2}(E_{\delta_{i}}(\nu_{i}^{1}))^{(1)}.
\end{equation}

\end{proposition}

\subsection{Preuve de la  \texorpdfstring{\autoref{thm:presquepfiltration1}}{Théorème}}
\begin{proof}

  Écrivons \( m=m^{1}p+r \) et \( n=n^{1}p+s \) avec \( 0\leq r,s\leq p-1 \).

  \fbox{Supposons d'abord que \(n=0\).} Alors
  \(H^{2}(\mu)=H^{2}(m,-2)=0\) d'après la \autoref{rmk:Griffithcondition}.
  Dans ce cas, on a \(n^{1}=s=0\) et
\(\mu\) ne peut pas être
de type \(\nabla\) ou \(\beta\)-singulier, donc  les
\(E_{\delta_{i}}(\nu_{i}^{1})\)  possibles sont:
\begin{equation}\label{liste0'}
(m^{1},-1),(m^{1}-1,-1),(m^{1},-2),(m^{1}-1,-2),
E_{\alpha}(m^{1}+1,-2), E_{\beta}(m^{1}-1,0).
\end{equation}

On sait que \(H^{2}(m^{1},-1)=H^{2}(m^{1}-1,-1)=0\) pour tout
\(m^{1}\) (cf. \cite{Jan03} II.5.4.a)). Comme \(m^{1}\geq 0\), on a
\(H^{2}(m^{1},-2)=H^{2}(m^{1}-1,-2)=0\). De même, 
\(H^{2}(m^{1}+1,-2)=0\), et
\(H^{2}((m^{1}+1,-2)-\alpha)=H^{2}(m^{1}-1,-1)=0\), d'où
\(H^{2}(E_{\alpha}(m^{1}+1,-2))=0\). Enfin, on a
\(H^{2}(E_{\beta}(m^{1}-1,0))=0\) d'après le \autoref{lemmaE0}. Donc
l'égalité \eqref{eq:nablapfilt1} est vraie si \(n=0\).

  \fbox{Si \(n\geq 1\) et \(\mu\notin \widehat{\Gr}\) ,}
alors \(H^{2}(\mu)=0\). Montrons dans ce cas que
\(H^{2}(E)\) est aussi nul pour tout
\(E\) dans la liste \eqref{liste0}.

Comme \(n\geq 1\), il existe \(d\geq 0\) et \(a\in \{1,\cdots,p-1\}\)
 tels que \(ap^{d}\leq n<(a+1)p^{d}\). On a \(m\geq (a+1)p^{d}\) car
 \(\mu\notin \widehat{\Gr}\).

 Si \(d=0\), alors \(n^{1}=0\) et \(m^{1}\geq 0\). Alors on a déjà
 montré que tout \(E\) dans la liste \eqref{liste0'} n'a pas de
 cohomologie en degré \(2\). Dans \eqref{liste0}, il reste encore les deux termes
 \(E_{\alpha}(m^{1},-2)\) et
 \(E_{\beta}(m^{1}-1,-1)\). Mais ces deux termes n'apparaissent
 que si \(\mu\) est de type \(\nabla\) ou \(\beta\)-singulier, avec
 \(r< s\). Comme \(m\geq n\), il faut  que \(m^{1}\geq 1\) pour
 que   \(E_{\alpha}(m^{1},-2)\) ou
 \(E_{\beta}(m^{1}-1,-1)\) apparaissent. On sait que
 \(H^{2}(m^{1},-2)=H^{2}(m^{1}-1,-1)=0\). Comme
 \((m^{1},-2)-\alpha=(m^{1}-2,-1)\) et
 \((m^{1}-1,-1)-\beta=(m^{1},-3)\) n'ont pas de cohomologie en degré
 \(2\) non plus si \(m^{1}\geq 1\) d'après la \autoref{rmk:Griffithcondition}, on a
 \(H^{2}(E_{\alpha}(m^{1},-2))=H^{2}(E_{\beta}(m^{1}-1,-1))=0\). Donc
 l'égalité \eqref{eq:nablapfilt1} est vraie dans ce cas.

 Si \(d\geq 1\), alors \(ap^{d-1}\leq n^{1}<(a+1)p^{d-1}\) et
 \(m^{1}\geq (a+1)p^{d-1}\).  Dans ce cas, les poids suivants:
 \begin{multline*}
   (m^{1},-n^{1}-1),(m^{1}-1,-n^{1}-1),(m^{1},-n^{1}-2),(m^{1}-1,-n^{1}-2)\\
      (m^{1}+1,-n^{1}-2), (m^{1}+1,-n^{1}-2)-\alpha,
      (m^{1}-1,-n^{1}), (m^{1}-1,-n^{1})-\beta
 \end{multline*}
sont tous dans la chambre \( s_{\beta}\cdot C \) et hors de la région
de Griffith, donc n'ont pas de cohomologie en degré 2. Donc il reste à traiter \(E_{\alpha}(m^{1}, -n^{1}-2)\)
et \(E_{\beta}(m^{1}-1,-n^{1}-1)\) dans la liste \eqref{liste0}. 

Si \(m^{1}\geq (a+1)p^{d-1}+1\), alors
\((m^{1},-n^{1}-2)-\alpha=(m^{1}-2,-n^{1}-1)\) qui n'a pas de
cohomologie  en degré \(2\) car il est dans la chambre \(
s_{\beta}\cdot C \)  et hors de la région de
Griffith. Donc \(H^{2}(E_{\alpha}(m^{1},-n^{1}-2))=0\) dans
ce cas. Si \(m^{1}=(a+1)p^{d-1}\), alors on a aussi
\(H^{2}(E_{\alpha}(m^{1},-n^{1}-2))=0\) d'après la
\autoref{prop:Eparticulier}.

Si \(n^{1}\leq (a+1)p^{d-1}-2\), alors
\((m^{1}-1,-n^{1}-1)-\beta=(m^{1},-n^{1}-3)\) n'est pas dans la région
de Griffith, donc il n'a pas de cohomologie en degré 2 et
\(H^{2}(E_{\beta}(m^{1}-1,-n^{1}-1))=0\) dans ce cas. Si
\(n^{1}=(a+1)p^{d-1}-1\), alors
\[H^{2}(E_{\beta}(m^{1}-1,-n^{1}-1))=H^{2}(E_{\beta}(m^{1}-1,-(a+1)p^{d-1}))=0\]
d'après la \autoref{prop:Eparticulier}.

Par conséquent,  (\ref{eq:nablapfilt1}) est toujours vraie si \(\mu\notin\widehat{\Gr}\).

\fbox{Si \(\mu\in \widehat{\Gr}\),} raisonnons  par récurrence sur le degré
\(d\) de \( \mu \).

Si \(d=1\), alors \(\mu=(ap+r,-ap-s-2)\). Donc \(r\geq s\) et \(\mu\) doit être
de type \(\Delta\) ou \(\alpha\)-singulier ou \(\gamma\)-singulier ou
\(\alpha\)-\(\beta\)-singulier. Si \(\mu\) est de type \(\Delta\) ou \(\gamma\)-singulier, on a
\[H^{2}(\mu)\cong L(0,a-1)^{(1)}\otimes V(s,p-r-2)\cong L(s,ap-r-2)\]
et d'après \eqref{liste1} et \eqref{liste5} :
\begin{align*}
  \bigoplus_{i} L(\nu_{i}^{0})\otimes
  H^{2}(E_{\delta_{i}}(\nu_{i}^{1}))^{(1)}&=L(s,p-r-2)\otimes H^{2}(a-1,-a-2)^{(1)}\\
                          &\cong L(s,p-r-2)\otimes V(0,a-1)^{(1)}\\
                          &\cong L(s,ap-r-2)\\
  &\cong H^{2}(\mu),
\end{align*}
d'où (\ref{eq:nablapfilt1}). Dans les deux autres cas, on a \(r=p-1\)
et \(H^{2}(\mu)=0\). D'après \eqref{liste3} et \eqref{liste6}, on a
\(H^{2}(E_{\delta_{i}}(\nu_{i}^{1}))=0\) pour tout \(i\) car
\(m^{1}=n^{1}=a\). Donc les deux
cotés de \eqref{eq:nablapfilt1} sont nuls, et  l'égalité est
aussi vraie. Donc \eqref{eq:nablapfilt1} est vraie si \(\mu\in\widehat{\Gr}\)
est de degré \(d=1\).

Supposons l'égalité (\ref{eq:nablapfilt1})  vraie pour tout \(\mu\) de
degré \(\leq d\) dans une
\(H^{1}\)-chambre, et  montrons-la pour \( \mu \)  de degré
\(d+1\). D'après ce qu'on a déjà montré, il suffit de supposer
que \(\mu\in\widehat{\Gr}\) .

Écrivons  \(m=ap^{d+1}+a_{d}p^{d}+a_{d-1}p^{d-1}+\cdots+a_{1}p+r\)
 et \(n=ap^{d+1}+b_{d}p^{d}+b_{d-1}p^{d-1}+\cdots+b_{1}p+s\).
On a
\begin{equation}
    \label{eq:nablapfilt2}
    \begin{aligned}
      \ch H^{2}(\mu)=&\ch L(0,a-1)^{(d+1)}\ch H^{3}(\mu+(-a-1,a)p^{d+1})\\
      & +\ch L(0,a)^{(d+1)}\ch H^{2}(\mu+(-a,a)p^{d+1})\\
      &+\ch L(0,a-2)^{(d+1)}\ch H^{2}(\mu+(-a-1,a+1)p^{d+1}).
    \end{aligned}
  \end{equation}

  Notons \(\mu'=\mu+(-a,a)p^{d+1}=(m',-n'-2)\) et
  \(\mu''=\mu+(-a-1,a+1)p^{d+1}=(-n''-2,m'')\). Alors
  \[ m'=a_{d}p^{d}+a_{d-1}p^{d-1}+\cdots+a_{1}p+r,\]
  \[n'=b_{d}p^{d}+b_{d-1}p^{d-1}+\cdots+b_{1}p+s,\]
  \[m''=(p-1-b_{d})p^{d}+(p-1-b_{d-1})p^{d-1}+\cdots+(p-1-b_{1})p+p-s-2,\]
  \[n''=(p-1-a_{d})p^{d}+(p-1-a_{d-1})p^{d-1}+\cdots+(p-1-a_{1})p+p-r-2.\]
Donc \(\mu'\) et \(\mu''\) sont des poids de degré
\(\leq d\) dans une \(H^{1}\)-chambre (plus précisément, \( \mu'\in
s_{\beta}\cdot C \) et \( \mu''\in s_{\alpha}\cdot C \)).

Comme \(\mu'=\mu+(-a,a)p^{d+1}\), on sait que la D-filtration
de \(\widehat{Z}(\mu')\) est juste celle de \(\widehat{Z}(\mu)\) tensorisée
par \((-a,a)p^{d+1}\). De même, la D-filtration de
\(\widehat{Z}(\mu'')\) est celle de \(\widehat{Z}(\mu)\) tensorisée par
\((-a-1,a+1)p^{d+1}\) et la D-filtration de
\(\widehat{Z}(\mu+(-a-1,a)p^{d+1})\) est celle de \(\widehat{Z}(\mu)\)
tensorisée par \((-a-1,a)p^{d+1}\). 

Donc l'hypothèse de récurrence pour \(\mu'\) et \(\mu''\) (pour
\(\mu''\) on  utilise la symétrie entre \(\alpha\) et \(\beta\))
nous donne
\begin{equation}
  \label{eq:nablapfilt3}
  \ch H^{2}(\mu')=\sum_{i=1}^{\ell}\ch L(\nu_{i}^{0})\ch H^{2}(E_{\delta_{i}}(\nu^{1}_{i}+(-a,a)p^{d}))^{(1)}
\end{equation}
et
\begin{equation}
  \label{eq:nablapfilt4}
   \ch H^{2}(\mu'')=\sum_{i=1}^{\ell}\ch L(\nu_{i}^{0})\ch H^{2}(E_{\delta_{i}}(\nu^{1}_{i}+(-a-1,a+1)p^{d}))^{(1)}.
 \end{equation}

 De plus,   d'après la \autoref{thm:weyl} du paragraphe \ref{subsection:Jantzen}, on a
\[\ch H^{3}(\mu+(-a-1,a)p^{d+1})=\sum_{i=1}^{\ell}\ch L(\nu^{0}_{i})\ch
  H^{3}(E_{\delta_{i}}(\nu_{i}^{1}+(-a-1,a)p^{d})^{(1)}).\]

Posons \(m^{1}=ap^{d}+a_{d}p^{d-1}+\cdots+a_{1}=ap^{d}+\widetilde{r}\) et
\(n^{1}=ap^{d}+b_{d}p^{d-1}+\cdots+b_{1}=ap^{d}+\widetilde{s}\) avec \(
0\leq \widetilde{r},\widetilde{s}\leq p^{d}-1 \), alors tout poids de la liste
(\ref{liste0})  vérifie les conditions correspondantes de la
\autoref{thm:Edelta} et du \autoref{cor:2étages}. Plus précisément, si \( \delta_{i}=0 \), alors
\begin{equation}
  \nu_{i}^{1}\in\{(m^{1},-n^{1}-1),(m^{1}-1,-n^{1}-1),(m^{1},-n^{1}-2),(m^{1}-1,-n^{1}-2)\},
  \label{eq:35c2a4b1a3c1a6c9}  
\end{equation}
d'après \eqref{liste0}. On a \( m^{1}-1=ap^{d}+\widetilde{r}-1 \) avec \( -1\leq \widetilde{r}-1\leq
p^{d}-2 \) et \( n^{1}-1=ap^{d}+\widetilde{s}-1 \) avec \( -1\leq \widetilde{s}-1\leq
p^{d}-2 \), donc \( \widetilde{r},\widetilde{s},\widetilde{r}-1,\widetilde{s}-1 \)
vérifient l'hypothèse du \autoref{cor:2étages}, d'où
\begin{equation}
  \label{eq:edd4c74d46e14a91}
  \begin{aligned}
    \ch H^{2}(E_{0}(\nu_{i}^{1}))=&\ch L(0,a-1)^{(d)}\ch
    H^{3}(E_{0}(\nu_{i}^{1}+(-a-1,a)p^{d}))\\
    &+\ch L(0,a)^{(d)}\ch H^{2}(E_{0}(\nu_{i}^{1}+(-a,a)p^{d}))\\
    &+\ch L(0,a-2)^{(d)}\ch H^{2}(E_{0}(\nu_{i}^{1}+(-a-1,a+1)p^{d}))
  \end{aligned}
\end{equation}
si \( \delta_{i}=0 \).

Si \( \delta_{i}=\alpha \), alors
\[
E_{\alpha}(\nu_{i}^{1})\in\{E_{\alpha}(m^{1}+1,-n^{1}-2), E_{\alpha}(m^{1},-n^{1}-2)\}
\]
d'après \eqref{liste0}. On a \( m^{1}+1=ap^{d}+\widetilde{r}+1 \) avec \(
1\leq \widetilde{r}+1\leq p^{d} \) et \( n^{1}=ap^{d}+\widetilde{s} \) avec \(
0\leq \widetilde{s}\leq p^{d}-1 \), donc \( (m^{1}+1,-n^{1}-2) \) vérifie
l'hypothèse dans (i) de la \autoref{thm:Edelta}. D'autre part, le facteur \(
E_{\alpha}(m^{1},-n^{1}-2) \) apparaît seulement si \( \mu \) est de
type \( \nabla \), auquel cas on a \( s>r \). Mais \( m^{1}p+r=m\geq
n=n^{1}p+s \), donc \( m^{1}\geq n^{1}+1\geq ap^{d}+1 \), d'où \(
\widetilde{r}\geq 1 \) dans ce cas. Donc s'il existe \( i \) tel que
\( E_{\delta_{i}}(\nu_{i}^{1})=E_{\alpha}(m^{1},-n^{1}-2) \), alors \(
m^{1}=ap^{d}+\widetilde{r} \) avec \( 1\leq \widetilde{r}\leq p^{d}-1 \) et \(
n^{1}=ap^{d}+\widetilde{s} \) avec \( 0\leq \widetilde{s}\leq p^{d}-1 \), donc
\( (m^{1},-n^{1}-2) \) vérifie aussi l'hypothèse dans  (i) de la
\autoref{thm:Edelta}. Par conséquent, on a
\begin{equation}
  \label{eq:96d7df3c519a73ea}
  \begin{aligned}
    \ch H^{2}(E_{\alpha}(\nu_{i}^{1}))=&\ch L(0,a-1)^{(d)}\ch
    H^{3}(E_{\alpha}(\nu_{i}^{1}+(-a-1,a)p^{d}))\\
    &+\ch L(0,a)^{(d)}\ch H^{2}(E_{\alpha}(\nu_{i}^{1}+(-a,a)p^{d}))\\
    &+\ch L(0,a-2)^{(d)}\ch H^{2}(E_{\alpha}(\nu_{i}^{1}+(-a-1,a+1)p^{d}))
  \end{aligned}
\end{equation}
si \( \delta_{i}=\alpha \).

Si \( \delta_{i}=\beta \), alors on a
\[
E_{\beta}(\nu_{i}^{1})\in\{E_{\beta}(m^{1}-1,-n^{1}),E_{\beta}(m^{1}-1,-n^{1}-1)\}
\]
d'après \eqref{liste0}. On a \( m^{1}-1=ap^{d}+\widetilde{r}-1 \) avec \(
-1\leq \widetilde{r}\leq p^{d}-2 \) et \( n^{1}-2=ap^{d}+\widetilde{s}-2 \)
avec \( -2\leq \widetilde{s}\leq p^{d}-3 \), donc \( (m^{1}-1,-n^{1}) \)
vérifie l'hypothèse dans  (ii) de la \autoref{thm:Edelta}. D'autre
part, le facteur \( E_{\beta}(m^{1}-1,-n^{1}-1) \) apparaît seulement
si \( \mu \) est de type \( \nabla \), auquel cas on a \( s\geq r \). Mais \(
m^{1}p+r=m\geq n=n^{1}p+s \), donc on a \( n^{1}\leq m^{1}-1\leq
(a+1)p^{d}-2 \). Donc \( n^{1}-1=ap^{d}+\widetilde{s}-1 \) avec \( -1\leq
\widetilde{s}-1\leq p^{d}-3 \) dans ce cas, et l'hypothèse dans (ii) de la
\autoref{thm:Edelta} est aussi satisfaite. Par conséquent, on a
\begin{equation}
  \label{eq:c1b04cadb9e15d36}
  \begin{aligned}
    \ch H^{2}(E_{\beta}(\nu_{i}^{1}))=&\ch L(0,a-1)^{(d)}\ch
    H^{3}(E_{\beta}(\nu_{i}^{1}+(-a-1,a)p^{d}))\\
    &+\ch L(0,a)^{(d)}\ch H^{2}(E_{\beta}(\nu_{i}^{1}+(-a,a)p^{d}))\\
    &+\ch L(0,a-2)^{(d)}\ch H^{2}(E_{\beta}(\nu_{i}^{1}+(-a-1,a+1)p^{d}))
  \end{aligned}
\end{equation}
 si \( \delta_{i}=\beta \).

Par conséquent, on a
\begin{align*}
  \ch H^{2}(\mu)=& \ch L(0,a-1)^{(d+1)}\ch H^{3}(\mu+(-a-1,a)p^{d+1})\\
  &+\ch L(0,a)^{(d+1)}\ch H^{2}(\mu')+\ch L(0,a-2)^{(d+1)}\ch
    H^{2}(\mu'')\\
  =&\ch L(0,a-1)^{(d+1)}\sum_{i=1}^{\ell}\ch L(\nu^{0}_{i})\ch
     H^{3}( E_{\delta_{i}}(\nu_{i}^{1}+(-a-1,a)p^{d}))^{(1)} \\
  &+\ch L(0,a)^{(d+1)}\sum_{i=1}^{\ell}\ch L(\nu_{i}^{0})\ch
    H^{2}(E_{\delta_{i}}(\nu^{1}_{i}+(-a,a)p^{d}))^{(1)}\\
  &+\ch L(0,a-2)^{(d+1)}\sum_{i=1}^{\ell}\ch L(\nu_{i}^{0})\ch
    H^{2}(E_{\delta_{i}}(\nu^{1}_{i}+(-a-1,a+1)p^{d}))^{(1)}\\
  =&\sum_{i=1}^{\ell}\ch L(\nu_{i}^{0})[\ch L(0,a-1)^{(d)}\ch
     H^{3}(E_{\delta_{i}}(\nu_{i}^{1}+(-a-1,a)p^{d}))\\
  &+\ch L(0,a)^{(d)}\ch
    H^{2}(E_{\delta_{i}}(\nu^{1}_{i}+(-a,a)p^{d}))\\
  &+\ch L(0,a-2)^{(d)}\ch
    H^{2}(E_{\delta_{i}}(\nu^{1}_{i}+(-a-1,a+1)p^{d}))]^{(1)}\\
  =&\sum_{i=1}^{\ell}\ch L(\nu_{i}^{0})\ch
     H^{2}(E_{\delta_{i}}(\nu_{i}^{1})),
\end{align*}
où la dernière égalité résulte de  \eqref{eq:edd4c74d46e14a91},
\eqref{eq:96d7df3c519a73ea} et \eqref{eq:c1b04cadb9e15d36}.
Ceci termine la preuve de la \autoref{thm:presquepfiltration1} et donc
du \autoref{thm:pHifiltration}.

\end{proof}
\subsection{Conclusion}\label{subsection:Dfiltrationconclusion}

En combinant les Propositions \ref{thm:weyl} et  \ref{thm:H0pfilt}, le
\autoref{thm:pHifiltration} et le \autoref{lemma:H0H3}, on obtient le: 
\begin{thm}\label{pHiDfiltration}
  Soit \( \mu\in X(T) \). Soit
\( 0=N_{0}\subset N_{1}\subset N_{2}\subset\cdots \subset N_{\ell}=\widehat{Z}(\mu) \)
une D-filtration de \(\widehat{Z}(\mu)\) (cf. le paragraphe
\ref{subsection:Zhat})
telle que \(N_{i}/N_{i-1}\cong
\widehat{L}(\nu_{i}^{0})\otimes E_{\delta_{i}}(\nu_{i}^{1})^{(1)}\) où
\(\delta_{i}\in \{0,\alpha,\beta\}\). Alors pour tout \( j\in
\mathbb{N} \), il existe une filtration 
\( 0=\widetilde{N_{0}}\subset \widetilde{N_{1}}\subset
\widetilde{N_{1}}\subset\cdots\subset \widetilde{N_{\ell}}=H^{j}(\mu) \)
où \( \widetilde{N_{i}}\cong H^{j}(G/BG_{1},N_{i}) \) et  \(\widetilde{N_{i}}/\widetilde{N_{i-1}}\cong L(\nu_{i}^{0})\otimes
H^{j}(E_{\delta_{i}}(\nu_{i}^{1}))^{(1)}.\)
\end{thm}

\section{La  cohomologie des  \texorpdfstring{\(B\)-modules
    \(E_{\delta}(\mu)\)}{B-modules Edelta}}
\subsection{Motivation et premières propriétés}
Dans la section \ref{chapitre:pHifiltration}, on a montré que pour
tout \(\mu\in X(T)\), \(H^{i}(\mu)\) admet une filtration dont les quotients
sont de la forme \(L(\nu^{0})\otimes
H^{i}(E_{\delta}(\nu^{1}))^{(1)}\). Cette filtration introduit des
modules inconnus \(H^{i}(E_{\delta}(\nu))\), donc il faut étudier leur
structure pour bien connaître celle de \(H^{i}(\mu)\).

Pour \( i=0
\),
d'après la discussion suivant la \autoref{thm:H0pfilt} (\ref{subsection:Jantzen}), pour \(
\delta\in\{\alpha,\beta\} \),
tout \(
H^{0}(E_{\delta}(\nu)) \) qui apparaît dans la \( p \)-filtration de
\( H^{0}(\mu) \) est soit nul, soit une extension de \( H^{0}(\nu) \)
par \( H^{0}(\nu-\delta) \). Donc le problème pour \( i=0 \) ou \( 3
\) est déjà complètement résolu.

Pour \(i=1  \) ou \( i=2 \), la situation est plus
compliquée.

Rappelons qu'il existe des
suites exactes non scindées de \( B \)-modules:
\[
\xymatrix{0\ar[r]&\mu-\alpha\ar[r]&E_{\alpha}(\mu)\ar[r]&\mu\ar[r]&0}
\]
et
\[
\xymatrix{0\ar[r]&\mu-\beta\ar[r]&E_{\beta}(\mu)\ar[r]&\mu\ar[r]&0}.
\]

Appliquons le foncteur \(H^{0}(G/B,\bullet)\) aux suites exactes
ci-dessus. On obtient les suites exactes longues :
\begin{equation}
  \label{suite:Ealpha0}
    \cdots\to H^{1}(\mu-\alpha)\to H^{1}(E_{\alpha}(\mu))\to H^{1}(\mu)\xrightarrow{\partial_{\alpha}}
    H^{2}(\mu-\alpha)\to H^{2}(E_{\alpha}(\mu))\to H^{2}(\mu)\to\cdots
\end{equation}
et
\begin{equation}
  \label{suite:Ebeta0}
    \cdots\to H^{1}(\mu-\beta)\to H^{1}(E_{\beta}(\mu))\to H^{1}(\mu)\xrightarrow{\partial_{\beta}}
    H^{2}(\mu-\beta)\to H^{2}(E_{\beta}(\mu))\to H^{2}(\mu)\to \cdots
\end{equation}

Donc pour connaître la structure de \( H^{1}(E_{\delta}(\mu)) \) et \(
H^{2}(E_{\delta}(\mu)) \), il \og suffit\fg{} de connaître le
morphisme de bord \( \partial_{\delta} \). D'après le  \og Strong
Linkage Principle\fg{} (cf \cite{Jan03} II.6.13), on
sait que \(\partial_{\alpha}=0\) (resp. \(\partial_{\beta}=0\)) si
\(\mu-\alpha\notin W_{p}\cdot \mu\) (resp. \(\mu-\beta\notin
W_{p}\cdot \mu\)).  En outre, pour
\(\delta\in\{\alpha,\beta\}\), \(\mu-\delta\in
W_{p}\cdot \mu\) si et seulement si \(\langle \mu,
\delta^{\vee}\rangle\) est divisible par \(p\). Donc si
 \(p\nmid
\langle \mu,\delta^{\vee}\rangle\), alors \(H^{i}(E_{\delta}(\mu))\)
est la somme directe de \(H^{i}(\mu-\delta)\) et \(H^{i}(\mu)\) .

Soit \( T_{\nu}^{\mu} \) le foncteur de translation de \( \nu \) à \(
\mu \). On a la proposition suivante.
\begin{proposition}
Supposons que \(\mu=(x,y)\) et \(p\mid x\). Posons \(\nu=(x-1,y)\), c'est un
poids sur le mur entre \(\mu\) et \(\mu-\alpha\). Alors
\(H^{i}(E_{\alpha}(\nu))\cong T_{\nu}^{\mu}(H^{i}(\nu))\) si
 \(p\nmid y+1\).
\end{proposition}
\begin{proof}
 Par  définition de \(E_{\alpha}(\mu)\), on sait qu'il existe une
 suite exacte de \(B\)-modules:
  \[
    \begin{tikzcd}
      0\ar[r]&(0,-1)\ar[r]&L(1,0)\ar[r]&E_{\alpha}(1,0)\ar[r]&0.
    \end{tikzcd}
  \]
  Tensorisons par le poids \(\nu=(x-1,y)\). On obtient:
  \[
    \begin{tikzcd}
      0\ar[r]&\mu-\gamma\ar[r]&L(1,0)\otimes\nu\ar[r]&E_{\alpha}(\mu)\ar[r]&0.
    \end{tikzcd}
  \]

  Appliquant le foncteur \(H^{0}(G/B,\bullet)\) à cette suite exacte,
  on obtient une
  suite exacte longue de cohomologie :
    \[  \cdots\to H^{i-1}(E_{\alpha}(\mu))\xrightarrow{\partial_{i-1}}
        H^{i}(\mu-\gamma)\xrightarrow{\psi} L(1,0)\otimes
        H^{i}(\nu)\xrightarrow{\phi} H^{i}(E_{\alpha}(\mu))\xrightarrow{\partial_{i}}
        H^{i+1}(\mu-\gamma)\to \cdots.
    \]

    Si \(p\nmid y+1\), alors \(p\nmid x+y+1\). Dans ce cas
    \(\mu-\gamma\) n'appartient pas à \( W_{p}\cdot \mu \), d'où \(\partial_{i}=0\). Donc \(\phi\) est surjectif.

    Notons \(N=T_{\nu}^{\mu}(H^{i}(\nu))\). Alors \(N\cong \pr_{\mu}(L(1,0)\otimes
H^{i}(\nu))\subset L(1,0)\otimes H^{i}(\nu)\). Comme
\(\mu-\gamma\) n'appartient pas à \( W_{p}\cdot \mu \), alors \(\image \psi\cap
N=\{0\}\). Donc \(N\) est isomorphe à son image par \(\phi\). Or
\(\pr_{\mu}(H^{i}(E_{\alpha}(\mu)))=H^{i}(E_{\alpha}(\mu))\) est inclus dans l'image de
\(N=\pr_{\mu}(L(1,0)\otimes H^{i}(\nu))\) car \( \phi \) est surjectif. Donc \(N\cong \phi(N)=H^{i}(E_{\alpha}(\mu))\).   
\end{proof}

De même, on a une proposition analogue pour \(E_{\beta}\):
\begin{proposition}
Supposons que \(\mu=(x,y)\) et \(p\mid y\). Posons \(\nu=(x,y-1)\),
c'est  un
poids sur le mur entre \(\mu\) et \(\mu-\beta\). Alors
\(H^{i}(E_{\beta}(\nu))\cong T_{\nu}^{\mu}(H^{i}(\nu))\) si
 \(p\nmid x+1\).
\end{proposition}

\subsection{Morphismes de bord \texorpdfstring{\(\partial_{\alpha}\) et \(\partial_{\beta}\)}{}}\label{subsection:bordnul}
Commençons par la proposition suivante.
\begin{proposition} \label{prop:bordnul1}
  Soient \(\mu_{1}=(m_{1},-n_{1}-2)\) et \(\mu_{2}=(m_{2},-n_{2}-2)\)
  vérifiant 
  \begin{enumerate}
  \item \(m_{i}> n_{i}\geq 0\) pour \(i\in \{1,2\}\);
  \item \(k_{1}=v_{p}(m_{1})\geq 1\) et \(k_{2}=v_{p}(n_{2}+2)\geq
    1\);
  \item \(m_{i}-n_{i}\geq p^{k_{i}}\) pour \(i\in \{1,2\}\).
  \end{enumerate}
  Alors
  \begin{equation}
    \label{eq:alpha1}
    \ch H^{2}(E_{\alpha}(\mu_{1}))=\ch H^{2}(\mu_{1})+\ch H^{2}(\mu_{1}-\alpha),
  \end{equation}
  \begin{equation}
    \label{eq:beta1}
    \ch H^{2}(E_{\beta}(\mu_{2}))=\ch H^{2}(\mu_{2})+\ch H^{2}(\mu_{2}-\beta).
  \end{equation}
  C'est-à-dire, les morphismes de bord sont nuls.
\end{proposition}

\begin{rmk}\label{rmk:degrérelatif}
  Fixons \( i\in\{1,2\} \), notons \( d \) le degré de \( \mu_{i} \).
  C'est-à-dire, \( ap^{d}\leq m_{i}\leq (a+1)p^{d} \) pour un \(
  a\in\{1,2,\cdots,p-1\}\). Si \( i=1 \), alors \( d\geq
  v_{p}(m_{1})=k_{1} \). Si \( i=2 \) et si \( k_{2}=v_{p}(n_{2}+2)>d \),
  alors \( n_{2}\geq p^{k_{2}}-2\geq p^{d+1}-2 \). Mais dans ce cas, on a
  \(m_{2}\geq n_{2}+p^{k_{2}}\geq 2p^{d+1}-2  \), absurde. Donc 
  on a toujours \(d\geq k_{i}\) pour \( i\in\{1,2\} \).
\end{rmk}

\begin{proof}
  Notons \( d_{i} \) le degré de \( \mu_{i} \). On appelle \(
  d_{i}-k_{i} \) le \og degré relatif\fg{} de \( \mu_{i} \) et on le
  note  \(\widetilde{d}_{i}\).  On montre la proposition
  simultanément pour \(E_{\alpha}\) et \(E_{\beta}\) par récurrence
  sur le degré relatif. D'après la \autoref{rmk:degrérelatif}, on sait
  que le degré relatif est toujours \( \geq 0 \).

  Si \(\widetilde{d}_{i}=0\), alors \( d_{i}=k_{i} \) et \(\mu_{1}=(ap^{d_{1}},-n_{1}-2)\) avec \(n_{1}\leq
  (a-1)p^{d_{1}}\) et \(\mu_{2}=(m_{2},-ap^{d_{2}})\) avec \(m_{2}\geq
  (a+1)p^{d_{2}}-2\).

  Dans ce cas, \(\mu_{1}\), \(\mu_{1}-\alpha\), \( \mu_{2} \) et \(
  \mu_{2}-\beta \) sont tous dans une \( H^{1} \)-chambre
  hors de la région de Griffith. En particulier,  \eqref{eq:alpha1} et  
  \eqref{eq:beta1} sont  triviales.

  Supposons qu'on ait déjà montré la proposition pour tout \(\mu_{i}\)
  tel
  que \(\widetilde{d}_{i}(\mu_{i})\leq\ell\) pour un certain \(\ell\geq 0\). Pour \(
  i\in\{1,2\} \), soit \( \mu_{i}=(m_{i},-n_{i}-2) \) tel que \( \widetilde{d}_{i}(\mu_{i})=\ell+1 \).

  On se concentre d'abord sur \(\mu=\mu_{1}\) et on enlève l'indice \(
  1 \) pour alléger la notation. Écrivons
  \(m=ap^{d}+a_{d-1}p^{d-1}+\cdots +a_{k}p^{k}=ap^{d}+r\) avec
  \(a\neq 0\), \(a_{k}\neq 0\) et \(d-k= \ell+1\geq 1\). Si
  \(\mu\notin \widehat{\Gr}\), alors \( n<ap^{d} \)
  , donc \(\mu-\alpha\notin \widehat{\Gr}\) aussi, car
  \(m\geq ap^{d}+p^{k}\geq ap^{d}+2\), et \eqref{eq:alpha1} est
  vraie dans ce cas. Donc il suffit de
  considérer le cas où \(\mu\in\widehat{\Gr}\),
  d'où \(n=ap^{d}+s\) avec \(0\leq s\leq r-p^{k}\). En
  particulier, on a \(1\leq r\leq p^{d}-1  \) et \( 0\leq s \leq
  p^{d}-2 \), donc  d'après la
 \autoref{thm:Edelta}, on a
  \begin{align*}
    \ch H^{2}(E_{\alpha}(\mu))=&\ch L(0,a-1)^{(d)}\ch
                                 H^{3}(E_{\alpha}(r-p^{d},-s-2))\\
                               &+\ch L(0,a)^{(d)}\ch H^{2}(E_{\alpha}(r,-s-2))\\
                               &+\ch
                                 L(0,a-2)^{(d)}\ch H^{2}(E_{\alpha}(-p^{d}+r,p^{d}-s-2)).
  \end{align*}
   
  Comme \(v_{p}(r)=v_{p}(p^{d}-r)=v_{p}(m)=k\) et
  \((p^{d}-s-2)-(p^{d}-r-2)=r-s\geq p^{k}\), le poids \((r,-s-2)\)
  vérifie les hypothèses pour \(E_{\alpha}\) dans la Proposition et
  est de  degré relatif  majoré par \( \ell \). Le poids \((p^{d}-s-2,-p^{d}+r)\) vérifie les
  hypothèses pour \(E_{\beta}\) et est de  degré relatif aussi majoré par \( \ell\). D'après l'hypothèse
  de récurrence on a donc
  \begin{align*}
     \ch H^{2}(E_{\alpha}(\mu))=&\ch L(0,a-1)^{(d)}(\ch
                                 H^{3}(r-p^{d},-s-2)+\ch H^{3}(r-p^{d}-2,-s-1))\\
                                &+\ch L(0,a)^{(d)}(\ch H^{2}(r,-s-2)
                                  +\ch H^{2}(r-2,-s-1))\\
     &+\ch  L(0,a-2)^{(d)}(\ch H^{2}(-p^{d}+r,p^{d}-s-2)+\ch
       H^{2}(-p^{d}+r-2,p^{d}-s-1))\\
    =&\ch H^{2}(ap^{d}+r,-ap^{d}-s-2)+\ch
       H^{2}(ap^{d}+r-2,-ap^{d}-s-1)\\
    =&\ch H^{2}(\mu)+\ch H^{2}(\mu-\alpha),
  \end{align*}
où la deuxième égalité résulte du \autoref{thm2}  et du fait que \(
0\leq r-2\leq p^{d}-3 \) car \( r\geq s+p^{k}\geq p^{k} \).

  Traitons maintenant \( E_{\beta}(\mu_{2}) \) et enlevons l'indice \(
  2 \) pour alléger la notation.  On a  \(v_{p}(n+2)=k\),
  \(m=ap^{d}+r\) avec \(a\geq 1\), \(0\leq r\leq p^{d}-1\) et
  \(d-k=\ell+1\geq 1\).

  Si \(\mu\notin\widehat{\Gr}\), alors \(n+2\leq ap^{d}+1\). Mais
  comme \(v_{p}(n+2)=k\), on a \(n+2\leq ap^{d}-p^{k}\). Dans ce cas
  \(\mu-\beta=(m+1,-n-4)\) n'est pas dans \( \widehat{\Gr} \) non plus.

  Si \(\mu\in \widehat{\Gr}\) mais \( \mu\notin\Gr \), c'est-à-dire \(r=p^{d}-1\), alors
  \(n=ap^{d}+s\) avec \(0\leq s\leq r-p^{k}\leq p^{d}-3\). Donc
  \(\mu-\beta=((a+1)p^{d},-ap^{d}-s-4)\notin\Gr\) , d'où
  \(H^{2}(\mu)=H^{2}(\mu-\beta)=0\).

  Donc il suffit de considérer  le cas où \(\mu\in\Gr\) et donc \(r\leq p^{d}-2\).
  Dans ce cas \(n=ap^{d}+s\) avec \(v_{p}(s+2)=k\) et \(p^{k}-2\leq
  s\leq r-p^{k}<p^{d}-3\). Alors le poids \((r,s)\) vérifie les hypothèses pour
  \(E_{\beta}\) et le poids \((p^{d}-s-2,-p^{d}+r)\) vérifie les
  hypothèses pour \(E_{\alpha}\), et ils sont de degrés relatifs majorés
  par \( \ell \), et les hypothèses pour l'existence d'une filtration
  à trois étages pour \( H^{2}(E_{\beta}(\mu)) \) sont
  vérifiées. Donc on a
  \begin{align*}
    \ch H^{2}(E_{\beta}(\mu))=&\ch L(0,a-1)^{(d)}\ch
                                 H^{3}(E_{\beta}(r-p^{d},-s-2))\\
                               &+\ch L(0,a)^{(d)}\ch H^{2}(E_{\beta}(r,-s-2))\\
                               &+\ch
                                 L(0,a-2)^{(d)}\ch
                                 H^{2}(E_{\beta}(-p^{d}+r,p^{d}-s-2))\\
    =&\ch L(0,a-1)^{(d)}(\ch H^{3}(r-p^{d},-s-2)+\ch
       H^{3}(r+1-p^{d},-s-4))\\
                              &+\ch L(0,a)^{(d)}(\ch H^{2}(r,-s-2)+\ch H^{2}(r+1,-s-4))\\
    &+\ch L(0,a-2)^{(d)}(\ch H^{2}(-p^{d}+r,p^{d}-s-2)+\ch
      H^{2}(-p^{d}+r+1,p^{d}-s-4))\\
    =&\ch H^{2}(ap^{d}+r,-ap^{d}-s-2)+\ch
    H^{2}(ap^{d}+r+1,-ap^{d}-s-4)\\
    =&\ch H^{2}(\mu)+\ch H^{2}(\mu-\beta),
  \end{align*}
où la première égalité est la filtration à trois étages pour \(
H^{2}(E_{\beta}(\mu)) \), la deuxième égalité résulte de l'hypothèse
de récurrence et du fait que \( H^{2}(r-p^{d},-s-2)=0 \), et la
troisième égalité résulte du \autoref{thm2} et du fait que \(
0<r+1\leq p^{d}-1 \) et \( 0<s+2<p^{d}-1 \). 

  Ceci termine la preuve de la \autoref{prop:bordnul1}.

\end{proof}

\subsubsection{Décomposition  de  l'image du morphisme de bord}

\begin{lemma}\label{lemma:filtimagebord1}
  Si \(\mu=(x,y)\) avec \(x,y\leq -1\), alors pour
  \(\delta\in\{\alpha,\beta\}\), on a :
  \[
\ch H^{3}(E_{\delta}(\mu))=\chi^{3}(\mu)+\chi^{3}(\mu-\delta)
\]
où \( \chi^{i}(\mu)=\ch H^{i}(\mu) \).
\end{lemma}
\begin{proof}
  Si \(x\leq -2\) et \(y\leq -2\), alors \(\mu-\delta\) n'a  de la
  cohomologie qu'en degré \(3\), 
   d'où le résultat.

  Si \(x=-1\) ou \(y=-1\), alors \(H^{i}(\mu)=0\) pour tout \(i\),
  donc \(H^{i}(E_{\delta}(\mu))\cong H^{i}(\mu-\delta)\) pour tout
  \(i\), donc le résultat est aussi vrai dans ce cas.
\end{proof}

\begin{defi}\label{defi:idelta}
  Pour \(\delta\in\{\alpha,\beta\}\), on note \(I_{\delta}(\mu)\subset
  H^{2}(\mu-\delta)\) l'image du morphisme de bord \(H^{1}(\mu)\to
  H^{2}(\mu-\delta)\). Donc si \(\mu-\delta \notin w_{0}\cdot X(T)^{+}\), on a
  \[
\ch I_{\delta}(\mu)=\chi^{2}(\mu)+\chi^{2}(\mu-\delta)-\ch H^{2}(E_{\delta}(\mu)).
  \]
\end{defi}

\begin{proposition}\label{cor:filtimage}
  Soit \(\mu=(ap^{d}+r,-ap^{d}-s-2)\) avec \( 1\leq a\leq p-1 \). Posons \(\mu'=(r,-s-2)\) et
  \(\mu''=(-p^{d}+r,p^{d}-s-2)\).

  Alors si  \(0\leq s< r\leq
  p^{d}-1\), on a
  \begin{equation}
    \label{eq:imageaubord1}
    \ch I_{\alpha}(\mu)=\ch L(0,a)^{(d)}\ch I_{\alpha}(\mu')+\ch L(0,a-2)^{(d)}\ch I_{\alpha}(\mu'').
  \end{equation}

  Si \( -1\leq s<r\leq p^{d}-2 \), alors
  \begin{equation}
    \label{eq:imageaubordbeta}
    \ch I_{\beta}(\mu)=\ch L(0,a)^{(d)}\ch I_{\beta}(\mu')+\ch L(0,a-2)^{(d)}\ch I_{\beta}(\mu'').
  \end{equation}
 
\end{proposition}

\begin{proof}
  Montrons d'abord \eqref{eq:imageaubord1} où \( \delta=\alpha
  \). Comme \( 0\leq s<r\leq p^{d}-1 \), \( \mu  \) vérifie les
  hypothèses pour l'existence de la filtration à trois étages pour
  \(H^{2}(E_{\alpha}(\mu))\) de la \autoref{thm:Edelta}. De plus, comme \( \mu'-\alpha=(r-2,-s-1)
  \) vérifie \( r-2\geq -1 \) et \(
  \mu''-\alpha=(-p^{d}+r-2,p^{d}-s-1) \) vérifie \( p^{d}-s-1\geq 0
  \), on a \( \mu'-\alpha\notin w_{0}\cdot X(T)^{+} \) et \(
  \mu''-\alpha\notin w_{0}\cdot X(T)^{+} \). Donc 
  en utilisant la filtration à trois étages pour \(
  H^{2}(E_{\alpha}(\mu)) \) pour la première égalité, et  le \autoref{lemma:filtimagebord1} et la \autoref{defi:idelta} pour la deuxième
  égalité, on a
  \begin{align*}
    \ch H^{2}(E_{\alpha}(\mu))=&\ch
                                 L(0,a-1)^{(d)}\ch H^{3}(E_{\alpha}(r-p^{d},-s-2))\\
                               &+\ch L(0,a)^{(d)}\ch H^{2}(E_{\alpha}(\mu'))+\ch
                                 L(0,a-2)^{(d)}\ch H^{2}(E_{\alpha}(\mu''))\\
    =&
                                                     \ch L(0,a-1)^{(d)}(\chi^{3}(r-p^{d},-s-2)+\chi^{3}((r-p^{d},-s-2)-\alpha))\\
                               &+\ch
                                 L(0,a)^{(d)}(\chi^{2}(\mu')+\chi^{2}(\mu'-\alpha)-\ch
                                 I_{\alpha}(\mu'))\\
                               &+\ch
                                 L(0,a-2)^{(d)}(\chi^{2}(\mu'')+\chi^{2}(\mu''-\alpha)-\ch
                                 I_{\alpha}(\mu''))\\
    =&\chi^{2}(\mu)+\chi^{2}(\mu-\alpha)-\ch L(0,a)^{(d)}\ch I_{\alpha}(\mu')-\ch L(0,a-2)^{(d)}\ch I_{\alpha}(\mu''),
  \end{align*}
où la dernière égalité résulte du \autoref{cor:2étages} en remarquant que
\( \mu-\alpha=(ap^{d}+r-2,-ap^{d}-s-1) \) avec \( -1\leq r-2, s-1<p^{d}-1
\).
  Donc on a
  \begin{align*}
    \ch L(0,a)^{(d)}\ch I_{\alpha}(\mu')+\ch L(0,a-2)^{(d)}\ch
    I_{\alpha}(\mu'')=& \chi^{2}(\mu)+\chi^{2}(\mu-\alpha)-\ch
                        H^{2}(E_{\alpha}(\mu))\\
    =&\ch I_{\alpha}(\mu),                       
  \end{align*}
  car \( \mu-\alpha\notin w_{0}\cdot X(T)^{+}\).

  Montrons maintenant \eqref{eq:imageaubordbeta} où \( \delta=\beta
  \).
   Comme \( -1\leq s<r\leq p^{d}-2 \), \( \mu  \) vérifie les
  hypothèses pour l'existence de la filtration à trois étages pour
  \(H^{2}(E_{\beta}(\mu))\) de la \autoref{thm:Edelta}. De plus, comme \( \mu'-\beta=(r+1,-s-4)
  \) vérifie \( r+1\geq 0 \) et \(
  \mu''-\beta=(-p^{d}+r+1,p^{d}-s-4) \) vérifie \( p^{d}-s-4\geq -1
  \), on a \( \mu'-\beta\notin w_{0}\cdot X(T)^{+} \) et \(
  \mu''-\beta\notin w_{0}\cdot X(T)^{+} \). Donc 
  en utilisant la filtration à trois étages pour \(
  H^{2}(E_{\beta}(\mu)) \) pour la première égalité, et  le \autoref{lemma:filtimagebord1} et la \autoref{defi:idelta} pour la deuxième
  égalité, on a
  \begin{align*}
    \ch H^{2}(E_{\beta}(\mu))=&\ch
                                 L(0,a-1)^{(d)}\ch H^{3}(E_{\beta}(r-p^{d},-s-2))\\
                               &+\ch L(0,a)^{(d)}\ch H^{2}(E_{\beta}(\mu'))+\ch
                                 L(0,a-2)^{(d)}\ch H^{2}(E_{\beta}(\mu''))\\
    =&
                                                     \ch L(0,a-1)^{(d)}(\chi^{3}(r-p^{d},-s-2)+\chi^{3}((r-p^{d},-s-2)-\beta))\\
                               &+\ch
                                 L(0,a)^{(d)}(\chi^{2}(\mu')+\chi^{2}(\mu'-\beta)-\ch
                                 I_{\beta}(\mu'))\\
                               &+\ch
                                 L(0,a-2)^{(d)}(\chi^{2}(\mu'')+\chi^{2}(\mu''-\beta)-\ch
                                 I_{\beta}(\mu''))\\
    =&\chi^{2}(\mu)+\chi^{2}(\mu-\beta)-\ch L(0,a)^{(d)}\ch I_{\beta}(\mu')-\ch L(0,a-2)^{(d)}\ch I_{\beta}(\mu''),
  \end{align*}
où la dernière égalité résulte du \autoref{cor:2étages} en remarquant que
\( \mu-\beta=(ap^{d}+r+1,-ap^{d}-s-4) \) avec \( 0< r+1, s+2\leq p^{d}-1
\).
  Donc on a
  \begin{align*}
    \ch L(0,a)^{(d)}\ch I_{\beta}(\mu')+\ch L(0,a-2)^{(d)}\ch
    I_{\beta}(\mu'')=& \chi^{2}(\mu)+\chi^{2}(\mu-\beta)-\ch
                        H^{2}(E_{\beta}(\mu))\\
    =&\ch I_{\beta}(\mu),                       
  \end{align*}
  car \( \mu-\beta\notin w_{0}\cdot X(T)^{+}\). Ceci termine la preuve
  de la \autoref{cor:filtimage}.
\end{proof}

\begin{lemma}\label{lemma:saturé}
  Soit \(\mu=(m,-n-2)\) avec \(m>n\geq 0\). Alors pour
  \(\delta\in\{\alpha,\beta\}\) et tout  \(L(\nu)\in
  \KF(I_{\delta}(\mu))\), on a 
\( [I_{\delta}(\mu):L(\nu)]=[H^{2}(\mu-\delta):L(\nu)] \).
\end{lemma}

\begin{proof}
  Comme dans la \autoref{prop:bordnul1},
  notons \( k_{1}(\mu)=v_{p}(m) \),  \( k_{2}(\mu)=v_{p}(n+2)
  \). Notons \(
  d \) le degré de \( \mu \), c'est-à-dire, il existe \( a\in
  \{1,2,\cdots,p-1\} \) tel que \( ap^{d}\leq m<(a+1)p^{d} \). Pour \( i\in\{1,2\} \), notons \(
  \widetilde{d}_{i}=d-k_{i} \). Notons aussi \( \alpha_{1}=\alpha \) et \(
  \alpha_{2}=\beta \).

  Considérons l'énoncé suivant qui dépend d'un indice \( \ell \in\mathbb{Z}\):
  \begin{align*}
    \mathcal{P}_{\ell}:&\text{ pour tout } i\in\{1,2\}, \text{ si
    } \mu=(m,-n-2) \text{ avec }m>n\geq 0 \text{ et
                                              }\widetilde{d}_{i}(\mu)\leq
                                              \ell,\\ &\text{ alors pour
                                                         tout } \nu\in X(T)^{+}
                      \text{on a }[I_{\delta}(\mu):L(\nu)]=[H^{2}(\mu-\delta):L(\nu)]. 
  \end{align*}

  Le but est de montrer que \( \mathcal{P}_{\ell} \) est vraie pour tout
  \( \ell \). Raisonnons par récurrence sur \( \ell \).  Il suffit de considérer le cas où \(\mu\) ne vérifie pas les hypothèses
  de la \autoref{prop:bordnul1} car l'énoncé est trivial si \(
  I_{\delta}(\mu)=0 \). 

  D'après la définition de \( d \) et \( k_{1} \), on a toujours \(
  d\geq k_{1} \), d'où \( \widetilde{d}_{1}\geq 0 \). Comme \( m>n \) et
  \( k_{2}=v_{p}(n+2) \), on a  \( d<k_{2} \) seulement s'il existe \(
  d\geq 1 \) tel que \( m=p^{d}-1 \) et \( n=p^{d}-2 \). Dans ce cas,
  on a 
  \[H^{2}(E_{\beta}(\mu))=H^{2}(E_{\beta}(p^{d}-1,-p^{d}))=0\]
  d'après la \autoref{prop:Eparticulier}, d'où \( I_{\beta}(\mu)=H^{2}(\mu-\beta) \)
 et l'énoncé est vrai. Donc \( \mathcal{P}_{-1} \) est vrai.

  Supposons \(\widetilde{d}_{1}(\mu)=0\) et \(\delta=\alpha\). Si \(
  k_{1}=0 \), alors \( p\nmid m \), d'où \( I_{\alpha}(\mu)=0
  \) car \( \mu-\alpha\notin W_{p}\cdot \mu \). Si \( k_{1}\geq 1 \),  alors comme \( \mu
  \) ne vérifie pas les hypothèses de la \autoref{prop:bordnul1}, on a
  \(\mu=(ap^{d},-(a-1)p^{d}-s-2)\) avec \(1\leq s\leq
  p^{d}-1\). D'après la \autoref{prop:Eparticulier}
   on sait que \(H^{2}(E_{\alpha}(\mu))=0\), d'où
  \(I_{\alpha}(\mu)=H^{2}(\mu-\alpha)\) et l'énoncé du lemme est
  évident.

  Supposons \(\widetilde{d}_{2}(\mu)=0\). Si \(
  k_{2}=0 \), alors \( p\nmid n+2 \) et \( \mu-\beta\notin
  W_{p}\cdot\mu \), d'où \( I_{\beta}(\mu)=0 \). Si \( k_{2}\geq 1 \),
  alors comme \( \mu \) ne vérifie pas les hypothèses de la
  \autoref{prop:bordnul1}, on a  \(\mu=(ap^{d}+r,-ap^{d})\)
  avec \(-1\leq r\leq p^{d}-3\) et \( d-v_{p}(m)=d-k_{1}=\ell+1 \). D'après la \autoref{prop:Eparticulier},
  \(H^{2}(E_{\beta}(\mu))=0\), d'où
  \(I_{\beta}(\mu)=H^{2}(\mu-\beta)\) et l'énoncé du lemme est
  évident.

  Donc \( \mathcal{P}_{0} \) est vraie. 

   Supposons que \( \mathcal{P}_{\ell} \) est vraie pour un \( \ell\geq
   0\). Soit \( \mu \) tel que \( \widetilde{d}_{1}(\mu)=\ell+1
   \). Si \( k_{1}=0 \), alors \( p\nmid m \) et \(
   I_{\alpha}(\mu)=0 \). Si \( k_{1}\geq 1 \), comme \(\mu  \) ne
   vérifie pas les hypothèses de la \autoref{prop:bordnul1}, on a
  \[m=ap^{d}+a_{d-1}p^{d-1}+\cdots +a_{k}p^{k}=ap^{d}+r\] et
  \(n=ap^{d}+s\) avec \(0\leq r-p^{k}<s<r\leq p^{d}-1\). Posons
  \(\mu'=(r,-s-2)\), 
  \(\mu''=(-p^{d}+r,p^{d}-s-2)\) et \( ^{t}\mu''=(p^{d}-s-2,-p^{d}+r)
  \). Comme \[ v_{p}(ap^{d}+r)=k_{1}=d-\ell-1\leq d-1, \] on a \[
    v_{p}(r)=v_{p}(-p^{d}+r)=k_{1}, \]  donc \(
  \widetilde{d}_{1}(\mu')\leq d-1-k_{1}=
  \ell \) et \( \widetilde{d}_{2}(^{t}\mu'')\leq d-1-k_{1}=\ell \).
On a 
  \[\ch I_{\alpha}(\mu)=\ch L(0,a)^{(d)}\ch I_{\alpha}(\mu')+\ch
    L(0,a-2)^{(d)}\ch I_{\alpha}(\mu'')\]
  car \( 0<s<r\leq p^{d}-1 \). 
  D'après le \autoref{lemma:poidsrestreints}, tout plus haut poids d'un
  facteur de composition de \( H^{2}(\mu'-\alpha) \) ou de \(
  H^{2}(\mu''-\alpha) \) est \( p^{d} \)-restreint. Donc
  \[
\KF(I_{\alpha}(\mu))=L(0,a)^{(d)}\otimes\KF(I_{\alpha}(\mu'))\amalg
L(0,a-2)^{(d)}\otimes \KF(I_{\alpha}(\mu'')).
\] 
Soit \(L(\nu)\in \KF(I_{\alpha}(\mu))\), alors
\(\nu=\nu^{1}p^{d}+\nu^{0}\) où \(\nu^{1}=(0,a)\) ou \((0,a-2)\) et
\(\nu^{0}\) est \(p^{d}\)-restreint. Si \(\nu^{1}=(0,a)\), alors
\(L(\nu^{0})\in \KF(I_{\alpha}(\mu'))\). Donc
\begin{multline*}
  [I_{\alpha}(\mu):L(\nu)]=[I_{\alpha}(\mu'):L(\nu^{0})]
                          =[H^{2}(\mu'-\alpha):L(\nu^{0})]\\
                          =[L(0,a)^{(d)}\otimes H^{2}(\mu'-\alpha):L(\nu)]
  =[H^{2}(\mu-\alpha):L(\nu)]
\end{multline*}
où la deuxième égalité résulte de l'hypothèse de récurrence pour
\(\mu'\) et la dernière égalité résulte du \autoref{thm2} et du
\autoref{lemma:poidsrestreints} 
appliqués à \( \mu-\alpha \).

Si \(\nu^{1}=(0,a-2)\), alors \( L(\nu^{0})\in \KF(I_{\alpha}(\mu''))
\). Posons \( \tau\nu=(y,x) \) si \( \nu=(x,y) \) comme dans le paragraphe  \ref{subsection:surlemur}, alors
\begin{multline*}
  [I_{\alpha}(\mu):L(\nu)]=[I_{\alpha}(\mu''):L(\nu^{0})]
  =[I_{\beta}(\tau\mu''):L(\tau\nu^{0})]
   =[H^{2}(\tau\mu''-\beta):L(\tau\nu^{0})]\\
                          =[H^{2}(\mu''-\alpha):L(\nu^{0})]
                          =[L(0,a-2)^{(d)}\otimes H^{2}(\mu''-\alpha):L(\nu)]
  =[H^{2}(\mu-\alpha):L(\nu)].
\end{multline*}
Donc la partie \( i=1 \) dans \( \mathcal{P}_{\ell} \) est vraie. 

 Soit \( \mu=(m,-n-2) \) tel que \(
 \widetilde{d}_{2}(\mu)=\ell+1 \). Notons \( k=k_{2} \) pour alléger la notation. Si \( k=0 \), alors \( p\nmid n+2
 \) et \( I_{\beta}(\mu)=0 \). Si \( k\geq 1 \),  comme \( \mu \)
 ne vérifie pas les hypothèses de la \autoref{prop:bordnul1}, alors
  \(m=ap^{d}+r\)  et \(n=ap^{d}+s\) avec \( 0\leq r\leq p^{d}-1 \) et \(s<r<s+p^{k}\) (a priori \( s
 \) peut être négatif). Mais comme \( d=k+\ell+1\geq k+1 \), on a \(
 k=v_{p}(n+2)=v_{p}(s+2) \). Si \( s<0 \), alors \(s+2\leq 1  \), d'où
 \( s+2\leq -p^{k} \) car \( v_{p}(s+2)=k\geq 1 \). Par conséquent, on
 a \( r<s+p^{k}\leq -2 \), contradiction avec \(
 r\geq 0 \). Donc \( 0\leq s<r\leq p^{d}-1 \). Or \( v_{p}(s+2)=k\leq
 d-1 \), donc \( s+2\leq p^{d}-p^{k} \) et \( r<s+p^{k}\leq p^{d}-2
 \). D'après la \autoref{cor:filtimage}, on a
  \[\ch I_{\beta}(\mu)=\ch L(0,a)^{(d)}\ch I_{\beta}(\mu')+\ch
    L(0,a-2)^{(d)}\ch I_{\beta}(\mu'')\]
  où \( \mu'=(r,-s-2) \) et \( \mu''=(-p^{d}+r,p^{d}-s-2) \). Posons
  \( \tau\mu''=(p^{d}-s-2,-p^{d}+r) \), alors \(
  \widetilde{d}_{2}(\mu')\leq d-1-k=\ell \) et \(
  \widetilde{d}_{1}(\tau\mu'')\leq d-1-k=\ell \) car \( v_{p}(s+2)=v_{p}(p^{d}-s-2)=k \).

  D'après le \autoref{lemma:poidsrestreints}, tout plus haut poids d'un
  facteur de composition de \( H^{2}(\mu'-\beta) \) ou de \(
  H^{2}(\mu''-\beta) \) est \( p^{d} \)-restreint. Donc
  \[
\KF(I_{\beta}(\mu))=L(0,a)^{(d)}\otimes\KF(I_{\beta}(\mu'))\amalg
L(0,a-2)^{(d)}\otimes \KF(I_{\beta}(\mu'')).
\] 
Soit \(L(\nu)\in \KF(I_{\beta}(\mu))\), alors
\(\nu=\nu^{1}p^{d}+\nu^{0}\) où \(\nu^{1}=(0,a)\) ou \((0,a-2)\) et
\(\nu^{0}\) est \(p^{d}\)-restreint. Si \(\nu^{1}=(0,a)\), alors
\(L(\nu^{0})\in \KF(I_{\beta}(\mu'))\). Donc
\begin{multline*}
  [I_{\beta}(\mu):L(\nu)]=[I_{\beta}(\mu'):L(\nu^{0})]
                          =[H^{2}(\mu'-\beta):L(\nu^{0})]\\
                          =[L(0,a)^{(d)}\otimes H^{2}(\mu'-\beta):L(\nu)]=[H^{2}(\mu-\beta):L(\nu)]
\end{multline*}
où la deuxième égalité résulte de l'hypothèse de récurrence pour
\(\mu'\) et la dernière égalité résulte du \autoref{thm2} et  du
\autoref{lemma:poidsrestreints} 
appliqués à \( \mu-\beta \).

Si \(\nu^{1}=(0,a-2)\), alors \( L(\nu^{0})\in \KF(I_{\beta}(\mu''))
\). Donc
\begin{multline*}
  [I_{\beta}(\mu):L(\nu)]=[I_{\beta}(\mu''):L(\nu^{0})]
  =[I_{\alpha}(\tau\mu''):L(\tau\nu^{0})]
   =[H^{2}(\tau\mu''-\alpha):L(\tau\nu^{0})]\\
                    =[H^{2}(\mu''-\beta):L(\nu^{0})]
                          =[L(0,a-2)^{(d)}\otimes H^{2}(\mu''-\beta):L(\nu)]
  =[H^{2}(\mu-\beta):L(\nu)].
\end{multline*}

Donc \( \mathcal{P}_{\ell+1} \) est vraie. 
Ceci termine la preuve du \autoref{lemma:saturé}.

\end{proof}

\begin{thm}\label{cor:imagesum}
 Soit \( \mu=(m,-n-2) \) avec \( m>n\geq 0 \). Si \(M\) est un sous-module de
  \(H^{2}(\mu-\delta)\) qui vérifie  \(\ch M=\ch I_{\delta}(\mu)\),
  alors \(M=I_{\delta}(\mu)\).

 Par conséquent, si \(m=ap^{d}+r\) et \(n=ap^{d}+s\) et si l'on pose
    \(\mu'=(r,-s-2)\) et \(\mu''=(-p^{d}+r,p^{d}-s-2)\), alors
  \begin{enumerate}[(i)]
  \item Si \( 0\leq s<r\leq p^{d}-1 \) , on a
   \( I_{\alpha}(\mu)=L(0,a)^{(d)}\otimes I_{\alpha}(\mu')\bigoplus
      L(0,a-2)^{(d)}\otimes I_{\alpha}(\mu'') \).
  \item Si \( -1\leq s<r\leq p^{d}-2 \), on a 
      \( I_{\beta}(\mu)=L(0,a)^{(d)}\otimes I_{\beta}(\mu')\bigoplus
      L(0,a-2)^{(d)}\otimes I_{\beta}(\mu'') \).
  \end{enumerate}
\end{thm}

\begin{proof}
  Il suffit d'appliquer  le \autoref{lemma:saturé} et la \autoref{cor:filtimage}.
\end{proof}

\subsubsection{\( I_{\delta}(\mu) \) est sans multiplicité}

\begin{proposition}\label{prop:sansmultiplicité}
  Soit \(\mu=(m,-n-2)\) avec \(m>n\geq 0\). Alors pour \(\delta\in
  \{\alpha,\beta\}\),
  \(I_{\delta}(\mu)\) est un \(T\)-module sans multiplicité. C'est-à-dire, pour tout
  poids \(\nu\in X\), on a \(\dim(I_{\delta}(\mu)_{\nu})\leq 1\).
\end{proposition}

Avant de montrer cette proposition, on montre d'abord le lemme utile suivant
:
\begin{lemma}\label{lemma:H2sansmultiplicité}
Soit \( \mu=(ap^{d}+p^{d}-2,-ap^{d}-s-1) \) avec \( d\geq 0 \), \(
a\in\{1,2,\cdots,p-1\} \) et \( s\leq p^{d}-1 \) (\( s \) n'est pas
nécessairement positif). Alors \( H^{2}(\mu) \) est un \( T \)-module
sans multiplicité. 
\end{lemma}

\begin{proof}
 
  Raisonnons par récurrence sur \( d \). Si \( d=0 \), alors \(
  \mu=(a-1,-a-s-1) \) avec \( s\leq 0 \). Donc \( H^{2}(\mu)=0 \)
  d'après la \autoref{rmk:Griffithcondition}. Supposons que \(
  \mu=(ap^{d+1}+p^{d+1}-2,-ap^{d+1}-s-1) \) pour un \( d\geq 0 \) et
  \( s\leq p^{d+1}-1 \). Si
  \( s\leq 0 \), alors \( H^{2}(\mu)=0 \). Si \( s>0 \),  alors d'après le
  \autoref{thm2}, on sait que \( H^{2}(\mu) \) est filtré par
    \(E_{1}=L(0,a-1)^{(d+1)}\otimes V(s-1,0)\) et
    \(E_{2}=L(0,a)^{(d+1)}\otimes H^{2}(p^{d+1}-2,-s-1)\) car \( H^{2}(\mu'')=H^{2}(-2,p^{d+1}-s-1)=0 \). Comme tout poids de
    \(V(s-1,0)\) et de \(H^{2}(p^{d+1}-2,-s-1)\) est
    \(p^{d+1}\)-restreint, et comme \(L(0,a)\) et \(L(0,a-1)\) n'ont pas de
    poids commun, \(E_{1}\) et \(E_{2}\) n'ont pas de poids commun. D'après
    l'hypothèse de récurrence, \(H^{2}(p^{d+1}-2,-s-1)=H^{2}((p-1)p^{d}+p^{d}-2,-(p-1)p^{d}-(s-(p-1)p^{d})-1)\) n'a pas de
    multiplicité comme  \(T\)-module car \( s-(p-1)p^{d}\leq
    p^{d+1}-1-p^{d+1}+p^{d}=p^{d}-1 \).
    On sait aussi que \(V(s-1,0)\) n'a pas de
    multiplicité comme  \(T\)-module. Par conséquent, \(H^{2}(\mu)\) n'a pas de
    multiplicité non plus.
\end{proof}

Démontrons maintenant la \autoref{prop:sansmultiplicité}.
\begin{proof}
Comme \( m\geq 1 \), il existe \( d\geq 0 \) et \(
a\in\{1,2,\cdots,p-1\} \) tels que
\( ap^{d}\leq m<(a+1)p^{d} \).

Écrivons \( m=ap^{d}+r \) et \( n=ap^{d}+s \), alors \( 0\leq r\leq
p^{d}-1 \) et \( s<r \) (\( s \) peut être négatif).

Raisonnons par récurrence sur \( d \). Si \( d=0 \), alors \(
\mu=(a,-a-s-2) \) avec \( s\leq -1 \). Donc \( \mu-\alpha=(a-2,-a-s-1) \)
et \( \mu-\beta=(a+1,-a-s-4) \), d'où \( H^{2}(\mu-\delta)=0 \) pour
tout \( \delta\in\{\alpha,\beta\} \). Par conséquent, \(
I_{\delta}(\mu)=0 \) car \( I_{\delta}(\mu)\subset H^{2}(\mu-\delta)
\) et l'énoncé est trivial.

Maintenant supposons \( m=ap^{d+1}+r \) et \( n=ap^{d+1}+s \) avec \(
0\leq r\leq p^{d+1}-1 \) et \( s<r \). 

Supposons d'abord que \( \delta=\alpha \). \fbox{Si \( s\leq 0 \) et
  \( r\geq 1 \)}, alors  \[
H^{2}(\mu-\alpha)=H^{2}(ap^{d+1}+r-2,-ap^{d+1}-s-1)=0 \] car \(
ap^{d+1}+r-2\geq ap^{d+1}-1 \) et \( ap^{d+1}+s-1\leq ap^{d+1}-1 \). Donc \(
I_{\alpha}(\mu)=0 \) et le résultat est trivial.

\fbox{Si \( s\leq 0 \) et \( r=0 \)}, alors \( s\leq -1 \) car \( s<r
\). Donc on a \( \mu-\alpha=(ap^{d+1}-2,-ap^{d+1}-s-1) \) et \(
H^{2}(\mu-\alpha) \) n'a pas de multiplicité comme  \( T \)-module
d'après le \autoref{lemma:H2sansmultiplicité}. Donc l'énoncé est vrai
car \( I_{\alpha}(\mu)\subset H^{2}(\mu-\alpha)  \).

\fbox{Si \( s>0 \)}, alors \( 0<s<r\leq p^{d+1}-1 \), et d'après la
\autoref{cor:filtimage}, on a
\begin{displaymath}
 \ch I_{\alpha}(\mu)=\ch L(0,a)^{(d+1)}\ch I_{\alpha}(\mu')+\ch L(0,a-2)^{(d+1)}\ch I_{\alpha}(\mu'') ,
\end{displaymath}
où \( \mu'=(r,-s-2) \) et \( \mu''=(-p^{d+1}+r,p^{d+1}-s-2) \). 
Comme \((0,2)\notin \mathbb{Z}\alpha+\mathbb{Z}\beta\), alors
\(L(0,a)\) et \(L(0,a-2)\) n'ont pas de poids commun. D'après le
\autoref{lemma:poidsrestreints}, tout poids de \(
I_{\alpha}(\mu')\subset H^{2}(\mu'-\alpha) \)
et de \( I_{\alpha}(\mu'') \subset H^{2}(\mu''-\alpha)\) est \( p^{d+1} \)-restreint, donc
\(L(0,a)^{(d+1)}\otimes  I_{\alpha}(\mu')  \) et \(
L(0,a-2)^{(d+1)}\otimes I_{\alpha}(\mu'') \) n'ont pas de poids
commun. Par conséquent,  \(
I_{\alpha}(\mu) \) n'a pas de multiplicité comme  \( T \)-module car
\( I_{\alpha}(\mu') \) et \( I_{\alpha}(\mu'') \) n'ont pas de
multiplicité d'après l'hypothèse de récurrence.

Supposons maintenant que \( \delta =\beta\). \fbox{Si \( s\leq -3 \)},
alors \( \mu-\beta=(ap^{d+1}+r+1,-ap^{d+1}-s-4) \) avec \( r+1\geq 1
\) et \( s+2\leq -1 \), d'où \( H^{2}(\mu-\beta)=0 \). En particulier,
\( I_{\beta}(\mu)=0 \) et l'énoncé est trivial.

\fbox{Si \( s=-2 \)}, alors \( \mu=(ap^{d+1}+r,-ap^{d+1}) \) avec \(
r\geq 0 \). Donc \( H^{2}(E_{\beta}(\mu))=0 \) d'après la
\autoref{prop:Eparticulier}, et par conséquent on a
\begin{displaymath}
  I_{\beta}(\mu)=H^{2}(\mu-\beta)=H^{2}(ap^{d}+r+1,-ap^{d}-2).
\end{displaymath}
Si \( r=p^{d}-1 \), alors \( H^{2}(ap^{d}+r+1,-ap^{d}-2)=0 \) et
l'énoncé est trivial. Si \( r\leq p^{d}-2 \), alors d'après le
\autoref{thm2}, on sait que  \( H^{2}(ap^{d}+r+1,-ap^{d}-2) \) est un
quotient de \( V(0,ap^{d}-r-3) \) car \( H^{2}(r+1,-2)=0 \). Comme \(
V(0,ap^{d}-r-3) \) n'a pas de multiplicité, l'énoncé est vrai dans ce
cas.

\fbox{Si \( s=-1 \)}, alors \( p\nmid ap^{d+1}+s+2 \), donc \(
\mu-\beta\notin W_{p}\cdot \mu \) et en particulier on a \(
I_{\beta}(\mu)=0 \).

\fbox{Si \( 0\leq s\leq p^{d+1}-3 \) et \( r=p^{d+1}-1 \)}, alors on a
\begin{displaymath}
H^{2}(\mu-\beta)=H^{2}((a+1)p^{d+1},-ap^{d+1}-s-4)=0
\end{displaymath}
car \( s+2\leq p^{d+1}-1 \), d'où \( I_{\beta}(\mu)=0 \).

\fbox{Si \( s=p^{d+1}-2 \) et \( r=p^{d+1}-1 \)}, alors
\begin{displaymath}
 H^{2}(E_{\beta}(\mu))=H^{2}(E_{\beta}((a+1)p^{d+1}-1,-(a+1)p^{d+1}))=0
\end{displaymath}
d'après la \autoref{prop:Eparticulier}. Donc
\begin{displaymath}
  I_{\beta}(\mu)=H^{2}(\mu-\beta)=H^{2}((a+1)p^{d+1},-(a+1)p^{d+1}-2)
\end{displaymath}
qui est un quotient de \( V(0,(a+1)p^{d+1}-2) \) d'après le
\autoref{thm2} car \( H^{2}(0,-2)=0 \). Comme \( V(0,(a+1)p^{d+1}-2)
\) n'a pas de multiplicité comme  \( T \)-module, le résultat en
découle.

\fbox{Si \( 0\leq s<r\leq p^{d+1}-2 \)}, alors d'après le
\autoref{cor:filtimage} on a
\begin{displaymath}
 \ch I_{\beta}(\mu)=\ch L(0,a)^{(d+1)}\ch I_{\beta}(\mu')+\ch L(0,a-2)^{(d+1)}\ch I_{\beta}(\mu'') ,
\end{displaymath}
où \( \mu'=(r,-s-2) \) et \( \mu''=(-p^{d+1}+r,p^{d+1}-s-2) \). 
Comme \((0,2)\notin \mathbb{Z}\beta+\mathbb{Z}\beta\), alors 
\(L(0,a)\) et \(L(0,a-2)\) n'ont pas de poids commun. D'après le
\autoref{lemma:poidsrestreints}, tout poids de \(
I_{\beta}(\mu')\subset H^{2}(\mu'-\beta) \)
et de \( I_{\beta}(\mu'')\subset H^{2}(\mu''-\beta) \) est \( p^{d+1} \)-restreint, donc
\(L(0,a)^{(d+1)}\otimes  I_{\beta}(\mu')  \) et \(
L(0,a-2)^{(d+1)}\otimes I_{\beta}(\mu'') \) n'ont pas de poids
commun. Par conséquent,  \(
I_{\beta}(\mu) \) n'a pas de multiplicité comme  \( T \)-module car
\( I_{\beta}(\mu') \) et \( I_{\beta}(\mu'') \) n'ont pas de
multiplicité d'après l'hypothèse de récurrence.
Ceci termine la preuve de la \autoref{prop:sansmultiplicité}.
\end{proof}

\subsection{Retour à  la \( p \)-\( H^{i}
  \)-D-filtration}\label{subsection:revenir}

Le but de ce paragraphe est d'écrire en détail la \( p \)-\( H^{i}
  \)-D-filtration où \( i\in\{1,2\} \) et \( \mu\notin C\cup
  w_{0}\cdot C \). On verra aussi que le
\autoref{cor:imagesum} est suffisant pour décrire tous les modules
inconnus de la forme \(H^{i}(E_{\delta}(\nu))\) dans le
\autoref{thm:pHifiltration}. Supposons maintenant que \( \mu\notin C\cup w_{0}\cdot C \). Alors il
existe \( m,n\in\mathbb{N} \) tels que \(\mu=(m,-n-2)  \) ou \(
\mu=(-n-2,m) \). D'après la symétrie entre \( \alpha \) et \( \beta
\), on peut supposer que \( \mu=(m,-n-2) \) sans perte de 
généralité. D'après la dualité de Serre, il suffit de considérer \(
H^{1}(\mu)=H^{1}(m,-n-2) \) et \( H^{2}(\mu)=H^{2}(m,-n-2) \) lorsque \(
m\geq n \)
(c'est-à-dire, \( \mu\in s_{\beta}\cdot C \)).

Si \(n\leq  m\leq p-1 \), alors \( H^{2}(m,-n-2)=0 \) et
\[
H^{1}(m,-n-2)\cong H^{0}(s_{\beta}\cdot\mu)=H^{0}(m-n-1,n)
\]
d'après le théorème de Borel-Weil-Bott (cf. \cite{Jan03} II.5.5).

Si \( m\geq p \), alors il existe \( d\geq 1 \) et \(
a\in\{1,2,\cdots,p-1\} \) tels que \( ap^{d}\leq m<(a+1)p^{d} \).
Écrivons \( m=ap^{d}+Rp+r \)  et \( n=ap^{d}+Sp+s \) avec \( 0\leq
r,s\leq p-1 \) (\( S \) peut être négatif mais \( S\geq -ap^{d-1} \)
car \( n\geq 0 \)), alors on a \( 0\leq R\leq
p^{d-1}-1 \) et \( S\leq R \). Notons \( m^{1}=ap^{d-1}+R \) et \(
n^{1}=ap^{d-1}+S \).

Pour  \( \nu=p\nu^{1}+\nu^{0} \) où \( \nu^{0}\in X_{1}(T) \), posons
\[ \mathcal{H}_{\delta}^{i}(\nu)=L(\nu^{0})\otimes
  H^{i}(E_{\delta}(\nu^{1}))^{(1)} \]
où \( \delta\in\{0,\alpha,\beta\} \). Notons aussi \(
\mathcal{H}^{i}(\nu)=\mathcal{H}_{0}^{i}(\nu) \).

\begin{rmk}
  En utilisant les résultats de ce paragraphe, on peut obtenir une
  autre démonstration de la proposition de Kühne-Hausmann \cite{KH84}
  6.3.2 (voir aussi \cite{DS88} 5.3) et préciser les conditions pour
  que \( \lambda \) soit \og générique\fg{}.
\end{rmk}

\subsubsection{Type \( \Delta \)}
Supposons que \( \mu \) est de type \( \Delta \), c'est-à-dire \(
0\leq s<r\leq p-2 \). Les neuf facteurs simples de \( \widehat{Z}(\mu) \)
sont donnés par la figure suivante  (où \( \nu_{1}=\mu \))
:
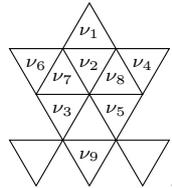
\begin{figure}[H]
  \centering
  \begin{tikzpicture}[scale=0.7]
    \draw (0,0)--(3,0)--(1.5,-3*sin{60})--cycle; \draw
    (0.5,-sin{60})--(1.5,sin{60})--(2.5,-sin{60})--cycle; \draw
    (1,0)--(2.5,-3*sin{60})--(3,-2*sin{60})--(0,-2*sin{60})--(0.5,-3*sin{60})--(2,0);
    \node[font=\tiny] at (1.5,0.3){\( \nu_{1} \)}; \node[font=\tiny] at (0.5,-0.3){\( \nu_{6} \)}; \node[font=\tiny] at
    (1.5,-0.3){\( \nu_{2} \)}; \node[font=\tiny] at (2.5,-0.3){\( \nu_{4} \)}; \node[font=\tiny] at
    (1,-sin{60}+0.3){\( \nu_{7} \)}; \node[font=\tiny] at (2,-sin{60}+0.3){\( \nu_{8} \)}; \node[font=\tiny] at
    (1,-sin{60}-0.3){\( \nu_{3} \)}; \node[font=\tiny] at (2,-sin{60}-0.3){\( \nu_{5} \)}; \node[font=\tiny] at
    (1.5,-2*sin{60}-0.3){\( \nu_{9} \)};
  \end{tikzpicture}.
\end{figure}

D'après
le \autoref{thm:pHifiltration}, on sait
que pour \( i\in\{1,2\} \), il existe une filtration de  \( H^{i}(\mu) \)
dont les quotients sont les suivants (l'ordre peut être différent)
\begin{displaymath}
  \mathcal{H}^{i}(\nu_{1}),\mathcal{H}^{i}(\nu_{2}),\mathcal{H}_{\alpha}^{i}(\nu_{4}),
  \mathcal{H}_{\beta}^{i}(\nu_{6}), \mathcal{H}^{i}(\nu_{7}),\mathcal{H}^{i}(\nu_{8}),\mathcal{H}^{i}(\nu_{9}).
\end{displaymath}

On sait que \( H^{0}(\nu_{4}^{1})=H^{0}(m^{1}+1,-n^{1}-2)=0 \) et \(
H^{3}(\nu_{3}^{1})=H^{3}(m^{1}-1,-n^{1}-1)=0 \) car \( m^{1},n^{1}\geq
0\), donc \( \mathcal{H}^{0}(\nu_{4})=\mathcal{H}^{3}(\nu_{3})=0 \).
Donc il existe une suite exacte longue
\begin{equation}
  \label{eq:eef5a3485d232fdb}
    0\to \mathcal{H}^{1}(\nu_{3})\to\mathcal{H}_{\alpha}^{1}(\nu_{4})\to
    \mathcal{H}^{1}(\nu_{4})\xrightarrow{\partial_{\alpha}}
    \mathcal{H}^{2}(\nu_{3})\to\mathcal{H}_{\alpha}^{2}(\nu_{4})\to
    \mathcal{H}^{2}(\nu_{4})\to
    0.
\end{equation}

De même, comme \( H^{3}(\nu_{5}^{1})=H^{3}(m^{1},-n^{1}-2)=0 \), on a
une suite exacte longue
\begin{equation}
  \label{eq:a1aac2f2b9fe0e5e}
    \cdots\to
    \mathcal{H}^{0}(\nu_{6})\xrightarrow{\partial_{\beta}^{0}}\mathcal{H}^{1}(\nu_{5})\to
    \mathcal{H}_{\beta}^{1}(\nu_{6})\to 
    \mathcal{H}^{1}(\nu_{6})\xrightarrow{\partial_{\beta}^{1}}
   \mathcal{H}^{2}(\nu_{5})\to\mathcal{H}_{\beta}^{2}(\nu_{6})\to
    \mathcal{H}^{2}(\nu_{6})\to 0.
\end{equation}

\fbox{Si \( n^{1}=0 \),} alors \( H^{2}(\mu)=0 \) car \( n\leq p-1 \)
et \( m\geq n \). On a aussi \( \mathcal{H}^{2}(\nu_{3})=0 \) car \(
\nu_{3}^{1}=(m^{1}-1,-n^{1}-1)=(m^{1}-1,-1) \). Donc d'après
\eqref{eq:eef5a3485d232fdb}, on sait que \( \mathcal{H}_{\alpha}^{1}(\nu_{4})
\) est juste une extension de \( \mathcal{H}^{1}(\nu_{4}) \) par \(
\mathcal{H}^{1}(\nu_{3}) \).

Or on a \( H^{i}(E_{\beta}(\nu_{6}^{1}))=H^{i}(E_{\beta}(m^{1}-1,0))=0
\) pour tout \( i \) d'après le \autoref{lemmaE0}, donc d'après
\eqref{eq:a1aac2f2b9fe0e5e}, \( \partial_{\beta}^{0} \) induit un isomorphisme
\( \mathcal{H}^{0}(\nu_{6})\cong\mathcal{H}^{1}(\nu_{5}) \).
Par conséquent, non seulement le facteur \( \mathcal{H}^{1}(\nu_{6})
\) n'apparaît pas, mais le facteur \( \mathcal{H}^{1}(\nu_{5}) \) est \og
effacé\fg{} dans la filtration de \( H^{1}(\mu) \), c'est-à-dire,  le
\( G \)-module \(
H^{1}(\mu) \) admet une filtration dont les quotients sont
\( \{\mathcal{H}^{1}(\nu_{i})|i=1,2,3,4,7,8,9\} \).La
situation est visualisée par la 
figure suivante, où la droite en gras est le mur entre \( C \) et \(
s_{\beta}\cdot C \), i.e. \( \{\mu\in
X(T)|\langle\mu+\rho,\beta\rangle=0\} \):
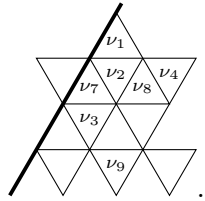
\begin{figure}[H]
  \centering
  \begin{tikzpicture}[scale=0.7]
    \draw (0,0)--(3,0)--(1.5,-3*sin{60})--cycle; \draw
    (0.5,-sin{60})--(1.5,sin{60})--(2.5,-sin{60})--cycle; \draw
    (1,0)--(2.5,-3*sin{60})--(3,-2*sin{60})--(0,-2*sin{60})--(0.5,-3*sin{60})--(2,0);
    \node[font=\tiny] at (1.5,0.3){\( \nu_{1} \)};  \node[font=\tiny] at
    (1.5,-0.3){\( \nu_{2} \)}; \node[font=\tiny] at (2.5,-0.3){\( \nu_{4} \)}; \node[font=\tiny] at
    (1,-sin{60}+0.3){\( \nu_{7} \)}; \node[font=\tiny] at (2,-sin{60}+0.3){\( \nu_{8} \)}; \node[font=\tiny] at
    (1,-sin{60}-0.3){\( \nu_{3} \)};  \node[font=\tiny] at
    (1.5,-2*sin{60}-0.3){\( \nu_{9} \)};
    \draw [line width=0.5mm] (1.6,1.2*sin{60})--(-0.5,-3*sin{60});
  \end{tikzpicture}.
\end{figure}

\fbox{Si \( n^{1}\geq 1 \) et \( \mu \notin\widehat{\Gr}\) ,} c'est-à-dire \( 1-ap^{d-1}\leq S\leq -1 \), alors
on a \( H^{2}(\mu)=0 \). De plus, on a
\(   H^{0}(\nu_{6}^{1})=H^{0}(m^{1}-1,-n^{1})=0 \), 
\( H^{2}(\nu_{3}^{1})=H^{2}(m^{1}-1,-n^{1}-1)=0 \)
et
\( H^{2}(\nu_{5}^{1})=H^{2}(m^{1},-n^{1}-2)=0 \).
Donc d'après \eqref{eq:eef5a3485d232fdb} et
\eqref{eq:a1aac2f2b9fe0e5e}, \( \mathcal{H}_{\alpha}^{1}(\nu_{4}) \)
est juste une extension de \( \mathcal{H}^{1}(\nu_{4}) \) par \(
\mathcal{H}^{1}(\nu_{3}) \), et \( \mathcal{H}_{\beta}^{1}(\nu_{6}) \)
est juste une extension de \( \mathcal{H}^{1}(\nu_{6}) \) par \(
\mathcal{H}^{1}(\nu_{4}) \). Donc dans ce cas, \( H^{2}(\mu)=0 \) et
\( H^{1}(\mu) \) admet une filtration dont les quotients sont
les \(\mathcal{H}^{1}(\nu_{i}) \) pour \( i\in\{1,2,\cdots,9\} \).

\fbox{Si \( n^{1}\geq 1 \) et \( \mu \in\widehat{\Gr}\),} c'est-à-dire \(
S\geq 0 \), alors \( H^{0}(\nu_{6}^{1})=H^{0}(m^{1}-1,-n^{1})=0
\). Donc \eqref{eq:a1aac2f2b9fe0e5e} devient:
\begin{equation}
  \label{eq:28734564018fcb40}
   0\to\mathcal{H}^{1}(\nu_{5})\to\mathcal{H}_{\beta}^{1}(\nu_{6})\to
    \mathcal{H}^{1}(\nu_{6})\xrightarrow{\partial_{\beta}^{1}}
   \mathcal{H}^{2}(\nu_{5})\to\mathcal{H}_{\beta}^{2}(\nu_{6})\to
    \mathcal{H}^{2}(\nu_{6})\to   0.
\end{equation}

\underline{Si \( S\geq 0 \) et  \(R=p^{d-1}-1  \),}  alors
\[
H^{2}(\nu_{5}^{1})=H^{2}(m^{1},-n^{1}-2)=H^{2}((a+1)p^{d-1},-ap^{d-1}-S-2)=0.
\]
En particulier, on a \( \partial_{\beta}^{1} =0\) dans \eqref{eq:28734564018fcb40}, donc pour \(
i\in\{1,2\} \), le \( G \)-module \( \mathcal{H}_{\beta}^{i}(\nu_{6})
\) est juste une extension de \( \mathcal{H}^{i}(\nu_{6}) \) par \(
\mathcal{H}^{i}(\nu_{5}) \).

D'autre part, on a
\[
  H^{2}(E_{\alpha}(\nu_{4}^{1}))=H^{2}(E_{\alpha}(m^{1}+1,-n^{1}-2))=H^{2}(E_{\alpha}(ap^{d-1},-ap^{d-1}-S-2))=0 \]
d'après la \autoref{prop:Eparticulier}. Donc \(
\mathcal{H}_{\alpha}^{2}(\nu_{4})=0 \) et,  d'après
\eqref{eq:eef5a3485d232fdb}, on a une suite exacte
\begin{equation}
  \label{eq:b6600c1a438f8a65}
  \begin{tikzcd}
   0\ar[r]&\mathcal{H}^{1}(\nu_{3})\ar[r]&\mathcal{H}_{\alpha}^{1}(\nu_{4})\ar[r,"f"]&
    \mathcal{H}^{1}(\nu_{4})\ar[r,"\partial_{\alpha}"]
    &\mathcal{H}^{2}(\nu_{3})\ar[r]&0.  
  \end{tikzcd}
\end{equation}

Notons \( \mathcal{Q}_{4}\subset \mathcal{H}^{1}(\nu_{4}) \) l'image
de \( f \), alors on a
\begin{displaymath}
  \begin{tikzcd}
    0\ar[r]&\mathcal{H}^{1}(\nu_{3})\ar[r]&\mathcal{H}_{\alpha}^{1}(\nu_{4})\ar[r,"f"]&\mathcal{Q}_{4}\ar[r]&0
  \end{tikzcd}
\end{displaymath}
et
\begin{displaymath}
  \begin{tikzcd}
    0\ar[r]&\mathcal{Q}_{4} \ar[r]&\mathcal{H}^{1}(\nu_{4})\ar[r,"\partial_{\alpha}"]
    &\mathcal{H}^{2}(\nu_{3})\ar[r]&0.
  \end{tikzcd}
\end{displaymath}

Donc dans ce cas, le facteur \( \mathcal{H}^{2}(\nu_{3}) \) est \og
effacé\fg{} dans la filtration de \( H^{1}(\mu) \) et de \( H^{2}(\mu)
\). Plus précisément, \( H^{2}(\mu) \) admet une filtration dont les
quotients sont
\( \{\mathcal{H}^{2}(\nu_{i})|i=1,2,4,5,6,7,8,9\}  \)
et \( H^{1}(\mu) \) admet une filtration dont les quotients sont
 \(  \{\mathcal{H}^{1}(\nu_{i})|i=1,2,3,5,6,7,8,9\}\cup\{\mathcal{Q}_{4}\} \)
où \( \mathcal{Q}_{4}\subset \mathcal{H}^{1}(\nu_{4}) \) est tel que
\( \mathcal{H}^{1}(\nu_{4})/\mathcal{Q}_{4}\cong \mathcal{H}^{2}(\nu_{3}) \).

\smallskip
\underline{De même, si \( S=0 \) et \( 0\leq R\leq p^{d-1}-1 \),}
alors le facteur \( \mathcal{H}^{2}(\nu_{5}) \) est \og effacé\fg{}
dans la filtration de \( H^{1}(\mu) \) et \( H^{2}(\mu)
\). Plus précisément, \( H^{2}(\mu) \) admet une filtration dont les
quotients sont
\(  \{\mathcal{H}^{2}(\nu_{i})|i=1,2,3,4,6,7,8,9\}  \)
et \( H^{1}(\mu) \) admet une filtration dont les quotients sont
  \( \{\mathcal{H}^{1}(\nu_{i})|i=1,2,3,4,5,7,8,9\}\cup\{\mathcal{Q}_{6}\} \)
où \( \mathcal{Q}_{6}\subset \mathcal{H}^{1}(\nu_{6}) \) est tel que
\( \mathcal{H}^{1}(\nu_{6})/\mathcal{Q}_{6}\cong \mathcal{H}^{2}(\nu_{5}) \).

\underline{Si \( 1\leq S\leq R\leq p^{d-1}-2 \),} alors
\[ \nu_{4}^{1}=(m^{1}+1,-n^{1}-2)=(ap^{d-1}+R+1,-ap^{d-1}-S-2) \]
avec \( 1\leq S<R+1\leq p^{d-1}-1 \). Donc \( \nu_{4}^{1} \) vérifie
l'hypothèse du
\autoref{cor:imagesum} pour \( \delta=\alpha \). D'autre part,
\[
\nu_{6}^{1}=(m^{1}-1,-n^{1})=(ap^{d-1}+R-1,-ap^{d-1}-(S-2)-2)
\]
avec
\( -1\leq S-2<R-1\leq p^{d}-3 \), donc \( \nu_{6}^{1} \) vérifie
l'hypothèse du
\autoref{cor:imagesum} pour \( \delta=\beta \).

Donc pour \( i\in\{1,2\} \), \( H^{i}(\mu) \) admet une filtration
dont les quotients sont
\begin{displaymath}
\{ \mathcal{H}^{i}(\nu_{1}),\mathcal{H}^{i}(\nu_{2}),\mathcal{H}_{\alpha}^{i}(\nu_{4}),
  \mathcal{H}_{\beta}^{i}(\nu_{6}),
  \mathcal{H}^{i}(\nu_{7}),\mathcal{H}^{i}(\nu_{8}),\mathcal{H}^{i}(\nu_{9})\}.
\end{displaymath}

De plus, on a des suites exactes longues:
\[    0\to\mathcal{H}^{1}(\nu_{3})\to\mathcal{H}_{\alpha}^{1}(\nu_{4})\to
    \mathcal{H}^{1}(\nu_{4})\xrightarrow{\partial_{\alpha}}
    \mathcal{H}^{2}(\nu_{3})\to\mathcal{H}_{\alpha}^{2}(\nu_{4})\to
    \mathcal{H}^{2}(\nu_{4})\to
    0.
\]
et
\[   0\to\mathcal{H}^{1}(\nu_{5})\to\mathcal{H}_{\beta}^{1}(\nu_{6})\to
    \mathcal{H}^{1}(\nu_{6})\xrightarrow{\partial_{\beta}}
  \mathcal{H}^{2}(\nu_{5})\to\mathcal{H}_{\beta}^{2}(\nu_{6})\to
    \mathcal{H}^{2}(\nu_{6})\to  0
\]
où  \( \image(\partial_{\alpha})\cong L(\nu_{4}^{0})\otimes
  I_{\alpha}(\nu_{4}^{1})^{(1)}  \)
et
\( \image(\partial_{\beta})\cong L(\nu_{6}^{0})\otimes
  I_{\beta}(\nu_{6}^{1})^{(1)} \),
qui peuvent être calculés récursivement par le \autoref{cor:imagesum}.

\subsubsection{Type \( \nabla \)}
Si \( \mu \) est de type \( \nabla \), c'est-à-dire \( r<s \), alors
on a forcément \( m^{1}>n^{1} \) et \( R>S \) car \( m\geq n \). Les neuf facteurs simples de \( \widehat{Z}(\mu) \)
sont donnés par la figure suivante (où \( \nu_{1}=\mu \)):
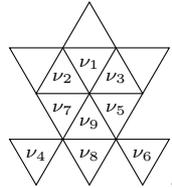
\begin{figure}[H]
  \centering
 \begin{tikzpicture}[scale=0.7]
    \draw (0,0)--(3,0)--(1.5,-3*sin{60})--cycle; \draw
    (0.5,-sin{60})--(1.5,sin{60})--(2.5,-sin{60})--cycle; \draw
    (1,0)--(2.5,-3*sin{60})--(3,-2*sin{60})--(0,-2*sin{60})--(0.5,-3*sin{60})--(2,0);
    \node[font=\tiny] at (1.5,-0.3){\( \nu_{1} \)}; \node[font=\tiny] at (1,-sin{60}+0.3){\( \nu_{2} \)}; \node[font=\tiny] at
    (2,-sin{60}+0.3){\( \nu_{3} \)}; \node[font=\tiny] at (1,-sin{60}-0.3){\( \nu_{7} \)}; \node[font=\tiny] at
    (2,-sin{60}-0.3){\( \nu_{5} \)}; \node[font=\tiny] at (1.5,-2*sin{60}+0.3){\( \nu_{9} \)}; \node[font=\tiny] at
    (1.5,-2*sin{60}-0.3){\( \nu_{8} \)}; \node[font=\tiny] at (0.5,-2*sin{60}-0.3){\( \nu_{4} \)}; \node[font=\tiny]
    at (2.5,-2*sin{60}-0.3){\( \nu_{6} \)};
  \end{tikzpicture}. 
\end{figure}

D'après
le \autoref{thm:pHifiltration}, on sait
que pour \( i\in\{1,2\} \), il existe une filtration de  \( H^{i}(\mu) \)
dont les quotients sont les suivants (l'ordre peut être différent) :
\begin{displaymath}
  \mathcal{H}^{i}(\nu_{1}),\mathcal{H}^{i}(\nu_{2}),\mathcal{H}^{i}(\nu_{3}),
  \mathcal{H}_{\alpha}^{i}(\nu_{5}), \mathcal{H}_{\beta}^{i}(\nu_{7}),\mathcal{H}^{i}(\nu_{8}),\mathcal{H}^{i}(\nu_{9}).
\end{displaymath}

On a  \( H^{0}(\nu_{5}^{1})=H^{0}(m^{1},-n^{1}-2)=0 \)  et
\( H^{3}(\nu_{4}^{1})=H^{3}(m^{1}-2,-n^{1}-1)=0 \) car \( m^{1}\geq
n^{1}+1\geq 1 \). Donc \( \mathcal{H}^{0}(\nu_{5})=0 \) et \(
\mathcal{H}^{3}(\nu_{4})=0 \), d'où une suite exacte
\begin{equation}
  \label{eq:7868a10a035b4ef8}
    0\to\mathcal{H}^{1}(\nu_{4})\to\mathcal{H}_{\alpha}^{1}(\nu_{5})\to
    \mathcal{H}^{1}(\nu_{5})\xrightarrow{\partial_{\alpha}}\mathcal{H}^{2}(\nu_{4})
    \to\mathcal{H}_{\alpha}^{2}(\nu_{5})\to
    \mathcal{H}^{2}(\nu_{5})\to
    0.
\end{equation}

De même, on a  \( H^{0}(\nu_{7}^{1})=H^{0}(m^{1}-1,-n^{1}-1)=0 \)  et
\( H^{3}(\nu_{6}^{1})=H^{3}(m^{1},-n^{1}-3)=0 \),
d'où une suite exacte
\begin{equation}
  \label{eq:2d989fff77716ffa}
    0\to\mathcal{H}^{1}(\nu_{6})\to\mathcal{H}_{\beta}^{1}(\nu_{7})\to
    \mathcal{H}^{1}(\nu_{7})\xrightarrow{\partial_{\beta}}
    \mathcal{H}^{2}(\nu_{6})\to\mathcal{H}_{\beta}^{2}(\nu_{7})\to
    \mathcal{H}^{2}(\nu_{7})\to
    0. 
\end{equation}

\fbox{Si \( S\leq -2 \) et \( R\geq 1 \),} alors
\[
H^{2}(\nu_{4}^{1})=H^{2}(m^{1}-2,-n^{1}-1)=H^{2}(ap^{d-1}+R-2,-ap^{d-1}-(S-1)-2)=0
\]
et
\[
H^{2}(\nu_{6}^{1})=H^{2}(m^{1},-n^{1}-3)=H^{2}(ap^{d-1}+R,-ap^{d-1}-(S+1)-2)=0.
\]
En particulier, on a \( \partial_{\alpha}=\partial_{\beta}=0 \). Donc
dans ce cas, \( H^{2}(\mu)=0 \) et \( H^{1}(\mu) \) admet une
filtration dont les quotients sont
\( \{\mathcal{H}^{1}(\nu_{i})|i=1,2,\cdots,9\} \).

\fbox{Si \( S\leq -2 \) et \( R=0 \),} alors on a encore
\[
H^{2}(\nu_{6}^{1})=H^{2}(m^{1},-n^{1}-3)=H^{2}(ap^{d-1}+R,-ap^{d-1}-(S+1)-2)=0,
\]
d'où \( \partial_{\beta}=0 \). D'autre part, on a
\[
H^{2}(E_{\alpha}(\nu_{5}^{1}))=H^{2}(E_{\alpha}(m^{1},-n^{1}-2))=H^{2}(E_{\alpha}(ap^{d-1},-ap^{d-1}-S-2))=0
\]
d'après la \autoref{prop:Eparticulier}, donc
\eqref{eq:7868a10a035b4ef8} devient
\begin{displaymath}
    0\to \mathcal{H}^{1}(\nu_{4})\to \mathcal{H}_{\alpha}^{1}(\nu_{5})\xrightarrow{f_{\alpha}}
    \mathcal{H}^{1}(\nu_{5})\xrightarrow{\partial_{\alpha}}
    \mathcal{H}^{2}(\nu_{4})\to 0.
\end{displaymath}

Dans ce cas, le facteur \( \mathcal{H}^{2}(\nu_{4}) \) est \og
effacé\fg{} dans la filtration de \( H^{1}(\mu) \) et de \( H^{2}(\mu)
\). Plus précisément, notons \( \mathcal{Q}_{5} \) l'image de \( f_{\alpha} \),
alors \( H^{1}(\mu) \) admet une filtration dont les quotients sont
\( \{\mathcal{H}^{1}(\nu_{i})|i=1,2,3,4,6,7,8,9\}\cup \{\mathcal{Q}_{5}\} \)
où \( \mathcal{Q}_{5}\subset \mathcal{H}^{1}(\nu_{5}) \) est tel que
\( \mathcal{H}^{1}(\nu_{5})/\mathcal{Q}_{5}\cong \mathcal{H}^{2}(\nu_{4}) \).
De plus, \(H^{2}(\mu)=0  \) même si \( \mathcal{H}^{2}(\nu_{4})\neq 0 \).

\fbox{Si \( S=-1 \) et \( R\geq 1 \), } alors
\[
H^{2}(\nu_{4}^{1})=H^{2}(m^{1}-2,-n^{1}-1)=H^{2}(ap^{d-1}+R-2,-ap^{d-1}-(S-1)-2)=0,
\]
d'où \( \partial_{\alpha}=0 \).
D'autre part, on a
\[
H^{2}(E_{\beta}(\nu_{7}^{1}))=H^{2}(E_{\beta}(m^{1}-1,-n^{1}-1))=H^{2}(E_{\beta}(ap^{d-1}+R-1,-ap^{d-1}))=0
\]
d'après la \autoref{prop:Eparticulier}. Donc
\eqref{eq:2d989fff77716ffa} devient
\begin{displaymath}
    0\to\mathcal{H}^{1}(\nu_{6})\to\mathcal{H}_{\beta}^{1}(\nu_{7})\xrightarrow{f_{\beta}}
    \mathcal{H}^{1}(\nu_{7})\xrightarrow{\partial_{\beta}}
    \mathcal{H}^{2}(\nu_{6})\to
    0.
\end{displaymath}

Dans ce cas, le facteur \( \mathcal{H}^{2}(\nu_{6}) \) est \og
effacé\fg{} dans la filtration de \( H^{1}(\mu) \) et de \( H^{2}(\mu)
\). Plus précisément, notons \( \mathcal{Q}_{7} \) l'image de \( f
_{\beta}\), alors \( H^{1}(\mu) \) admet une filtration dont les quotients sont
\( \{\mathcal{H}^{1}(\nu_{i})|i=1,2,3,4,5,6,8,9\}\cup \{\mathcal{Q}_{7}\} \)
où \( \mathcal{Q}_{7}\subset \mathcal{H}^{1}(\nu_{7}) \) est tel que
\( \mathcal{H}^{1}(\nu_{7})/\mathcal{Q}_{7}\cong \mathcal{H}^{2}(\nu_{6}) \).
De plus, \(H^{2}(\mu)=0  \) même si \( \mathcal{H}^{2}(\nu_{6}) \)
n'est pas forcément nul.

\smallskip
\fbox{De même, si \( S=-1 \) et \( R=0 \),} alors le facteur \(
\mathcal{H}^{2}(\nu_{4}) \) et le facteur \( \mathcal{H}^{2}(\nu_{6})
\) sont tous les deux \og effacés\fg{}. C'est-à-dire, \( H^{1}(\mu) \) admet une filtration dont les quotients sont
\( \{\mathcal{H}^{1}(\nu_{i})|i=1,2,3,4,6,8,9\}\cup \{\mathcal{Q}_{5},\mathcal{Q}_{7}\} \)
où
\( \mathcal{Q}_{5}\subset \mathcal{H}^{1}(\nu_{5}) \) et \(
\mathcal{Q}_{7}\subset \mathcal{H}^{1}(\nu_{7}) \) sont tels que
\( \mathcal{H}^{1}(\nu_{5})/\mathcal{Q}_{5}\cong \mathcal{H}^{2}(\nu_{4}) \)
et
\( \mathcal{H}^{1}(\nu_{7})/\mathcal{Q}_{7}\cong \mathcal{H}^{2}(\nu_{6}) \).
De plus, \(H^{2}(\mu)=0  \) même si \( \mathcal{H}^{2}(\nu_{4}) \) et
\(  \mathcal{H}^{2}(\nu_{6})\) ne sont
 pas  nuls.

\fbox{Si \( S\geq 0 \),} alors on a \( 0\leq S<R<p^{d-1}-1 \). Dans ce
cas,  \[ \nu_{5}^{1}
  =(m^{1},-n^{1}-2)=(ap^{d-1}+R,-ap^{d-1}-S-2)\] vérifie l'hypothèse
du \autoref{cor:imagesum}
pour \( \delta=\alpha \).

D'autre part,
\[\nu_{7}^{1}=(m^{1}-1,-n^{1}-1)=(ap^{d-1}+R-1,-ap^{d-1}-(S-1)-2)\]
avec \( -1\leq S-1<R-1\leq p^{d-1}-2 \). Donc \( \nu_{7}^{1} \)
vérifie l'hypothèse du \autoref{cor:imagesum} pour \(
\delta=\alpha \). Donc  pour \( i\in\{1,2\} \), \( H^{i}(\mu) \) admet une filtration
dont les quotients sont
\[
\{ \mathcal{H}^{i}(\nu_{1}),\mathcal{H}^{i}(\nu_{2}),\mathcal{H}^{i}(\nu_{3}),
  \mathcal{H}_{\alpha}^{i}(\nu_{5}), \mathcal{H}_{\beta}^{i}(\nu_{7}),\mathcal{H}^{i}(\nu_{8}),\mathcal{H}^{i}(\nu_{9})\}.
\]
De plus, on a des suites exactes longues:
\[
    0\to\mathcal{H}^{1}(\nu_{4})\to\mathcal{H}_{\alpha}^{1}(\nu_{5})\to
    \mathcal{H}^{1}(\nu_{5})\xrightarrow{\partial_{\alpha}}
    \mathcal{H}^{2}(\nu_{4})\to\mathcal{H}_{\alpha}^{2}(\nu_{5})\to
    \mathcal{H}^{2}(\nu_{5})\to
    0
\]
et
\[
   0\to\mathcal{H}^{1}(\nu_{6})\to\mathcal{H}_{\beta}^{1}(\nu_{7})\to
    \mathcal{H}^{1}(\nu_{7})\xrightarrow{\partial_{\beta}}
   \mathcal{H}^{2}(\nu_{6})\to\mathcal{H}_{\beta}^{2}(\nu_{7})\to
    \mathcal{H}^{2}(\nu_{7})\to  0
\]
où  \( \image(\partial_{\alpha})\cong L(\nu_{5}^{0})\otimes
  I_{\alpha}(\nu_{5}^{1})^{(1)}  \)
et
\( \image(\partial_{\beta})\cong L(\nu_{7}^{0})\otimes
  I_{\beta}(\nu_{7}^{1})^{(1)} \),
qui peut être calculés récursivement par le \autoref{cor:imagesum}.

\subsubsection{Cas \texorpdfstring{\( \alpha \)-singulier}{alpha-singulier}
  }
Supposons que \( \mu \) est \( \alpha \)-singulier, c'est-à-dire \(
0\leq s<r=p-1 \). Les quatre facteurs simples de \( \widehat{Z}(\mu) \)
sont donnés par la figure suivante (où \( \nu_{1}=\mu \)):
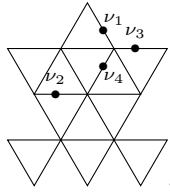
\begin{figure}[H]
  \centering
  \begin{tikzpicture}[scale=0.7]
    \draw (0,0)--(3,0)--(1.5,-3*sin{60})--cycle; \draw
    (0.5,-sin{60})--(1.5,sin{60})--(2.5,-sin{60})--cycle; \draw
    (1,0)--(2.5,-3*sin{60})--(3,-2*sin{60})--(0,-2*sin{60})--(0.5,-3*sin{60})--(2,0);
    \node[font=\tiny] at (1.8,0.4*sin{60}){\( \bullet \)}; \node[font=\tiny] at
    (2,0.4*sin{60}+0.2){\( \nu_{1} \)}; \node[font=\tiny] at (2.4,0){\( \bullet \)}; \node[font=\tiny] at
    (2.4,0.3){\( \nu_{3} \)}; \node[font=\tiny] at (1.8,-0.4*sin{60}){\( \bullet \)}; \node[font=\tiny] at
    (2,-0.4*sin{60}-0.2){\( \nu_{4} \)}; \node[font=\tiny] at (0.9,-sin{60}){\( \bullet \)};
    \node[font=\tiny] at (0.9,-sin{60}+0.3){\( \nu_{2} \)};
  \end{tikzpicture}.
\end{figure}

D'après
le \autoref{thm:pHifiltration}, on sait
que pour \( i\in\{1,2\} \), il existe une filtration de  \( H^{i}(\mu) \)
dont les quotients sont 
 \(  \mathcal{H}^{i}(\nu_{1}) \), \( \mathcal{H}_{\alpha}^{i}(\nu_{3})
 \), et
   \( \mathcal{H}^{i}(\nu_{4}) \).

On sait que \( H^{0}(\nu_{3}^{1})=H^{0}(m^{1}+1,-n^{1}-2)=0 \) et \(
H^{3}(\nu_{2}^{1})=H^{3}(m^{1}-1,-n^{1}-1)=0 \) car \( m^{1},n^{1}\geq
0\), donc \( \mathcal{H}^{0}(\nu_{3})=\mathcal{H}^{3}(\nu_{2})=0 \).
Donc il existe une suite exacte longue

\begin{equation}
    0\to\mathcal{H}^{1}(\nu_{2})\to\mathcal{H}_{\alpha}^{1}(\nu_{3})\to
    \mathcal{H}^{1}(\nu_{3})\xrightarrow{\partial_{\alpha}}
    \mathcal{H}^{2}(\nu_{2})\to\mathcal{H}_{\alpha}^{2}(\nu_{3})\to
    \mathcal{H}^{2}(\nu_{3})\to
    0.\label{eq:e548c476b8511c8e}
\end{equation}

\fbox{Si \( \mu \notin\widehat{\Gr}\) ,} c'est-à-dire \( S\leq -1 \), alors
on a \( H^{2}(\mu)=0 \). De plus, on a
\[
H^{2}(\nu_{2}^{1})=H^{2}(m^{1}-1,-n^{1}-1)=H^{2}(ap^{d-1}+R-1,-ap^{d-1}-S-1)=0.
\]

Donc d'après \eqref{eq:e548c476b8511c8e} , \( \mathcal{H}_{\alpha}^{1}(\nu_{3}) \)
est juste une extension de \( \mathcal{H}^{1}(\nu_{3}) \) par \(
\mathcal{H}^{1}(\nu_{2}) \). Donc dans ce cas, \( H^{2}(\mu)=0 \) et
\( H^{1}(\mu) \) admet une filtration dont les quotients sont
\( \{\mathcal{H}^{1}(\nu_{i})|i=1,2,3,4\} \).

\fbox{Si \( \mu\in\widehat{\Gr} \) et \( R=p^{d-1}-1 \),} c'est-à-dire \( S\geq 0 \)  et  \( R=p^{d-1}-1
  \),
 alors
 on a
\[
  H^{2}(E_{\alpha}(\nu_{3}^{1}))=H^{2}(E_{\alpha}(m^{1}+1,-n^{1}-2))=H^{2}(E_{\alpha}(ap^{d-1},-ap^{d-1}-S-2))=0 \]
d'après la \autoref{prop:Eparticulier}. Donc \(
\mathcal{H}_{\alpha}^{2}(\nu_{3})=0 \) et  d'après
\eqref{eq:e548c476b8511c8e}, on a une suite exacte
\begin{displaymath}
  \begin{tikzcd}
   0\ar[r]&\mathcal{H}^{1}(\nu_{2})\ar[r]&\mathcal{H}_{\alpha}^{1}(\nu_{3})\ar[r,"f"]&
    \mathcal{H}^{1}(\nu_{3})\ar[r,"\partial_{\alpha}"]
    &\mathcal{H}^{2}(\nu_{2})\ar[r]&0.  
  \end{tikzcd}
\end{displaymath}

Donc dans ce cas, le facteur \( \mathcal{H}^{2}(\nu_{2}) \) est \og
effacé\fg{} dans la filtration de \( H^{1}(\mu) \) et \( H^{2}(\mu)
\). Plus précisément, notons \( \mathcal{Q}_{3}\subset
\mathcal{H}^{1}(\nu_{3}) \) l'image de \( f \), alors \( H^{2}(\mu) \) admet une filtration dont les
quotients sont
\( \{\mathcal{H}^{2}(\nu_{i})|i=1,3,4\}  \)
et \( H^{1}(\mu) \) admet une filtration dont les quotients sont
 \(  \{\mathcal{H}^{1}(\nu_{i})|i=1,2,4\}\cup\{\mathcal{Q}_{3}\} \)
où \( \mathcal{Q}_{3}\subset \mathcal{H}^{1}(\nu_{3}) \) est tel que
\( \mathcal{H}^{1}(\nu_{3})/\mathcal{Q}_{3}\cong \mathcal{H}^{2}(\nu_{2}) \).

\fbox{Si \( 0\leq S\leq R\leq p^{d-1}-2 \),} alors
\[ \nu_{3}^{1}=(m^{1}+1,-n^{1}-2)=(ap^{d-1}+R+1,-ap^{d-1}-S-2) \]
avec \( 1\leq S<R+1\leq p^{d-1}-1 \). Donc \( \nu_{3}^{1} \) vérifie
l'hypothèse du
\autoref{cor:imagesum}  pour \( \delta=\alpha \).

Donc pour \( i\in\{1,2\} \), \( H^{i}(\mu) \) admet une filtration
dont les quotients sont
\( \{ \mathcal{H}^{i}(\nu_{1}),\allowbreak \mathcal{H}_{\alpha}^{i}(\nu_{3}),\allowbreak
   \mathcal{H}^{i}(\nu_{4})\} \).
De plus, on a une suite exacte longue:
\[
    0\to\mathcal{H}^{1}(\nu_{2})\to\mathcal{H}_{\alpha}^{1}(\nu_{3})\to
    \mathcal{H}^{1}(\nu_{3})\xrightarrow{\partial_{\alpha}}
    \mathcal{H}^{2}(\nu_{2})\to\mathcal{H}_{\alpha}^{2}(\nu_{3})\to
    \mathcal{H}^{2}(\nu_{3})\to
    0
\]

où \( \image(\partial_{\alpha})\cong L(\nu_{3}^{0})\otimes
  I_{\alpha}(\nu_{3}^{1})^{(1)}  \),
qui peut être calculé récursivement par le \autoref{cor:imagesum}.

\subsubsection{Cas \texorpdfstring{\( \beta \)-singulier}{beta-singulier}}
Si \( \mu \) est \(\beta  \)-singulier, c'est-à-dire \( 0\leq r<s=p-1 \), alors
on a forcément \( m^{1}>n^{1} \) et \( R>S \) car \( m\geq n \). Les quatre facteurs simples de \( \widehat{Z}(\mu) \)
sont donnés par la figure suivante (où \( \nu_{1}=\mu \)):
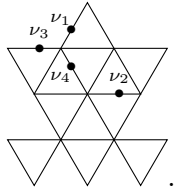
\begin{figure}[H]
  \centering
 \begin{tikzpicture}[scale=0.7]
    \draw (0,0)--(3,0)--(1.5,-3*sin{60})--cycle; \draw
    (0.5,-sin{60})--(1.5,sin{60})--(2.5,-sin{60})--cycle; \draw
    (1,0)--(2.5,-3*sin{60})--(3,-2*sin{60})--(0,-2*sin{60})--(0.5,-3*sin{60})--(2,0);
    \node[font=\tiny] at (1.2,0.4*sin{60}){\( \bullet \)}; \node[font=\tiny] at
    (1,0.4*sin{60}+0.2){\( \nu_{1} \)}; \node[font=\tiny] at (0.6,0){\( \bullet \)}; \node[font=\tiny] at
    (0.6,0.3){\( \nu_{3} \)}; \node[font=\tiny] at (1.2,-0.4*sin{60}){\( \bullet \)}; \node[font=\tiny] at
    (1,-0.4*sin{60}-0.2){\( \nu_{4} \)}; \node[font=\tiny] at (2.1,-sin{60}){\( \bullet \)};
    \node[font=\tiny] at (2.1,-sin{60}+0.3){\( \nu_{2} \)};
  \end{tikzpicture}.
\end{figure}

D'après
le \autoref{thm:pHifiltration}, on sait
que pour \( i\in\{1,2\} \), il existe une filtration de  \( H^{i}(\mu) \)
dont les quotients sont 
   \( \mathcal{H}^{i}(\nu_{1}) \), \( \mathcal{H}_{\beta}^{i}(\nu_{3})
   \), et
   \( \mathcal{H}^{i}(\nu_{4}) \).

On  a \(  H^{0}(\nu_{3}^{1})=H^{0}(m^{1}-1,-n^{1}-1)=0 \)  et
\( H^{3}(\nu_{2}^{1})=H^{3}(m^{1},-n^{1}-3)=0 \),
d'où une suite exacte

\begin{equation}
    0\to\mathcal{H}^{1}(\nu_{2})\to\mathcal{H}_{\beta}^{1}(\nu_{3})\to
    \mathcal{H}^{1}(\nu_{3})\xrightarrow{\partial_{\beta}}
    \mathcal{H}^{2}(\nu_{2})\to\mathcal{H}_{\beta}^{2}(\nu_{3})\to
    \mathcal{H}^{2}(\nu_{3})\to
    0.\label{eq:ed30cc3b92235750}
\end{equation}

\fbox{Si \( S\leq -2 \),} alors

\[
H^{2}(\nu_{2}^{1})=H^{2}(m^{1},-n^{1}-3)=H^{2}(ap^{d-1}+R,-ap^{d-1}-(S+1)-2)=0.
\]
En particulier, on a \(\partial_{\beta}=0 \). Donc
dans ce cas, \( H^{2}(\mu)=0 \) et \( H^{1}(\mu) \) admet une
filtration dont les quotients sont
\( \{\mathcal{H}^{1}(\nu_{i})|i=1,2,3,9\} \).

\fbox{Si \( S=-1 \), } alors
 on a
\[
H^{2}(E_{\beta}(\nu_{3}^{1}))=H^{2}(E_{\beta}(m^{1}-1,-n^{1}-1))=H^{2}(E_{\beta}(ap^{d-1}+R-1,-ap^{d-1}))=0
\]
d'après la \autoref{prop:Eparticulier}. Donc
\eqref{eq:ed30cc3b92235750} devient
\begin{displaymath}
  \begin{tikzcd}
    0\ar[r]&\mathcal{H}^{1}(\nu_{2})\ar[r]&\mathcal{H}_{\beta}^{1}(\nu_{3})\ar[r,"f_{\beta}"]&
    \mathcal{H}^{1}(\nu_{3})\ar[r, "\partial_{\beta}"]{dll}
    &\mathcal{H}^{2}(\nu_{2})\ar[r]&
    0.
  \end{tikzcd} 
\end{displaymath}

Dans ce cas, le facteur \( \mathcal{H}^{2}(\nu_{2}) \) est \og
effacé\fg{} dans la filtration de \( H^{1}(\mu) \) et \( H^{2}(\mu)
\). Plus précisément, notons \( \mathcal{Q}_{3} \) l'image de \( f
_{\beta}\), alors \( H^{1}(\mu) \) admet une filtration dont les quotients sont
\( \{\mathcal{H}^{1}(\nu_{i})|i=1,2,4\}\cup \{\mathcal{Q}_{3}\} \)
où \( \mathcal{Q}_{3}\subset \mathcal{H}^{1}(\nu_{3}) \) est tel que
\( \mathcal{H}^{1}(\nu_{3})/\mathcal{Q}_{3}\cong \mathcal{H}^{2}(\nu_{2}) \).
De plus, \(H^{2}(\mu)  \) admet une filtration dont les quotients sont
\( \{\mathcal{H}^{2}(\nu_{i})|i=1,3,4\} \).

\fbox{Si \( S\geq 0 \),} alors on a \( 0\leq S<R<p^{d-1}-1 \). Dans ce
cas, on a 
\[\nu_{3}^{1}=(m^{1}-1,-n^{1}-1)=(ap^{d-1}+R-1,-ap^{d-1}-(S-1)-2)\]
avec \( -1\leq S-1<R-1\leq p^{d-1}-2 \). Donc \( \nu_{3}^{1} \)
vérifie l'hypothèse du \autoref{cor:imagesum} pour \(
\delta=\beta \). Donc  pour \( i\in\{1,2\} \), \( H^{i}(\mu) \) admet une filtration
dont les quotients sont
\( \{ \mathcal{H}^{i}(\nu_{1}),\mathcal{H}_{\beta}^{i}(\nu_{3}),
   \mathcal{H}^{i}(\nu_{4})\} \).
De plus, on a une suite exacte longue:
\[
   0\to\mathcal{H}^{1}(\nu_{2})\to\mathcal{H}_{\beta}^{1}(\nu_{3})\to
    \mathcal{H}^{1}(\nu_{3})\xrightarrow{\partial_{\beta}}
   \mathcal{H}^{2}(\nu_{3})\to\mathcal{H}_{\beta}^{2}(\nu_{3})\to
    \mathcal{H}^{2}(\nu_{3})\to   0
\]
où 
\( \image(\partial_{\beta})\cong L(\nu_{3}^{0})\otimes
  I_{\beta}(\nu_{3}^{1})^{(1)} \),
qui peuvent être calculés récursivement par le \autoref{cor:imagesum}.

\subsubsection{Cas \texorpdfstring{\( \gamma \)-singulier}{gamma-singulier} ou
  \texorpdfstring{\( \alpha \)-\( \beta \)-singulier}{alpha-beta-singulier}}

Si \( \mu \) est \( \gamma \)-singulier ou \( \alpha \)-\( \beta
\)-singulier, alors il n'y a pas de \( E_{\alpha} \) ou \( E_{\beta} \)
dans la filtration. Donc d'après le \autoref{thm:pHifiltration}, si \(
\mu \) est \( \gamma \)-singulier, alors pour \( j\in\{1,2\} \), \(
H^{i}(\mu )\) admet une filtration dont les quotients sont
\( \{\mathcal{H}^{j}(\nu_{i})| i=1,2,3,4\} \),
où la valeur de \( \nu_{i} \) est donnée par la figure suivante (où \( \nu_{1}=\mu \)):
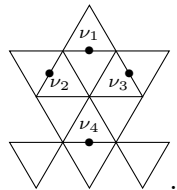
\begin{figure}[H]
  \centering
  \begin{tikzpicture}[scale=0.7]
    \draw (0,0)--(3,0)--(1.5,-3*sin{60})--cycle; \draw
    (0.5,-sin{60})--(1.5,sin{60})--(2.5,-sin{60})--cycle; \draw
    (1,0)--(2.5,-3*sin{60})--(3,-2*sin{60})--(0,-2*sin{60})--(0.5,-3*sin{60})--(2,0);
    \node[font=\tiny] at (1.5,0){\( \bullet \)}; \node[font=\tiny] at (1.5,0.3){\( \nu_{1} \)}; \node[font=\tiny] at
    (2.25,-0.5*sin{60}){\( \bullet \)}; \node[font=\tiny] at
    (2.05,-0.5*sin{60}-0.2){\( \nu_{3} \)}; \node[font=\tiny] at
    (1.5,-2*sin{60}){\( \bullet \)}; \node[font=\tiny] at (1.5,-2*sin{60}+0.3){\( \nu_{4} \)};
    \node[font=\tiny] at (0.75,-0.5*sin{60}){\( \bullet \)}; \node[font=\tiny] at
    (0.95,-0.5*sin{60}-0.2){\( \nu_{2} \)};
  \end{tikzpicture}.
\end{figure}

Si \( \mu \) est \( \alpha \)-\( \beta \)-singulier, alors \(
\mu=(m^{1}p+p-1,-n^{1}p-p-1) \) et pour \( i\in\{1,2\} \), on a 
\[H^{i}(\mu)\cong L(p-1,p-1)\otimes  H^{i}(m^{1},-n^{1}-2)^{(1)}.\]

\def\refname{References}


\begin{thebibliography}{BNPS19}

\bibitem[AH19]{AH19} Pramod Achar and William Hardesty, Calculations
  with graded perverse-coherent sheaves, Q.J. Math. {\bf 70} (2019),
  no.4, 1327-1352.

\bibitem[And79]{And79} Henning Haahr Andersen, The first cohomology group of a line bundle on \(G/B\), 
Invent. Math. {\bf 51} (1979), 287-296. 



\bibitem[And86a]{And86a} Henning Haahr Andersen, Torsion in the
  cohomology of line bundles on homogeneous spaces for Chevalley
  groups, Proc. Amer. Math. Soc. {\bf 96} (1986), no. 4, 537–544.
\bibitem[And86b]{And86b} Henning Haahr Andersen, On the generic structure of  cohomology modules for 
semi-simple algebraic groups, Trans. Amer. Math. Soc.  {\bf 295}
(1986), 397-415.

\bibitem[BNPS19]{BNPS19} Christopher P. Bendel, Daniel K. Nakano,
  Cornelius Pillen, Paul Sobaje, Counterexamples to the Tilting and
  (p,r)-Filtration Conjectures, arXiv:1901.06687 [math.RT].



\bibitem[Don02]{Don02}  Stephen Donkin, A note on the characters of
  the cohomology of induced vector bundles on \( G/B \) in
  characteristic \( p \), Special issue in celebration of Claudio
  Procesi's 60th birthday, J. Algebra {\bf 258} (2002), no. 1, 255–274.

\bibitem[Don06]{Don06} Stephen Donkin, The cohomology of lines bundles on the three-dimensional flag 
variety, J. Algebra {\bf 307} (2006), 570-613.



\bibitem[DS88]{DS88} Stephen R. Doty and John B. Sullivan, On the structure of the higher cohomology modules of line bundles on \(G/B\), J. Algebra {\bf 114} (1988), 286-332. 

\bibitem[Gri80]{Gri80} Walter Lawrence Griffith, Cohomology of flag varieties in characteristic \(p\), Illinois J. Math. {\bf 24} (1980), no. 3, 
452-461. 

\bibitem[Har16]{Har16} William Hardesty, Support varieties of line
  bundle cohomology groups for SL3(k),  
J. Algebra {\bf 448} (2016), 127-173.

\bibitem[Hum86]{Hu86} James E. Humphreys, Cohomology of \(G/B\) in characteristic \(p\), Adv. in Math. 
  {\bf 59} (1986), 170-183.

\bibitem[Irv86]{Irv86} Ronald S. Irving,  The structure of certain
  highest weight modules for \( \SL_{3} \), J. Algebra {\bf 99} (1986), no. 2, 438–457.
 
\bibitem[Jan77]{Jan77} Jens Carsten Jantzen, Darstellungen
  halbeinfacher Gruppen und kontravariante Formen. J. Reine
  Angew. Math {\bf 290} (1977), 117–141.
\bibitem[Jan80]{Jan80} Jens Carsten Jantzen, Darstellungen halbeinfacher Gruppen und ihrer Fro\-be\-nius-Kerne, 
J. reine angew. Math {\bf 317} (1980), 157-199. 

\bibitem[Jan03]{Jan03}  Jens Carsten Jantzen, Representations of algebraic groups. Second edition. Mathematical Surveys and Monographs, 107. American Mathematical Society, Providence, RI, 2003. xiv+576 pp. ISBN: 0-8218-3527-0 

\bibitem[KH85]{KH84} Kerstin  Kühne-Hausmann,
Zur Untermodulstruktur der Weylmoduln fur SL3, Bonner math. Schr {\bf
  162} (1985).
  
\bibitem[Lin90]{Lin90} Zongzhu Lin, Structure of  cohomology of lines bundles on \(G/B\) for semi simple groups, 
J. Algebra {\bf 134} (1990), 225-256. 

\bibitem[Lin91]{Lin91} Zongzhu Lin, Socle series of  cohomology groups of lines bundles on \(G/B\), 
J. Pure Applied Algebra {\bf 72} (1991), 275-294. 


\bibitem[Ye82]{Ye82} Samy El Badawy Yehia, Extensions of simple modules for the universal Chevalley groups and its parabolic subgroups, Ph. D. thesis, University of Warwick 1982.



\end{thebibliography}
\end{document}